\documentclass[reqno]{amsart}
\usepackage[utf8]{inputenc}

\usepackage{fullpage}
\usepackage{amsfonts, amsmath, amssymb, amsthm,mathtools,tikz-cd}
\usepackage{etoolbox}
\usepackage{dynkin-diagrams}
\newcommand{\Z}{\mathbb{Z}}
\newcommand{\kf}{\mathbb{K}}
\newcommand{\C}{\mathbb{C}}

\newcommand{\N}{\mathbb{N}}

\newcommand{\U}{\mathcal{U}^+}
\newcommand{\He}{\mathcal{H}}

\usetikzlibrary{braids}
\usetikzlibrary{knots}
\theoremstyle{definition}
\newtheorem{thm}{Theorem}[section]
\newtheorem{cor}[thm]{Corollary}
\newtheorem{lem}[thm]{Lemma}

\newtheorem{prop}[thm]{Proposition}
\newtheorem{rem}[thm]{Remark}

\newtheorem{ex}[thm]{Example}

\newtheorem{defn}[thm]{Definition}
\numberwithin{equation}{section}
\usepackage{graphicx}
\usepackage{caption}

\usepackage{hyperref}

\hypersetup{
    colorlinks=false,
    pdfborder={0 0 0},
    pdfborderstyle={/S/U/W 0},
}

\begin{document}

\title{2-categorical affine symmetries of quantum enveloping algebras}
\author{Sam Qunell}

\begin{abstract}
We produce 2-representations of the positive part of affine quantum enveloping algebras on their finite-dimensional counterparts in type $A_n$. These 2-representations naturally extend the right-multiplication 2-representation of $U_q^+(\mathfrak{sl}_{n+1})$ on itself and are closely related to evaluation morphisms of quantum groups. We expect that our 2-representation exists in all simple types and show that the corresponding 1-representation exists in types $D_4$ and $C_2$. We also show that a certain quotient of our 1-representation in type $A_n$ is isomorphic to a prefundamental representation. We use this to provide a new proof of the prefundamental representation character formulas in these cases.
\end{abstract}
\maketitle

\tableofcontents


\section{Introduction}\label{Introduction}

We produce a 2-representation of $U_q^+(\hat{\mathfrak{sl}}_{n+1})$ on $U_q^+(\mathfrak{sl}_{n+1})$. These representations extend the natural right-multiplication 2-representation of $U_q^+(\mathfrak{sl}_{n+1})$ on itself and are given by formulas based on the quantum evaluation homomorphisms $U_q(\hat{\mathfrak{sl}}_{n+1})\twoheadrightarrow U_q(\mathfrak{gl}_{n+1})$. We expect that our 2-representation  can be defined outside of type $A_n$, which was a limitation of the quantum evaluation homomorphisms. To construct these 2-representations, we use the monoidal categorifications of quantum groups via KLR algebras introduced in \cite{khla} and \cite{2km}.

We prove several other interesting results along the way. Specifically, we construct surjective algebra homomorphisms including $U_q^+(\hat{\mathfrak{sl}}_2)\twoheadrightarrow U_q^+(\mathfrak{sl}_2\times \mathfrak{sl}_2)$. We also show that a certain quotient of $U_q^+(\mathfrak{sl}_{n+1})$ is isomorphic to one of the dual prefundamental representations of \cite{herjim}. We use this construction to provide a new proof of the character formulas described in \cite{charform} in these cases.

\subsection{Affine quantum groups}
Our initial motivation for producing this representation is the ``affinization" of the representation theory of $U_q(\mathfrak{g})$, i.e. exploring the relationship between the representation theory of $U_q(\mathfrak{g})$ and that of $U_q(\hat{\mathfrak{g}})$. The classification of generalized Cartan matrices shows that these algebras are closely related, and the loop space presentation of $\hat{\mathfrak{g}}$ makes a finer study of $U_q(\hat{\mathfrak{g}})$ possible. Chari and Pressley used this presentation to determine all of the finite-dimensional representations of $U_q(\hat{\mathfrak{g}})$ when $q$ is specialized to a transcendental complex number \cite{chapre}. For a given finite-dimensional representation of $U_q(\mathfrak{g})$, Chari introduced \cite{ch} a minimal affinization, which is roughly a representation of $U_q(\hat{\mathfrak{g}})$ that is as close to being an evaluation representation as possible. These minimal affinizations are now mostly classified.

Our 2-representation is related to evaluation homomorphisms of quantum enveloping algebras. The non-quantum evaluation homomorphism of enveloping algebras is straightforward. The affine Lie algebra $\hat{\mathfrak{g}}$ can be obtained from $\mathfrak{g}$ by extending scalars to $\C[z,z^{-1}]$, taking a central extension, and then adding a derivation. The evaluation homomorphism simply maps $z\rightarrow c$ for some fixed nonzero complex number $c$ and sends the central element and derivation to $0$. In terms of the Chevalley generators, the extra affine generator $E_0$ is mapped to a scalar multiple of the lowest weight root vector of $\mathfrak{g}$.

The quantum versions of these homomorphisms were introduced by Jimbo \cite{jimbo}. In the quantum evaluation morphism, the image of $E_0$ is $K * [F_n,[F_{n-1},\cdots[F_2,F_1]_q]_q]_q$, where $[A,B]_q=AB-qBA$ and $K\in U_q(\mathfrak{gl}_{n+1})$ is a semisimple element. Unfortunately, it is known that no such homomorphism splitting the obvious inclusion can exist outside of type $A_n$. The quantum evaluation homomorphism in type $A_2$ was categorified by Mackaay, Macpherson, and Vaz by taking mapping cones of certain morphisms in the KLR categorification of the full quantum group \cite{cat_eval}. This construction uses similar formulas to ours, although these two constructions do not appear to be directly related.

The representation we produce is not restricted to type $A_n$. In fact, it factors through a representation that is known to work in all symmetrizable Kac-Moody types. We use a renormalized $q$-boson algebra $B(\mathfrak{g})$ associated to a simple Lie algebra $\mathfrak{g}$. These are generated by $e_i$ and $f_i$ satisfying quantum Serre relations and $f_ie_j-q_i^{-C_{ij}}e_jf_i=\delta_{ij}/(1-q_i^2)$, where $C_{ij}$ is the associated Cartan matrix. The $q$-boson algebras were introduced by Kashiwara in \cite{kashi}, where they were also shown to act on $U_q^+(\mathfrak{g})$ as follows. The $e_i$ act by right multiplication by $E_i$, and the $f_i$ act by $E_i^*$, the adjoint to this right multiplication map under Lusztig's bilinear form $(*,*)_L$ satisfying $(E_i,E_j)_L=\delta_{ij}/(1-q_i^2)$. The relationship between the $q$-boson algebras and $U_q(\hat{\mathfrak{g}})$ has been investigated recently, for example in \cite{recentboson} and \cite{newkashiboson}.

In this paper, we focus on the case $\mathfrak{g}=\mathfrak{sl}_{n+1}$ due to the simpler structure of its positive root poset, although we show that our 1-representation works at least in types $D_4$ and $C_2$. We are primarily interested in the resulting action of $U_q^+(\hat{\mathfrak{g}})$ on $U_q^+(\mathfrak{g})$. It is known that the representation of $B(\mathfrak{g})$ on $U_q^+(\mathfrak{g})$ is faithful, see for example \cite{newkashiboson}. So, the existence of our representations of algebras can also be interpreted as the existence of algebra homomorphisms $U_q^+(\hat{\mathfrak{g}})\rightarrow B(\mathfrak{g})$. After this article was submitted, we constructed a categorification of $B(\mathfrak{g})$ in our paper \cite{meboson}. We expect that the results of the present paper can be interpreted as the existence of a monoidal functor between the homotopy categories of the catgorifications of $U_q^+(\hat{\mathfrak{g}})$ and $B(\mathfrak{g})$, although we do not prove this. The image of $E_0$ under such a functor may be related to the projective resolutions in \cite{bkm}.

For our desired action of $U_q^+(\hat{\mathfrak{sl}}_{n+1})$, we replace the $F_i$ in the above presentation of $E_0$ with the adjoint operators $E_i^*$. We will see that no extra semisimple element $K$ is needed for our action. As a result of the 2-representation we construct in this paper, we show that making $E_0$ act by $[E_n^*,[E_{n-1}^*,\dots [E_2^*,E_1^*]_q]_q]_q$ gives a well-defined representation of $U_q^+(\hat{\mathfrak{sl}}_{n+1})$. This representation extends to one of the affine Borel quantum group $U_q(\hat{\mathfrak{b}})$. We show also that a minor adjustment to the formula gives additional representations in type $A_n$ and also gives a representation in types $D_4$ and $C_2$.

The $\mathfrak{sl}_2$ case of our representation has also been studied by Hernandez in \cite{shifted}. Here, $U_q^+(\mathfrak{g})$ is viewed as an analogue of a Verma module for a $q$-oscillator algebra $U_{qos}(\mathfrak{sl}_2)$, which is an extension of $B(\mathfrak{sl}_2)$ by semisimple elements.  Bazhanov, Lukyanov, and Zamolodchikov observed in \cite{blz} that there is an evaluation homomorphism $U_q(\hat{\mathfrak{b}})\rightarrow U_{qos}(\mathfrak{sl}_2)$. It is easy to check that the resulting representation on $U_q^+(\mathfrak{sl}_2)$ cannot be extended to the full affine algebra $U_q(\hat{\mathfrak{sl}}_2)$, but it can be extended to a shifted quantum affine algebra. These shifted algebras were first defined by Finkelberg and Tsymbaliuk in \cite{ft} and studied further by Hernandez in \cite{shifted}. They are obtained by modifying Drinfeld's loop presentation for $U_q(\hat{\mathfrak{g}})$. In the $\mathfrak{sl}_2$ case, one can directly produce an evaluation homomorphism from the shifted affine algebra to a localized $q$-oscillator algebra analogous to the quantum evaluation homomorphism of $U_q(\hat{\mathfrak{sl}}_2)$. Similar morphisms were introduced in \cite{laxmat} for $\mathfrak{g}=\mathfrak{gl}_{n+1}$ using an RTT-type presentation.

While an earlier draft of this article was in preparation, we were informed that an essentially dual representation of $U_q(\hat{\mathfrak{b}})$ on $U_q^-(\mathfrak{sl}_{n+1})$ was constructed in \cite{unipotent}, and later generalized to other types in \cite{unipotent2}. The main results of these papers are new explicit constructions of prefundamental representations. These prefundamental representations were introduced in \cite{herjim} and are important in the representation theory of $U_q(\hat{\mathfrak{b}})$. For example, in an appropriate version of category $\mathcal{O}$ for $U_q(\hat{\mathfrak{b}})$, every simple module arises as a subquotient of tensor powers of the prefundamental and 1-dimensional representations. We will show that a quotient of $U_q^+(\mathfrak{sl}_{n+1})$ is isomorphic to a prefundamental representation, compared to the unipotent coordinate ring submodule of \cite{unipotent}. However, due to a few simplifying assumptions made in our case, we are able to prove slightly more. In \cite{unipotent}, the character formulas of \cite{charform} are needed to make the identification with the prefundamental representation. Instead, we prove the simplicity of the quotient directly, thereby giving a new proof of the character formulas. Similar constructions of the prefundamental representations as quotients of ``Verma-like" modules appear in \cite{neguţcato}, although these constructions are less explicit. These prefundamental representations were originally defined as limits of Kirillov-Reshetikhin modules, with the action of $U_q(\hat{\mathfrak{b}})$ given in terms of Drinfeld's loop generators. So, the results of \cite{unipotent} and this article are somewhat surprising in that the action of the Chevalley generators can be given explicitly for many prefundamental representations. It would be interesting to compare our construction more precisely to those of \cite{unipotent} and \cite{neguţcato}.

It seems likely that our results can be generalized to other types. Our most significant technical result, Theorem \ref{thm:newproj}, which produces an explicit basis for a key module, does not assume any particular type. Moreover, we use various technical tools from \cite{bkm} which are not limited to type $A_n$. Related representations of $U_q^+(\hat{\mathfrak{g}})$ in more general types appear in \cite{unipotent} and \cite{unipotent2}. Generalizing our constructions to all simple types is work in progress.

It is an open problem to find purely algebraic categorifications of the loop generators in an affine quantum group. Because of this, it is interesting to see that evaluation-like homomorphisms have natural categorifications via the Kac-Moody presentation.
\subsection{Techniques}
In order to construct our 2-representation, we use a few novel techniques. Firstly, we need to modify the category being acted upon in order to define the relevant functors. Typically, one studies categories of finitely-generated projective or finite-dimensional representations of KLR algebras. However, we will be using restriction to categorify the action of $E^*$, and restriction is not defined if we only consider finitely-generated modules. So, we expand our category to the category of left-bounded locally finite-dimensional graded projective modules over KLR algebras. 
 
 Another of our novel techniques is in how we categorify the deformed Lie brackets $[E_i^*,E_j^*]_q=E_i^*E_j^*-qE_j^*E_i^*$. Additive categories are generally not suitable for categorifying differences. Subtraction is often categorified by taking an appropriate mapping cone within the triangulated homotopy category, see for example \cite{laurent}. However, we find that this is not necessary for our representation, and so we can work with only additive categories. We categorify $[E_i^*,E_j^*]_q$ by taking the cokernel of a degree $1$ natural transformation from the functor corresponding to $E_j^*E_i^*$ to that of $E_i^*E_j^*$. These natural transformations can be viewed as multiplication by certain elements in the KLR algebras. The resulting functor $E_0$ can be viewed as tensoring via a specific bimodule whose properties are carefully investigated in this paper. Utilizing cokernels instead of mapping cones introduces several technical difficulties. Firstly, it needs to be shown that the bimodules we produce remain projective as left modules, and therefore tensoring gives a well-defined functor on our categories. Secondly, it needs to be shown that the functor we produce has the correct action on the Grothendieck group, i.e. that the maps we take cokernels of are really injective. An earlier version of this paper resolved both difficulties with an explicit basis theorem for the quotient bimodule, although we now use some techniques of \cite{bkm} following the advice of our anonymous reviewer. Our basis theorem is still used elsewhere in the paper, and we can now interpret it as providing a basis for a large class of KLR algebra modules including various standard modules of \cite{bkm}. The main tool used in our remaining calculations is the existence of least common multiples, or joins, in the symmetric group $S_k$. Since the KLR algebras are deformations of the group algebra for $S_k$, this tool allows us to control which deformed braid relations are applied during our computations.
 
 While it is generally difficult to explicitly produce representations of our monoidal categories, we find that the most naive choice of endofunctors and natural transformations satisfying the properties we need gives us a 2-representation of $U_q^+(\hat{\mathfrak{g}})$, which suggests that this representation is more than just happenstance.

\subsection{Organization}
This paper is organized as follows. In Section \ref{Background}, we review the basic theory of quantum groups and 2-representation theory as well as fix notational conventions. A few key motivating examples are given here.

In Section \ref{sec:general}, we construct the 2-representation in types $A_n$ for $n\geq 2$. Our $E_0$ functor is defined as tensoring by a bimodule that is constructed as a quotient of a certain projective module. The most difficult part of this section is producing an explicit basis for the corresponding quotient module, as this is needed for some of our more technical computations. This analysis is where working with other types becomes difficult, although our work strongly suggests that the results should hold in all types. We conclude this section with additional investigation of the decategorified representation.

In Section \ref{sec:sl2}, we construct our 2-representation in type $A_1$, i.e. of $U_q^+(\hat{\mathfrak{sl}}_2)$ on $U_q^+(\mathfrak{sl}_2)$. Originally, this 2-representation was discovered by Rouquier. Moreover, the corresponding 1-representation is the homomorphism to $U_{qos}(\mathfrak{sl}_2)$ of \cite{blz} composed with the action on $U_q^+(\mathfrak{sl}_2)$ of \cite{kashi}. The proofs and techniques in this section are, however, new. We treat this case separately because there are certain special adjustments that need to be made only in the $\mathfrak{sl}_2$ case. As a consequence of some of the results in this section, we deduce the existence of quotient maps of algebras including $U_q^+(\hat{\mathfrak{sl}}_2)\twoheadrightarrow U_q^+(\mathfrak{sl}_2\times \mathfrak{sl}_2)$. 

In Section \ref{sec:d4}, we discuss how to potentially generalize our construction in order to make it work outside of type $A_n$ and to give additional representations within type $A_n$. We produce a representation of $U_q^+(\hat{\mathfrak{so}}_8)$ on $U_q^+(\mathfrak{so}_8)$ and of $U_q^+(\hat{\mathfrak{sp}}_4)$ on $U_q^+(\mathfrak{sp}_4)$ generalizing our previous constructions. We also briefly discuss possible categorifications.

Finally, in Section \ref{sec:prefund}, we prove that a certain quotient of $U_q^+(\mathfrak{sl}_{n+1})$ is isomorphic to a dual prefundamental representation. We use this to provide new proofs of the character formulas in \cite{charform}, and we give the actions of the Kac-Moody generators of $U_q(\hat{\mathfrak{b}})$ explicitly.

\subsection*{Acknowledgments}
We thank Rapha\"el Rouquier for his guidance and feedback. We also thank David Hernandez for his interest and for introducing us to several of the references, including \cite{unipotent}. We thank the anonymous reviewer who had several useful suggestions, including the suggestion that the results of \cite{bkm} and \cite{affinizationklr} might be useful for simplifying our construction of the $E_0$ functor and $T_{00}$ natural transformation. We also thank Gurbir Dhillon for helpful conversations.

\section{Background}\label{Background}
\subsection{Quantum groups}
We review the relevant facts about quantum groups and Kac-Moody Lie algebras and fix notation. Standard references for this material are \cite{lusbook} and \cite{kacinf}.

\begin{defn}\label{quantum}
Let $\mathfrak{g}=\mathfrak{g}(C)$ be a complex Kac-Moody Lie algebra associated to the indecomposable generalized Cartan matrix $(C_{ij})_{i,j\in I}$, where $I$ is an indexing set. In this paper, we implicitly assume $C$ is symmetric for notational convenience. Denote by $\N$ the set of nonnegative integers. Let $q$ be an indeterminant, and for $0\leq k \leq n\in \N$, denote by $[n]_q:=(q^n-q^{-n})/(q-q^{-1})$, $[n]_q!:=[n]_q*[n-1]_q*\dots[2]_q*[1]_q$, and $\binom{n}{k}_q:=([n]_q!)/([k]_q![n-k]_q!)$. 
  We define the \emph{quantum universal enveloping algebra} of $\mathfrak{g}$ to be the associative algebra $U_q(\mathfrak{g})$ over $\C(q)$ generated by elements $\{E_i,F_i,K_i^{\pm 1}\}_{i\in I}$ with relations

\[K_iK_i^{-1}=K_i^{-1}K_i=1, K_iK_j=K_jK_i,\]
\[K_iE_j=q^{C_{ij}}E_jK_i, K_iF_j=q^{-C_{ij}}F_jK_i,\]
\[E_iF_j-F_jE_i=\delta_{ij}\frac{K_i-K_i^{-1}}{q-q^{-1}},\]
\[\sum_{k=0}^{1-C_{ij}}\binom{1-C_{ij}}{k}_q (-1)^kE_i^kE_jE_i^{1-C_{ij}-k} =0 \text{ for $i\neq j$},\]
\[\sum_{k=0}^{1-C_{ij}}\binom{1-C_{ij}}{k}_q (-1)^kF_i^kF_jF_i^{1-C_{ij}-k} =0 \text{ for $i\neq j$}.\]
\end{defn}

The last two of these relations are known as the \emph{quantum Serre relations}. Denote by $\N[I]$ the semigroup of all nonnegative formal linear combinations of $I$ elements. For $i\in I$, we denote by 
$\alpha_i$ its corresponding generator in $\N[I]$. The algebra $U_q(\mathfrak{g})$ has a natural $\Z[I]$-grading, with $gr(E_i)=-gr(F_i)=\alpha_i$, and $gr(K_i)=0$. This grading corresponds to a weight space decomposition of $U_q^+(\mathfrak{g})$, and all representations studied in this paper have a weight space decomposition.

\begin{defn}
The \emph{Dynkin diagram} of $C$ is the graph with vertices labelled by $I$, no simple edge loops, and $|C_{ij}|$ edges between any two distinct vertices $i$ and $j$. 
\end{defn}

In this article, we will only consider when the generalized Cartan matrix $C$ is finite type (invertible) or affine (corank 1). There is a canonical bijection between the isomorphism classes of generalized Cartan matrices in (untwisted) affine and finite types. In this bijection, the affine-type Dynkin diagram may be obtained from the corresponding finite-type Dynkin diagram by adding a unique extending vertex along with 1 or 2 edges adjacent to it. In case $C$ is finite type, we will henceforth denote its corresponding affine matrix as $\hat{C}$ and associated Kac-Moody algebra as $\hat{\mathfrak{g}}$. 
For $C$ finite type, we use $I=\{1,2,3,\dots n\}$, and for $\hat{C}$, we use $I=\{0,1,2,\dots n\}$, where 0 labels the unique extending vertex, and all other labellings identical to those in the Dynkin diagram for $C$. 

Denote by $U_q(\mathfrak{b})$, respectively $U_q^+(\mathfrak{g})$, the subalgebra of $U_q(\mathfrak{g})$ generated by the $E_i$ and the $K_i^{\pm 1}$, respectively just the $E_i$.

\begin{defn}
    On $U_q(\mathfrak{g})$ there is a $\C$-linear algebra involution $\bar{*}$ defined by $\bar{E_i}=E_i$, $\bar{F_i}=F_i$, $\bar{K_i}=K_i^{-1}$, and $\bar{q}=q^{-1}$. This involution can also be restricted either to $U_q(\mathfrak{b})$ or to $U_q^+(\mathfrak{g})$.
\end{defn}
\begin{defn}
    On $U_q^+(\mathfrak{g})$, there is a unique symmetric bilinear form $(*,*)_L$ defined by $(1,1)_L=1$, $(1,E_i)_L=0$, $(E_i,E_j)_L=\frac{\delta_{ij}}{1-q^2}$, and extended to the rest of $U_q^+(\mathfrak{g})$ via the Hopf pairing rule $(ab,c)_L=(a\otimes b,\Delta(c))_L$, for $\Delta$ the comultiplication of the standard twisted bialgebra structure. This is exactly the form defined on Lusztig's algebra $\mathbf{f}$, although our $q$ is his $v^{-1}$. This form is nondegenerate.
\end{defn} 
The intermediate algebra that facilitates our representation is the $q$-boson algebra analogue for $\mathfrak{g}$. We define a renormalized version.
\begin{defn}
    The \emph{$q$-boson} algebra associated to a  Cartan matrix $(C_{ij})_{i,j\in I}$ is the $\C(q)$-algebra $B(C)$ with generators $e_i,f_i$ for $i\in I$ satisfying quantum Serre relations and 
    \[f_ie_j-q^{-C_{ij}}e_jf_i=\frac{\delta_{ij}}{1-q^2}.\]
    We will also denote it as $B(\mathfrak{g})$ for $\mathfrak{g}=\mathfrak{g}(C)$.
\end{defn}
Kashiwara showed in \cite{kashi} that $B(\mathfrak{g})$ acts on $U_q^+(\mathfrak{g})$ via $e_i$ acting as right multiplication by $E_i$ and $f_i$ the corresponding adjoint $E_i^*$ under $(*,*)_L$, albeit with different choices of normalization.

\subsection{Symmetric groups}
We denote by $S_n$ the symmetric group on $n$ symbols, and for $\omega\in S_n$, we denote by $l(\omega)$ the length of any reduced presentation of this element.
\subsection{KLR algebras}
We review Khovanov-Lauda-Rouquier (KLR) algebras as well as certain categories of their representations. The generators and relations for these algebras will be exactly those for our 2-representation of $U_q^+(\hat{\mathfrak{g}})$. These algebras were originally defined by Khovanov, Lauda, and Rouquier in \cite{2km} and \cite{khla}, and all results of this subsection are from there. For a more introductory survey, see \cite{brundan}.

\begin{defn}\label{klr}
    
Let $\mathbb{K}$ be an algebraically closed field of characteristic 0. Let $\vec{Q}$ be a quiver, i.e., a finite directed graph without simple edge loops. Denote by $I$ the vertex set, and for $i,j\in I$, $m_{ij}$ the number of edges pointing from $i$ to $j$. We define $\N[I]$ as before. For $\alpha\in \N[I]$, denote by $I^\alpha$ the set of all words $i_1i_2\dots i_n$  for which all $i_k\in I$ and $ \alpha_{i_1}+\alpha_{i_2}+\dots \alpha_{i_n}=\alpha$. For $v\in I^\alpha$, denote by $v_i$ to be $i$'th letter of $v$. Denote also $|\alpha|:=|\sum_{i\in I} c_i*\alpha_i|=\sum_{i\in I} c_i$ where $c_i\in \N$. The symmetric group $S_{|\alpha|}$ has a transitive left action on $I^\alpha$ with the transposition $s_k$ acting on $v$ by swapping its $k$ and $k+1$ entries. For $i,j\in I$, define $Q_{ij}(u,v)\in \mathbb{K}[u,v]$ by \[Q_{ij}(u,v):=(1-\delta_{i,j})(v-u)^{m_{ij}}(u-v)^{m_{ji}}.\]

We define the \emph{KLR algebra} $H_{\alpha}(\vec{Q})$ to be the associative $\mathbb{K}$-algebra generated by elements $\{1_v\}_ {v\in I^\alpha}\cup \{x_i\}_{1\leq i\leq |\alpha|}\cup \{\tau_i\}_{1\leq i\leq |\alpha|-1}$ with relations
\begin{enumerate}
    \item $\text{The elements } 1_v \text{ are orthogonal idempotents whose sum is } 1$,

    \item $1_vx_i=x_i1_v,$
    
    \item $1_v\tau_i=\tau_i1_{s_i(v)}$

    \item $x_ix_j=x_jx_i$, 

    \item $(\tau_ix_j-x_{s_{i}(j)}\tau_i)1_v= \begin{cases}
    1_v \text{ if } j=i+1 \text{ and } v_i=v_{i+1}\\
    -1_v \text{ if } j=i \text{ and } v_i=v_{i+1}\\
    0 \text{ otherwise}
    \end{cases}$,

    \item $\tau_i\tau_j=\tau_j\tau_i$ if $|i-j|>1$,
    \item $\tau_i^21_v=Q_{v_i,v_{i+1}}(x_i,x_{i+1})1_v,$

    \item $(\tau_{i+1}\tau_i\tau_{i+1}-\tau_i\tau_{i+1}\tau_i)1_v=\delta_{v_i,v_{i+2}}\frac{Q_{v_{i},v_{i+1}}(x_{i+2},x_{i+1})-Q_{v_i,v_{i+1}}(x_{i},x_{i+1})}{x_{i+2}-x_i}1_v$.

\end{enumerate}

It can be shown that $Q_{v_i,v_{i+1}}(x_i,x_{i+1})-Q_{v_i,v_{i+1}}(x_{i+2},x_{i+1})$ is always divisible by $x_{i+2}-x_i$, and so the expression on the right-hand side of relation (8) is always a well-defined element of $H_\alpha(\vec{Q})$. The relations above will henceforth be called the \emph{KLR relations}. 

For $v\in I^\alpha$ and $w\in I^\beta$, we may write either $vv'$ or $v,v'$ for the concatenation $vv'\in I^{\alpha+\beta}$ given by $(vv')_i=v_i$ if $i\leq |\alpha|$ and $(vv')_i=v'_{i-|\alpha|}$ otherwise. If we have $w\in I^\beta$ for some $\beta \in \N[I]$ with $|\beta|<|\alpha|$, then we denote 
\[1_{*w}:=\sum_{\substack{v\in  I^{\alpha}\\ 
    v=w'w,\\
    w'\in I^{\alpha-\beta}}} 1_v\in H_\alpha(\vec{Q})
                    .\]
                    
We give $H_\alpha(\vec{Q})$ the structure of a graded ring where $\text{deg}(1_v)=0$, $\text{deg}(x_i)=2$, $\text{deg}(\tau_i1_v)=-2$ if $v_i=v_{i+1}$, and $\text{deg}(\tau_i1_v)=-m_{v_i,v_{i+1}}-m_{v_{i+1},v_i}$ if $v_i\neq v_{i+1}$.

These algebras are also called \emph{quiver Hecke algebras}. When $\vec{Q}$ is clear, we will simply write $H_\alpha$ in place of $H_\alpha(\vec{Q})$.
\end{defn}

The following basis theorem aids in later computations. 
\begin{thm}\label{pbw}
    Write $n:=|\alpha|$, and for each $\omega\in S_n$, fix a reduced presentation $\omega=s_{i_1}\dots s_{i_k}$. Denote $\tau_\omega=\tau_{i_1}\dots \tau_{i_k}$. Then the following set forms a basis for $H_\alpha(\vec{Q})$ as a free $\kf$-module.
    \[\{x_1^{l_1}\dots x_n^{l_n}\tau_\omega 1_v \vert l_i\in \N, \omega\in S_n, v\in I^\alpha\}.\]
    \end{thm}
    
    We may refer to this theorem later as the PBW-basis theorem for KLR algebras and to the corresponding basis elements as PBW basis elements. This shows also that the $\{\tau_\omega1_v\vert \omega\in S_n, v\in I^\alpha\}$ form a basis for $H_\alpha(\vec{Q})$ as a $\kf[x_1,\dots x_n]$-algebra. Later, we will want to choose specific convenient reduced presentations for each $S_n$ element.
    
    If we have a fixed $\omega\in S_n$ with fixed reduced presentation $\omega=s_{i_1}\dots s_{i_k}$, then we will informally call $\tau_{i_1}\dots \tau_{i_k}\in H_\alpha(\vec{Q})$ the KLR algebra element \emph{associated} to this $\omega$ and presentation. We will also say that the $1_{\omega(v)}\tau_{i_1}\dots \tau_{i_k}1_v\in 1_{\omega(v)}H_\alpha(\vec{Q})1_v$ are associated for any $v\in I^\alpha$.

Here is one more property of the $\tau_i$ that we will use.
\begin{prop}\label{symcommute}
    The $\tau_i1_v$ commute with polynomials that are symmetric in $x_i$ and $x_{i+1}$.
\end{prop}

\subsection{Categories}
We describe the many types of categories that we will use. We also define some of the specific categories studied in this paper and review a key theorem from the literature.

All categories used in this paper will be $\kf$-linear. 

\begin{defn}
    A \emph{graded category}, or $\Z$-\emph{graded category}, is a category equipped with an autoequivalence $T$. 
\end{defn}
When our category is additive, this autoequivalence endows the Grothendieck group with the structure of a $\Z[q,q^{-1}]$ module, where $q[M]=[T(M)]$ and $q$ is an indeterminant. If, moreover, our category has a compatible monoidal structure, then the Grothendieck group will in fact have the structure of a $\Z[q,q^{-1}]$ algebra. In what follows, we will refer to morphisms $T^j(M)\rightarrow N$ as morphisms $M\rightarrow N$ of degree $j$. 
\begin{defn}
    Given an additive graded category $\mathcal{C}$, we can contract the grading to define a simpler additive category $\mathcal{C}-\text{ungr}$. The objects of $\mathcal{C}-\text{ungr}$ are the same as those of $\mathcal{C}$, and $\text{Hom}_{\mathcal{C}-\text{ungr}}(M,N)=\bigoplus_{j}\text{Hom}_{\mathcal{C}}(M,T^j(N))$, viewed only as an abelian group. In $\mathcal{C}-\text{ungr}$, we have that $T^j(M)\simeq M$ for any $j$.
\end{defn}

We now define one of the fundamental categories studied in this paper.
\begin{defn}\label{2cat}
    For $\vec{Q}$ a quiver with vertex set $I$, we define $\mathcal{U}^{'+}(\vec{Q})$ to be the strict monoidal additive $\kf$-linear category generated by objects $E_i$ and morphisms $X_i:E_i\rightarrow E_i$ and $T_{ij}:E_iE_j\rightarrow E_jE_i$, where $i\in I$. These morphisms satisfy the following relations. We denote the product $A\otimes B$ simply as $AB$, since we will eventually be viewing these objects as endofunctors of various categories. The identity morphism of $E_i$ will also be denoted as $E_i$. For $F$ an endomorphism of $M$ and $G$ and endomorphism of $N$, we denote $FG:=F\otimes G:MN\rightarrow MN$. We denote by $Q$ the same matrix associated to $\vec{Q}$ as in Definition \ref{klr}.
    \begin{enumerate}        
        \item $T_{ij}\circ X_iE_j-E_jX_i\circ T_{ij}=\delta_{ij}\text{id}_{E_iE_j},$
        \item 
        $T_{ij}\circ E_iX_j-X_jE_i\circ T_{ij}=-\delta_{ij}\text{id}_{E_iE_j}$,
        \item $T_{ij}\circ T_{ji}=Q_{ij}(E_jX_i,X_jE_i)$,
        \item $ T_{jk}E_i\circ E_jT_{ik}\circ T_{ij}E_k- E_kT_{ij}\circ T_{ik}E_j \circ E_iT_{jk}=\delta_{i,k}\frac{Q_{ij}(X_iE_j,E_iX_j)E_i-E_iQ_{ij}(E_jX_i,X_jE_i)}{X_iE_jE_i-E_iE_jX_i}$.
    \end{enumerate}
    The relations above will also be called the \emph{KLR relations}.
\end{defn}
One can show that there is an isomorphism of algebras $H_\alpha(\vec{Q})\xrightarrow{\sim} 
 End_{\mathcal{U}^{'+}}(\bigoplus_{v\in I^\alpha} E_{v_{|\alpha|}}\dots E_{v_1})$. Under this map, 
 \[1_v\rightarrow \text{id}_{E_{v_{|\alpha|}}\dots E_{v_1}},\]
 \[x_i1_v\rightarrow E_{v_{|\alpha|}}\dots E_{v_{i+1}}X_{v_i}E_{v_{i-1}}\dots E_{v_1},\]
 \[\tau_{i}1_v\rightarrow E_{v_|\alpha|}\dots E_{v_{i+2}}T_{v_{i+1},v_i}E_{v_{i-1}}\dots E_{v_1}.\]
 Note that the coordinates of $v$ are written left-to-right in $H_n(\vec{Q})$, but are written right-to-left on the right-hand side.

 We need to make a few adjustments to $\mathcal{U}^{'+}(\vec{Q})$ to ensure it has all of the properties we will need. Firstly, we would like this category to be idempotent closed. There are several idempotents in $H_\alpha(\vec{Q})$, but the associated morphisms in $\mathcal{U}^{'+}(\vec{Q})$ do not always have an image object.
 
 \begin{defn}
     We define $\mathcal{U}^+(\vec{Q}):=\mathcal{U}^{'+}(\vec{Q})^i$ is the idempotent (Karoubi) completion of $\mathcal{U}^{'+}(\vec{Q})$. The objects of this category are pairs $(M,e)$ where $M$ is an object of $\mathcal{U}'^+(\vec{Q})$ and $e$ is an idempotent endomorphism of $M$. The morphisms from $(M,e)$ to $(M',e')$ are morphisms $f:M\rightarrow M'$ such that $e'fe=f$. The pair $(M,e)$ should then be viewed as the ``image of $e$" as an object of $\mathcal{U}^{'+}(\vec{Q})^i$.
 \end{defn} Secondly, we would like this category to be $\Z$-graded. Since our morphisms are based on the KLR algebra, they come with a natural grading. We use that $\text{deg}(X_i)=\text{deg}(x_i)=2$ and $\text{deg}(T_{ij})=\text{deg}(\tau_11_{ij})=-C_{ij}$. We will formally add shifted versions of all objects. 
 
 \begin{defn}
     We define $\mathcal{U}_q^+(\vec{Q}):= \mathcal{U}^{'+}(\vec{Q})^i-\text{gr}$. The objects of this category are collections $\{M_j\}_{j\in \Z}$ where $M_j\in \mathcal{U}^{'+}(\vec{Q})^i$ and only finitely many of the $M_j$ are nonzero. Morphisms are given by \[\text{Hom}_{\mathcal{U}_q^+(\vec{Q})}(\{M_{j}\},\{N_j\})=\bigoplus_j \text{Hom}_{\mathcal{U}'^+(\vec{Q})^i}(M_j,N_j)_0,\] where we only allow degree zero maps on the right-hand side. The autoequivalence $T$ of $\mathcal{U}_q^+(\vec{Q})$ is the functor $T(\{M_j\})_k=M_{k-1}$ and obvious effect on morphisms. Note that $\U_q(\vec{Q})-\text{ungr}=\U(\vec{Q})$. 
 \end{defn}
 It is easier to study this category after showing it is equivalent to another category defined in terms of KLR algebras. First, some notation.

 \begin{defn}
     For $R=\bigoplus_{i\in \Z} R_i$ a graded ring with graded pieces $R_i$, denote by $R-\text{grmod}$ the category whose objects are graded finitely-generated (left) $R$-modules and morphisms are degree zero graded maps. This is a graded category with $T(\{M_i\})_j=M_{j-1}$. Similarly, we have a graded full subcategory $R-\text{grproj}$ the category of finitely generated (left) projective $R$-modules.  We will eventually consider also graded modules that are not necessarily finitely generated. The corresponding categories will be denoted $R-\text{grMod}$ and $R-\text{grProj}$ with otherwise the same notation. 
 \end{defn}
 In the cases we consider, $R_i=0$ for $i<<0$ and $\text{dim}_{\kf}(R_i)<\infty$, so if our module $\{M_i\}$ is finitely generated, we have $M_i=0$ for $i<<0$. Note that there is a faithful functor $(R-\text{grMod})-\text{ungr}\rightarrow R-\text{Mod}$ with $\{M_i\}\rightarrow \bigoplus_i M_i$. 
 \begin{defn}
      For $\alpha\in \N[I]$, let $\He^{fg}_\alpha:= H_\alpha-\text{grproj}$. Then define $\He^{fg}(\vec{Q})=\bigoplus_{\alpha\in \N[I]} \He^{fg}_\alpha$. This is clearly a graded additive category, and it can be given the structure of a monoidal category as follows. For any $\alpha, \beta\in \N[I]$, there is a canonical injective (non-unital) morphism of graded rings $H_{\alpha}\otimes_{\kf} H_{\beta}\hookrightarrow H_{\alpha+\beta}$ given by concatenation, i.e. $1_{v}\otimes_{\kf} 1_{v'}\rightarrow 1_{vv'}$. The corresponding induction functor can be converted into our desired bifunctor. We define the bifunctor $\otimes:H_{\alpha}-\text{grmod}\times H_{\beta}-\text{grmod}\rightarrow H_{\alpha+\beta}-\text{grmod}$ via $(M,N)\rightarrow H_{\alpha+\beta}\otimes_{\alpha,\beta}(M\otimes_{\kf} N)$, where the outermost tensor is over $H_{\alpha}\otimes_{\kf} H_{\beta}$-modules. One can show that products of projective modules remain projective, and therefore, these products can be restricted to $\He_\alpha^{fg}(\vec{Q})$. The unit object of the corresponding monoidal product on $\He^{fg}(\vec{Q})$ is $H_{0}(\vec{Q})= \kf\in \mathcal{H}_0^{fg}$. This product endows $K_0(\He^{fg}(\vec{Q}))$ with the structure of a $\Z[q,q^{-1}]$-algebra. 
 \end{defn}
 For any $v\in I^\alpha$, we have that $H_{\alpha}1_v$ is an object of $\He^{fg}_\alpha$. However, these are in general not indecomposable.
 
 There is an equivalence of additive monoidal graded categories $\mathcal{U}^+_q(\vec{Q})\rightarrow \He^{fg}(\vec{Q})$ sending $E_i$ to $H_{\alpha_i}(\vec{Q})\in \He^{fg}_{\alpha_i}$. Since $\mathcal{U}^+_q(\vec{Q})$ is defined in terms of generators and relations, there is a straightforward process to defining monoidal functors out of it. Since $\He^{fg}(\vec{Q})$ has very explicit objects, it is easier to define functors acting on it.

 The following result is the key reason why KLR algebras have become so prominent in categorical representation theory. 
 \begin{thm}[\cite{khla2}]\label{decat}
     Let $\vec{Q}$ be an orientation of the Dynkin diagram for generalized Cartan matrix $C$. There is an isomorphism of $\N[I]$-graded $\C(q)$-algebras $U_{q}^+(\mathfrak{g}(C))\simeq \C(q)\otimes_{\Z[q,q^{-1}]} K_0(\He^{fg}(\vec{Q}))$.
 \end{thm}
Under this isomorphism, $E_i\rightarrow [H_{\alpha_i}]$. We henceforth use the symbol $E_i$ to refer to both the element $E_i\in U_q^+(\mathfrak{g}(C))$ and the object $E_i\in \mathcal{U}_q^+(\vec{Q})$ when no confusion is possible. The self-dual indecomposable projectives of $\He^{fg}(\vec{Q})$ for the duality described below correspond to elements of Lusztig's canonical basis for $U_q^+(\mathfrak{g})$ \cite{qha,vv}. One can also deduce that by ignoring the grading, we obtain an isomorphism with the classical enveloping algebra, i.e. $U^+(\mathfrak{g})\simeq \C\otimes_\Z K_0(\mathcal{U}(\vec{Q}))$.

Some of the other classical structures on $U_q^+(\mathfrak{g})$ can be given categorical meaning.
\begin{defn}
    For $M=\{M_i\}_{i\in \Z}$ a $\Z$-graded $\kf$-vector space with each $\text{dim}_{\kf}(M_i)<\infty$, we define the \emph{graded dimension} of $M$, denoted $\text{grdim}(M)$, as $\sum_{i\in \Z}\text{dim}_\kf(M_i)q^i$. If we assume that all of our graded vector spaces satisfy $M_i=0$ for $i<<0$, then the graded dimension takes values in $\Z[[q]][q^{-1}]$, the ring of formal integral Laurent series in the variable $q$. Note that $\text{grdim}(T(M))=q\cdot\text{grdim}(M)$.
\end{defn}
\begin{defn}
     For $P,Q\in \He^{fg}_\alpha$, we denote by $\text{Hom}_{H_\alpha}^\bullet (P,Q)$ the graded $\kf$-vector whose degree $i$ component is the group of degree $i$ morphisms from $P$ to $Q$. This vector space gives us a duality on $\He^{fg}_\alpha$ with $P\rightarrow P^{\#}:=\text{Hom}_{H_\alpha}^\bullet (P,H_\alpha)$. Here, the left $H_\alpha$-module structure comes from twisting with a diagram-reversal  antiinvolution of $H_\alpha$.  Note that $H_\alpha$ and its summands are self-dual, but $T(H_\alpha)^\#=T^{-1}(H_\alpha)$. On $K_0$, this descends to a duality $v\rightarrow \bar{v}$ exchanging $q$ and $q^{-1}$.
\end{defn}
\begin{defn}\label{form}
    There is a $\Z[[q]][q^{-1}]$-valued, $\Z$-bilinear form $(*,*)$ on $K_0(\He^{fg}_\alpha)$ defined by \[([P],[Q])=\text{grdim}(\text{Hom}_{H_\alpha}^\bullet(P,Q)).\] It is $q$-semilinear in the sense that $([P],[T(Q)])=([T^{-1}(P)],[Q])=q\cdot([P],[Q])$. We can extend this form to all of $K_0(\He^{fg})$ by enforcing that if $P\in \He^{fg}_\alpha$ and $Q\in \He^{fg}_\beta$ for $\alpha\neq \beta$ then $([P],[Q])=0$.
\end{defn}

This is  the form $(*,*)'$ studied in \cite{khla}. On $U_q^+(\mathfrak{g})$, the symmetric form $(\overline{*},*)$ can be identified with $(*,*)_L$.  

In the remainder of this paper, we generally drop $\vec{Q}$ from our categories when no ambiguity is possible. 
\subsection{2-representations}
We describe the data of a 2-representation of $U_q^+(\mathfrak{g})$ and give a motivating example.

\begin{defn}\label{2rep}
    We define a \emph{2-representation} of $U_q^+(\mathfrak{g})$, or of $\mathcal{U}_q^+(\vec{Q})$ for $\vec{Q}$ an orientation of the Dynkin diagram for $\mathfrak{g}$, is a strict monoidal additive graded functor $\rho:\mathcal{U}_q^+(\vec{Q})\rightarrow End_{\oplus,\Z}(\mathcal{C})$, where $\mathcal{C}$ is some additive graded category. This is the data of an endofunctor of $\mathcal{C}$ for each object of $\mathcal{U}_q^+(\vec{Q})$ along with natural transformations $X_i\in End(\rho(E_i))$ and $T_{ij}\in End(\rho(E_iE_j))$ that satisfy the KLR algebra relations in Definition \ref{2cat}. There are also equalities $\rho(X\oplus Y)\simeq \rho(X)\oplus \rho(Y)$ and $\rho(XY)=\rho(X)\circ \rho(Y)$. For compatibility with the grading, we require $\rho(T(X))=T\circ \rho(X)$. The 2-representations we consider in this paper will also be \emph{weight} 2-representations, i.e., $\mathcal{C}=\bigoplus_{\alpha\in \N[I]} \mathcal{C}_\alpha$ for $\mathcal{C}_\alpha$ additive categories, and each $\rho(E_i)$ is a collection of functors $\rho(E_i)_\alpha:\mathcal{C}_\alpha\rightarrow \mathcal{C}_{\alpha+\alpha_i}$. We can similarly define a 2-representation of $U^+(\mathfrak{g})$ on an additive category $\mathcal{D}$ as a strict monoidal functor $\rho:\mathcal{U}^+(\vec{Q})\rightarrow End_\oplus(\mathcal{D})$. In what follows, we will generally omit the $\rho$ and identify $E_i\in \mathcal{U}_q^+(\vec{Q})$ with its corresponding endofunctor of $\mathcal{C}$. 
\end{defn}
Producing a 2-representation of $U_q^+(\mathfrak{g})$ on an additive graded category $\mathcal{C}$ is slightly more difficult than doing so for $U^+(\mathfrak{g})$ since the natural transformations $X_i\in End(\rho(E_i))$ and $T_{ij}\in End(\rho(E_iE_j))$ must now have the correct grading. Given a 2-representation of $U_q^+(\mathfrak{g})$, we obtain a representation of $\C(q)$ algebras of $U_q^+(\mathfrak{g})$ on $\C(q)\otimes K_0(\mathcal{C})$.

The following example is the basis for our later constructions.
\begin{ex}\label{leftmult}
    The right multiplication representation of $U_q^+(\mathfrak{g})$ is naturally categorified. Fix again an orientation $\vec{Q}$ for the Dynkin diagram for $\mathfrak{g}$. For $M$ a right $H_\alpha$-module and $N$ a left $H_\alpha$-module, denote by $M \otimes_\alpha N$ the corresponding $H_\alpha$-bilinear tensor product. Recall the monoidal structure on $\He^{fg}(\vec{Q})$. For any $N\in \He^{fg}_\alpha$, the monoidal product $N\otimes H_{\alpha_i}$ is therefore identified with $H_{\alpha+\alpha_i}1_{*i}\otimes_{\alpha} N$. Here, we are viewing $H_{\alpha+\alpha_1}1_{*i}$ as a unital $(H_{\alpha+\alpha_i},H_\alpha)$-bimodule. We therefore define $E_{i,\alpha}:\He^{fg}_\alpha\rightarrow \He^{fg}_{\alpha+\alpha_i}$ to be the functor $H_{\alpha+\alpha_i}1_{*i}\otimes_\alpha$, and define $E_i:\He^{fg}(\vec{Q})\rightarrow \He^{fg}(\vec{Q})$ to be the sum over all $E_{i,\alpha}$. We can also make the identification $E_iE_j=H_{\alpha+\alpha_i+\alpha_j}1_{*ji}\otimes_{\alpha}$. We have natural transformations 
    \[X_i:E_i\rightarrow E_i,\]  \[a\otimes_\alpha b\rightarrow ax_{|\alpha|+1}\otimes_\alpha b\] and 
    \[T_{ij}:E_iE_j\rightarrow E_jE_i,\]
    \[a\otimes_\alpha b\rightarrow a\tau_{|\alpha|+1}\otimes_\alpha b.\]
    
\end{ex}
\section{The \texorpdfstring{$\mathfrak{sl}_{n+1}$}{sl n+1} case}\label{sec:general}

\subsection{Topological \texorpdfstring{$K_0$}{K0} \& adjoint functors}\label{targets}
We introduce the categories that will be acted upon. They are made larger than those ordinarily studied so that our functors are well-defined. 

We fix $n \geq 2$. We orient the Dynkin diagram for $\mathfrak{sl}_{n+1}$ by choosing the right-facing orientation, i.e. there are only arrows of the form $i\rightarrow i+1$. To orient the Dynkin diagram for $\hat{\mathfrak{sl}}_{n+1}$, the additional edges will be oriented $0\rightarrow n$ and $0\rightarrow 1$. These orientations are displayed in Table \ref{orientpic2}. With these orientations fixed, we omit the orientation in all categories and algebras in this section. 
\begin{table}[ht]
\begin{center}
\begin{tikzcd}
	&& 0 \\
	\\
	1 & 2 & {\dots } & {n-1} & n
	\arrow[from=1-3, to=3-1]
	\arrow[from=1-3, to=3-5]
	\arrow[from=3-1, to=3-2]
	\arrow[from=3-2, to=3-3]
	\arrow[from=3-3, to=3-4]
	\arrow[from=3-4, to=3-5]
\end{tikzcd}
\end{center}
\caption{Dynkin diagram orientation in this section}
\label{orientpic2}
\end{table}

 We want to define a 2-representation of $U_q^+(\hat{\mathfrak{sl}}_{n+1})$ on $\He^{fg}(\mathfrak{sl}_{n+1})$ extending the ``right multiplication" 2-representation of Example \ref{leftmult}. As a first step, we want to categorify the adjoints of right multiplication by $E_i$. Since $E_i$ can be viewed as an induction functor $H_{\alpha+\alpha_i} 1_{*i}\otimes_{\alpha}$, we may try to define $E_i^*$ as the corresponding restriction functor $1_{*i}H_{\alpha+\alpha_i}\otimes_{\alpha+\alpha_i}$. While $1_{*i}H_{\alpha+\alpha_i}$ is indeed projective as a left $H_\alpha$ module, it is not finitely generated over $H_\alpha$, so this functor is unfortunately not well-defined on $\He^{fg}$. There is an easy fix for this, however. 

\begin{defn}
 We enlarge each $\He^{fg}_\alpha(\mathfrak{sl}_{n+1})$ by defining $\mathcal{H}_\alpha$ as the full subcategory of $H_\alpha(\mathfrak{sl}_{n+1})-\text{grProj}$ containing objects $M$ for which $M_i=0$ for $i << 0$ (left-bounded) and $\text{dim}_\kf(M_i)<\infty$ for all $i$ (locally finite-dimensional). We then also enlarge $\He^{fg}$ by defining $\mathcal{H}:=\bigoplus_{\alpha\in \N[I]} \mathcal{H}_\alpha$.
\end{defn}

Extra caution much be taken in studying such a category since certain infinite direct sum decompositions are now possible, and this will collapse the ungraded split Grothendieck group. The correct notion of Grothendieck group to use is the \emph{topological} split Grothendieck group, which we also denote by $K_0$. Naisse and Vaz developed the theory of these groups in \cite{nava}. The topological split Grothendieck group of $\mathcal{H}$ carries the structure of a topological $\Z[[q]][q^{-1}]$ module, which essentially says that sums like $(1+q+q^2+q^3+\dots )[M]=[N]$ make sense. The following theorem shows that $\He$ has not lost any Lie-theoretic information. Let $\C((q))$ be the ring of formal complex Laurent series in the variable $q$. 

\begin{thm}[\cite{nava} Theorem 5.13]
    For $C$ a generalized Cartan matrix and $\vec{Q}$ an orientation of its Dynkin diagram, $\C((q))\otimes_{\Z[[q]][q^{-1}]} K_0(\He(\vec{Q}))\simeq \C((q))\otimes_{\C(q)} U_q^+(\mathfrak{g}(C))$.
\end{thm}
This isomorphism restricts to that of  Theorem \ref{decat} when we look at the full subcategory $\He^{fg}$. We see that the right multiplication 2-representation of Example \ref{leftmult} extends naturally to $\He$.

We can deduce from Theorem \ref{pbw} that $1_{*i}H_{\alpha+\alpha_i}\in \mathcal{H}_\alpha$, and therefore we can define our new adjoint functors 
\[E_i^*: \mathcal{H}_{\alpha+\alpha_i}\rightarrow \mathcal{H}_\alpha, E_i^*(M)=1_{*i}H_{\alpha+\alpha_i}\otimes_{\alpha+\alpha_i}M.\]

This functor descends to an operator $[E_i^*]$ on $\C((q))\otimes_{\Z[[q]][q^{-1}]}K_0(\mathcal{H}(\vec{Q}))$. Restricting to the subring $U_q^+(\mathfrak{g}(C))$, it is not hard to see that $[E_i^*]$ agrees with the right adjoint to right multiplication by $E_i$ under the semilinear form $(\bar{*},*)_L$.  Since $\bar{E_i}=E_i$, we see $[E_i^*]$ is also right adjoint under $(*,*)_L$, and thus equals the action of the $f_i\in B(\mathfrak{g})$. One could also see this by using the categorification of the $q$-boson relations in Proposition $\ref{prop:kk}$ below and induction over the weight spaces of $U_q^+(\mathfrak{g}(C))$. We henceforth use $E_i^*$ to refer to both the functor and the corresponding linear operator when no ambiguity is possible. Note that this 1-representation can be defined on $U_q^+(\mathfrak{sl}_{n+1})$ even though we needed to extend scalars of the Grothendieck group to Laurent series. 

Categorifying the action of $E_0\in U_q^+(\hat{\mathfrak{sl}}_{n+1})$ is more difficult, since at the level of algebras, its definition involves subtractions. To categorify the innermost bracket $[E_2^*,E_1^*]_q=E_2^*E_1^*-qE_1^*E_2^*$, we look for an object-wise injective natural transformation $\tau:E_1^*E_2^*\rightarrow E_2^*E_1^*$ of degree $1$ and take the cokernel. It is easy to see that left multiplication by $\tau_{|\alpha|-1}$ works. Categorifying the remaining brackets amounts to iterating this process. Our functor for $E_0$ will therefore be tensoring by a $(H_{\alpha-\sum \alpha_i},H_\alpha)$-bimodule that is defined as a quotient of $1_{*n,n-1,\dots 1}H_\alpha$. This functor is constructed in Subsection \ref{dtf}. Our categorification result is stated precisely in Theorem \ref{thm:thebigkahuna}. But first, we will need to introduce some technical tools.

\subsection{Weak Bruhat order}\label{bruhat}
We review the weak Bruhat order on $S_k$ and least common multiples for this order. These techniques will be crucial in the rest of our construction. The results in this subsection are known to experts, see for example \cite{lcmmeetjoin}.

The weak Bruhat order gives us a way to compare reduced presentations for different elements of $S_k$.
\begin{defn}
    The \emph{weak (right) Bruhat order} on $S_k$ is the partial ordering given by $\sigma \leq \tau$ iff there are reduced presentations $\sigma=s_{i_1}\dots s_{i_l}$ and $\tau=s_{i_1}\dots s_{i_{l'}}$ for some $l'\geq l$. If $\sigma \leq \tau$, we informally say that $\tau$ is \emph{prefixed} by $\sigma$. We also say that any reduced presentation of $\tau$ with $\sigma$ as a leftmost subword is prefixed by $\sigma$.
\end{defn}
 Note that if $\sigma \leq \tau$ then $l(\sigma)\leq l(\tau)$.
\begin{prop}
    For any $\sigma\in S_k$, $l(s_i\sigma)=l(\sigma)-1$ iff $s_i\leq \sigma$.
\end{prop}
We use the following tool as a convenient way of comparing elements in the weak Bruhat order.

\begin{defn}
For $\sigma\in S_k$, denote by $\text{In}_\sigma$ the \emph{(left) inversion set} of $\sigma$, i.e., the set of pairs $(a,b)$ for $1\leq a < b \leq k$ for which $\sigma^{-1} (a)>\sigma^{-1}(b)$. For a pair $(a,b)\in \text{In}_\sigma$, we say that $\sigma^{-1}$ \emph{inverts} $a$ and $b$.
\end{defn}

\begin{prop}\label{prop:inversion}
    For $\sigma,\omega\in S_k$, $\sigma \leq \omega$ iff $\text{In}_\sigma \subset \text{In}_\omega$.
\end{prop}
\begin{prop}\label{prop:redinversion}
    For $\sigma$, $\omega\in S_k$, $l(\sigma\omega)=l(\sigma)+l(\omega)$ iff $\text{In}_{\sigma} \subset \text{In}_{\sigma\omega}$.
\end{prop}
We often need to compare multiple different reduced presentations of a fixed $\sigma\in S_k$. This is made easier by the existence of least common multiples for the weak Bruhat order.
\begin{prop}[\cite{lcmmeetjoin} Theorem 3.2.1]
    For any two $\sigma,\tau\in S_k$, there exists a unique least common multiple $\Delta$ in $S_k$, i.e., $\sigma, \tau \leq \Delta$ and $\Delta$ is least for this property. 
\end{prop}
Such a $\Delta$ is sometimes also called the \emph{join} of $\sigma$ and $\tau$.
\begin{ex}\label{ex:basecase}
The least common multiple of $s_i$ and $s_j$ for $|i-j|>1$ is $s_is_j=s_js_i$. The least common multiple of $s_i$ and $s_{i+1}$ is $s_is_{i+1}s_i=s_{i+1}s_is_{i+1}$.
\end{ex}

We provide a proof for the following lemma, since we could not find it stated explicitly in the literature.
\begin{lem}\label{lem:abpres}
Fix $s_i$ and $\sigma\in S_k$ for which $l(s_i\sigma)=l(\sigma)+1$. Then if $\Delta$ is the least common multiple of $s_i$ and $\sigma$, then $s_i\sigma \leq \Delta$.
\end{lem}
The simplest case of this lemma can be seen in Example \ref{ex:basecase}.
\begin{proof}
Pick a reduced presentation of $\Delta$ that is prefixed by $\sigma$. Since $s_i\nleq \sigma$, applying the weak exchange condition to this presentation shows that $\sigma \leq s_i\Delta $, and so $s_i\sigma \leq \Delta$.
\end{proof}
We conclude this subsection with a more well-known result.
\begin{thm}[\cite{bourb}]\label{thm:bourb}
Any two reduced presentations for a fixed element of $S_k$ can be obtained from each other by applying a sequence of the relations $s_is_j=s_js_i$ for $|i-j|>1$ and $s_is_{i+1}s_i=s_{i+1}s_is_{i+1}$.
\end{thm}
Combining Theorems \ref{thm:bourb} and \ref{pbw} and using induction gives us the following.
\begin{cor}\label{cor:error}
    Suppose we have two reduced presentations of the same element $s_{i_1}\dots s_{i_l}=s_{j_1}\dots s_{j_l}$ in $S_k$. Fix any $\alpha \in \N[I]$ with $|\alpha|=k$, and $v\in I^\alpha$. Fix now reduced presentations for all $\sigma\in S_k$, and denote by $\tau_\sigma$ the associated product of $\tau_i$ in $H_\alpha$.
    
    We have in $H_\alpha$ that $\tau_{i_1}\dots \tau_{i_l}1_v\in \tau_{j_1}\dots \tau_{j_l}1_v+\bigoplus_{l(\sigma)<l} \kf[x_1,\dots x_k]\tau_\sigma 1_v$.
\end{cor}

\subsection{Defining the \texorpdfstring{$E_0$}{E0} functor}\label{dtf}
We define the endofunctor of $\mathcal{H}$ corresponding to the action of $E_0\in U_q^+(\hat{\mathfrak{sl}}_{n+1})$ on $U_q^+(\mathfrak{sl}_{n+1})$. Much like the functors for the other $E_i$, it can be expressed as tensoring by a certain bimodule. The most important result in this subsection is the one providing an explicit description of this module, as this will enable various computations later.

 Following the inclusion of monoidal categories $\U_q(\mathfrak{sl}_{n+1})\hookrightarrow \U_q(\hat{\mathfrak{sl}}_{n+1})$, the objects $E_i$ for $1\leq i \leq n$ will act as induction along $\alpha_i$ exactly as in the right multiplication representation of Example \ref{leftmult}. We will define the functor corresponding to $E_0$ in accordance with the ideas expressed in Subsection \ref{targets}. Namely, our functor will have the form $\bigoplus_\alpha M_\alpha \otimes_\alpha \text{\textendash}$, where $M_\alpha$ is a $(H_{\alpha-\alpha_1-\dots \alpha_n},H_\alpha)$-bimodule quotient of $1_{*n,n-1,\dots 1}H_\alpha$. We will first describe the bimodule more abstractly to highlight its relationship to certain modules studied in the literature before.

We introduce some notation for KLR algebras before proceeding. First, for $i\leq n$, we write $\beta_i:=\sum_{j\leq i} \alpha_j$, and $\vec{\beta}_i:=(i,i-1,\dots 1)$. We will also use $\beta:=\beta_n$ and $\vec{\beta}:=\vec{\beta}_n$. Moreover, for any $\alpha\in \N[I]$, $1\leq i\leq |\alpha|$, and $1\leq j\leq |\alpha|-1$, we write that $x_{-i}:=x_{|\alpha|-i+1}$ and $\tau_{-j}:=\tau_{|\alpha|-j}$. Note that in this notation we have $(\tau_{-1}x_{-1}-x_{-2}\tau_{-1})1_{*11}=1_{*11}$, and that under a module embedding $H_\alpha\hookrightarrow 1_{*i}H_{\alpha+\alpha_i}$, we have $x_{-i}\rightarrow x_{-i-1}1_{*i}$ and $\tau_{-i}\rightarrow \tau_{-i-1}1_{*i}$.

The main component of our construction of $M_\alpha$ is the standard right module $\Delta(\beta)$ as in \cite{bkm} associated to a certain ordering of the positive roots. This module also arises as a certain affinization in the sense of \cite{affinizationklr}. We review a construction of $\Delta(\beta)$ due to \cite{bkm}. First, we order the positive roots with the ``down and to the left" convex ordering, i.e. the one given by $\alpha_n < \alpha_{n-1}+\alpha_n<\alpha_{n-1}<\alpha_{n-2}+\alpha_{n-1}+\alpha_n<\dots <\alpha_1$. So, $\beta_{i-1}>\alpha_i$. 
With this order, we have for any $i\leq n$ that $(\beta_{i-1},\alpha_i)$ is a \emph{minimal pair} for $\beta_i$ in the sense of \cite{bkm}. Define the simple root standard (right) modules $\Delta(\alpha_i)$ as the right regular graded representation of $H_{\alpha_i}=\kf[x]$ and note $\beta_1=\alpha_1$. We define the right $H_{\beta_i}$-modules $\Delta(\beta_i)$ inductively following Section 4.4 of \cite{bkm}.
\begin{prop}[\cite{bkm} Lemma 4.9 and Theorem 4.10]\label{prop:bkmses}
    For any $i<n$, there is an injective homomorphism of right $H_{\beta_{i+1}}$-modules $q\Delta(\beta_i)\otimes \Delta(\alpha_{i+1})\rightarrow \Delta(\alpha_{i+1})\otimes \Delta(\beta_{i})$ descended from left multiplication by $\tau_1\tau_2\dots \tau_{n-1}$.
\end{prop}
We define $\Delta(\beta_i)$ to be the corresponding quotient. So, we have that $\Delta(\beta)$ is the quotient of $ 1_{\vec{\beta}}H_\beta$ by the right $H_\beta$-submodule generated by the $\{1_{\vec{\beta}}\tau_j\tau_{j+1}\dots \tau_{n-1}\}_{j\leq n-1}$.

\begin{defn}
    For any $\alpha\in \N[I]$, we define the \emph{affine bimodule} $M_\alpha$ to be the $(H_{\alpha-\beta},H_\alpha)$-bimodule $H_{\alpha-\beta}\otimes \Delta(\beta)$, where we view $H_{\alpha-\beta}$ as a right $H_{\alpha-\beta}$-module and $\Delta(\beta)$ as a right $H_\beta$-module for the monoidal product. It inherits the left $H_{\alpha-\beta}$-module structure from $H_{\alpha-\beta}$. It is projective as a left $H_{\alpha-\beta}$-module since $H_{\alpha-\beta}$ is, although it is infinitely generated. As a left $H_{\alpha-\beta}$-module, $M_\alpha$ is in $\mathcal{H}_\alpha$.
    
    Due to exactness of the monoidal product, we may also explicitly define $M_\alpha$ as the quotient \[1_{*\vec{\beta}}H_\alpha/1_{*\vec{\beta}}(\tau_{-1},\tau_{-2}\tau_{-1},\dots, \tau_{-(n-1)}\dots \tau_{-1})H_\alpha.\] Here, the denominator is the $(H_{\alpha-\beta},H_\alpha)$-subbimodule of $1_{*\vec{\beta}}H_\alpha$ generated by the set \[\{1_{*\vec{\beta}}\tau_{-1},1_{*\vec{\beta}}\tau_{-2}\tau_{-1},\dots,1_{*\vec{\beta}}\tau_{-(n-1)}\dots\tau_{-1}\}.\]
\end{defn}
These generators can be simplified.

\begin{lem}\label{lem:newgens}
    The graded $(H_{\alpha-\beta},H_\alpha)$-subbimodules of $1_{*\vec{\beta}}H_\alpha$ generated by  
    \[\{1_{*\vec{\beta}}\tau_{-1},1_{*\vec{\beta}}\tau_{-2},\dots,1_{*\vec{\beta}}\tau_{-(n-1)}\}\]
    and
    \[\{1_{*\vec{\beta}}\tau_{-1},1_{*\vec{\beta}}\tau_{-2}\tau_{-1},\dots,1_{*\vec{\beta}}\tau_{-(n-1)}\dots\tau_{-1}\}\]
    coincide.
\end{lem}

\begin{proof}
    We compute that $1_{*\vec{\beta}}\tau_{-k}\dots \tau_{-1}\times \tau_{-1}\dots \tau_{-(k-1)}=1_{\vec{\beta}}\tau_{-k}$. 
\end{proof}

We now have a well-defined functor $E_{0,\alpha}\coloneqq M_\alpha \otimes_\alpha \text{\textendash}:\He_\alpha\rightarrow \He_{\alpha-\beta}$. Taking the direct sum gives our desired functor 
\[E_0\coloneqq \bigoplus_\alpha E_{0,\alpha}.\] Although $E_0$ resembles a restriction functor, this functor is not the adjoint of the functor for taking the monoidal product with a standard left module for $\beta$ due to the $q$-antilinearity of taking adjoints.

Before giving a more explicit description of $M_\alpha$, we show that the functor $E_0$ yields exactly the representation described in the introduction upon decategorifying. More precisely, we first write that the lowest root vector for $\mathfrak{sl}_{n+1}$ can be written in terms of the Chevalley generators as $v=[F_n,[F_{n-1},\dots [F_2,F_1]]]$. We will show that, after taking Grothendieck groups, the functor $\bigoplus_\alpha M_\alpha\otimes_\alpha \text{\textendash}$ acts as $[E_n^*,[E_{n-1}^*,\dots [E_2^*,E_1^*]_q]_q]_q$, where $[A,B]_q=AB-qBA$. The functors $1_{*i}H_\alpha\otimes_\alpha \text{\textendash}$ for restriction along $\alpha_i$ categorify the $E_i^*$. We categorify a bracket $[A,B]_q$ by finding an injective natural transformation $BA\rightarrow AB$ of degree 1, and taking the cokernel. 

\begin{lem}\label{lem:K0action}
    On $K_0(\He)$, the action of $[E_0]$ is $[E_n^*,[E_{n-1}^*,[\dots[E_2^*,E_1^*]_q]_q]_q]_q$.
\end{lem}
\begin{proof}
   Denote by $M_{\alpha,i}$ the $(H_{\alpha-\beta_i},H_\alpha)$-bimodule $H_{\alpha-\beta_i}\otimes \Delta(\beta_i)$, where $H_{\alpha-\beta_i}$ is viewed as a right $H_{\alpha-\beta_i} $-module for the monoidal product. If we take the monoidal product of $H_{\alpha-\beta_{i+1}}$ with the defining short exact sequence for $\Delta(\beta_{i+1})$, we obtain 
    \[M_{\alpha-\alpha_{i+1},i}\otimes_{\alpha-\alpha_{i+1}}1_{*i+1}H_\alpha\hookrightarrow 1_{*i+1}H_{\alpha-\beta_i}\otimes_{\alpha-\beta_i}M_{\alpha,i}\twoheadrightarrow M_{\alpha,i+1},\] where the injection is degree $1$. As a short exact sequence of projective left modules, this sequence splits.
    So, for any graded projective left $H_{\alpha}$-module $B$, we can tensor this sequence to obtain the split-exact sequence
    \[(M_{\alpha-\alpha_{i+1},i}\otimes_{\alpha-\alpha_i}1_{*i+1}H_\alpha)\otimes_\alpha B\hookrightarrow (1_{*i+1}H_{\alpha-\beta_i}\otimes_{\alpha-\beta_i}M_{\alpha,i})\otimes_\alpha B\twoheadrightarrow M_{\alpha,i+1}\otimes_\alpha B.\] This gives the following equality in $K_0(\He_{\alpha-\beta_{i+1}})$:
    \[[(1_{*i+1}H_{\alpha-\beta_i}\otimes_{\alpha-\beta_i}M_{\alpha,i})\otimes_\alpha B]-q[(M_{\alpha-\alpha_{i+1},i}\otimes_{\alpha-\alpha_i}1_{*i+1}H_\alpha)\otimes_\alpha B]=[M_{\alpha,i+1}\otimes_\alpha B].\]
    Since we have that $[1_{*i+1}H_\alpha\otimes_\alpha \text{\textendash}]$ acts via $E_{i+1}^*$, an induction on $i$ gives the claim.
\end{proof}

We will now produce an explicit basis for $M_\alpha$ as a left $H_{\alpha-\beta}$-module. This will be needed for later computations. Our approach to doing so is to explicitly split the defining quotient map from $1_{*\vec{\beta}}H_\alpha$ to $M_\alpha$. We will partition the PBW basis elements of Theorem \ref{pbw} for a particular set of reduced presentations.

The result we prove is somewhat more general than we need. There are a few reasons for this. Firstly, the more general result may be needed if one wants to construct a similar functor outside of type $A_n$. Additionally, this sort of explicit splitting of a quotient map may be of independent interest.

For now, we fix $\vec{Q}$ \emph{any} quiver without simple 
edge loops, $\alpha\in \N[I]$, and $v\in I^\beta$ for some $\beta \leq \alpha$ coordinate-wise. We set $n=|\beta|$. As before, we drop $\vec{Q}$ from our algebras.

 Let $S$ be any subset of $S_{|\alpha|}$. We can partition $S_{|\alpha|}$ by defining $P\subset S_{|\alpha|}$ as the set of all $\omega\in S_{|\alpha|}$ for which there exists $\sigma\in S$ with $\sigma \leq \omega$, and denote by $P^c$ its complement in $S_{|\alpha|}$. We will produce a corresponding splitting of $1_{*v}H_\alpha$ via the PBW basis. Before we can attempt to split the PBW basis of $1_{*v}H_\alpha$, we will need to once and for all fix a reduced presentation for each $\omega\in S_{|\alpha|}$. We do so as follows.  We define a total order on reduced presentations of a given symmetric group element. For two reduced presentations of the same element, $\sigma=s_{i_1}s_{i_2}\dots s_{i_k}=s_{j_1}s_{j_2}\dots s_{j_k}$, we say that $s_{j_1}s_{j_2}\dots s_{j_k} > s_{i_1}s_{i_2}\dots s_{i_k}$ iff the list of indices is larger lexicographically, i.e. $j_1 > i_1$ or  $j_1=i_1$, $j_2>i_2$ or $j_1=i_1,j_2=i_2, j_3> i_3\dots $. For each $ \sigma \in P^c$, pick its minimal presentation. For $\sigma\in P$, pick the maximal presentation among those prefixed by the $\omega$ for $\omega\in S$.  The precise choice of presentation is not as important, but is convenient for our later analysis. These presentations will be referred to as the \emph{chosen} reduced presentations to emphasize that this is somewhat arbitrary, but to also distinguish them for use in proof. For a chosen reduced presentation $\sigma=s_{i_1}\dots s_{i_k}$, we associate an element of the KLR algebra $\tau_\sigma :=\tau_{i_1}\dots \tau_{i_k}$. If we need to associate a KLR algebra element to some non-chosen reduced presentation, we will denote this element as $\tau'_\sigma$ or $\tau''_\sigma$.

Suppose we also have a splitting of graded left $\kf[x_1,\dots x_{-(n+1)}]$-modules $1_{*v}\kf[x_1,\dots x_{-1}]=L\oplus R$. For example, one could use $1_{*v}\kf[x_1,\dots x_{-1}]=1_{*v}\kf[x_1,\dots x_{-n}]\oplus 1_{*v}(x_{-n}-x_{-(n-1)},\dots x_{-2}-x_{-1})\kf[x_1,\dots x_{-1}]$. This gives a splitting of left $\kf[x_1,\dots x_{-(n+1)}]$-modules
\[1_{*v}H_\alpha = 1_{*v}\bigoplus_{\sigma\in P^c}L\tau_\sigma \oplus 1_{*v}(R\bigoplus_{\sigma \in P^c} \tau_\sigma \oplus \bigoplus_{\sigma \in P} \kf [x_1,\dots x_{-1}]\tau_\sigma).\]
For convenience, we denote now the first summand as $N_l$ and the second as $N_r$. Our main idea is that under certain conditions on $v$, $L$, $R$, and $S$ that this will in fact be a splitting of left $H_{\alpha-\beta}$ modules. 
\begin{thm}\label{thm:newproj}
    Suppose that we have an $(H_{\alpha-\beta},H_\alpha)$-subbimodule $N$ of $1_{*v}H_\alpha$ along with subsets $G_x,G_\tau\subset N$ and a left $\kf[x_1,\dots x_{-(n+1)}]$-module $L$ that satisfy the following conditions. 
    \begin{enumerate}
    \item \label{it:mostpoly} The set $R:= 1_{*v}\kf[x_1,\dots,x_{-1}]\cap N$ of all polynomials in $N$ yields a splitting of graded left $\kf[x_1,\dots x_{-(n+1)}]$-modules $1_{*v}\kf[x_1,\dots x_{-1}]=L\oplus R$, and $L\supset 1_{*v}\kf[x_1,\dots x_{-n}]$. Moreover, for any $i < |\alpha|-n$, we have that $\tau_i L \subset L+L\tau_i$. 
    \item\label{it:rightside} $N=H_{\alpha-\beta}1_{*v}(G_x\cup G_\tau)H_\alpha$, where $G_x$ is a finite set of polynomials contained in $\kf[x_{-n}\dots x_{-1}]$, and $G_\tau$ is a set of words in the $\tau_{-i}$ for $i < n$. Note that the elements of $G_x$ and $G_\tau$ commute with elements of $H_{\alpha-\beta}$. For each $\tau_{i_1}...\tau_{i_k}\in G_\tau$, we require that $s_{i_1}...s_{i_k}$ is a reduced expression for the corresponding element of $S_{|\alpha|}$.  
        \item\label{it:goodpoly} For each $\tau\in G_\tau$, the braid diagram associated to $1_{*v}\tau$ does not swap any two braids with the same label. Because of this, the braid relations hold in $1_{*v}H_\alpha$ for these $1_{*v}\tau$. We will therefore denote elements of $G_\tau$ by $\tau_\sigma$, where $\sigma$ is the corresponding $S_{|\alpha|}$ element. We note that this notation for $\tau_\sigma$ agrees with the associated $1_{*v}H_\alpha$ element of $\sigma$. We take $S$ to be the set of $\sigma$ for which $\tau_\sigma\in G_\tau$, and define $P$, $P^c$, $N_r$, and $N_l$ as above.
       
        \item\label{it:lcmerror} ``Lcm errors" of $G_\tau$ elements are prefixed by $G_\tau$ elements. Let $\sigma, \sigma'\in S_{|\alpha|}$ be such that $\tau_\sigma$, $\tau_{\sigma'}\in G_\tau$. Let $\Delta$ be the least common multiple of $\sigma$ and $\sigma'$ in $S_{|\alpha|}$, i.e., the unique least element above both $\sigma$ and $\sigma'$ in the weak Bruhat order. Then if $s_{i_1}\dots s_{i_k}$ and $s_{j_1}\dots s_{j_k}$ are two reduced presentations for $\Delta$, then $1_{*v}(\tau_{i_1}\tau_{i_2}\dots \tau_{i_k}-\tau_{j_1}\dots \tau_{j_k})\in N_r$.
        
        \item\label{it:removethissoon} Fix $\tau_\sigma\in G_\tau$ with reduced presentation $\tau_\sigma=\tau_{i_1}\dots \tau_{i_k}$. Fix also any $j \leq k$ and denote $\Delta = lcm(s_{i_j}, s_{i_{j+1}}\dots s_{i_{k}})$. Then if $\tau'_\Delta$ and $\tau''_\Delta$ are the KLR algebra elements associated to any two reduced presentations of $\Delta$, then $1_{*v}(\tau_{i_1}\dots \tau_{i_{j}}\tau'_\Delta-\tau_{i_1}\dots \tau_{i_j}\tau''_{\Delta})\in N_r$. In case $j=k$, we say  $\Delta=s_{i_k}$. 
        \item\label{it:isodelete} ``Isomorphism deletions" of $G_\tau$ elements are prefixed by $G_\tau$ elements.  Fix $\tau_\sigma\in G_\tau$ with reduced presentation $\tau_\sigma=\tau_{i_1}\dots \tau_{i_k}$. Suppose some $\tau_{i_j}$ is an isomorphism, i.e., swaps braids of non-adjacent and distinct labels in the braid diagram for $1_{*v}\tau_\sigma$. Denote $\Delta = lcm(s_{i_j}, s_{i_{j+1}}\dots s_{i_{k}})$. Then if $1_{*v}\tau'_{s_{i_j}\Delta}$ is a $1_{*v}H_\alpha$ element associated to any reduced presentation of $s_{i_j}\Delta$, then we have $1_{*v}\tau_{i_1}\dots \tau_{i_{j-1}}\tau'_{s_{i_j}\Delta}\in N_r$. In case $j=k$, we say  $\Delta=s_{i_k}$.

        \item\label{it:nonisodelete} ``No other polynomial coefficients in $N$".  Fix $\tau_\sigma\in G_\tau$ with reduced presentation $\tau_\sigma=\tau_{i_1}\dots \tau_{i_k}$. Suppose some $\tau_{i_j}$ is not an isomorphism, i.e., swaps braids of adjacent labels in the braid diagram for $1_{*v}\tau_\sigma$. Denote $\Delta=lcm(s_{i_j},s_{i_{j+1}}\dots s_{i_k})$. Then if $1_{*v}\tau'_{s_{i_j}\Delta}$ is a $1_{*v}H_\alpha$ element associated to any reduced presentation of $s_{i_j}\Delta$, then we have that $1_{*v}\tau_{i_1}\dots \tau_{i_{j-1}}(x_{i_j}-x_{i_j+1})^l\tau'_{s_{i_j}\Delta}\in N_r$. Here, $l=-C_{a,b}$ for $a$ and $b$ the labels of the two braids swapped by $\tau_{i_j}$ in $1_{*v}\tau_{\sigma}$. In case $j=k$, we say  $\Delta=s_{i_k}$.
    \end{enumerate}
    Then $N=N_r$, and there is a splitting of left $H_{\alpha-\beta}$-modules $1_{*v}H_\alpha=N_l\oplus N_r$. 
\end{thm}

 The set $R$ is closed under multiplication by any $\kf[x_1,\dots x_{-1}]$ element since $N$ is a right $H_\alpha$-module. This implies $N_r$ is also closed under multiplication on the left by any polynomial. Conditions \ref{it:isodelete} and \ref{it:nonisodelete} could easily be combined into one, but are kept separate for conceptual reasons. In practice, condition \ref{it:isodelete} is best viewed as a restriction on $G_\tau$, whereas condition \ref{it:nonisodelete} is best viewed as a restriction on $R$. In fact, in this article, we will only need the weaker form of condition \ref{it:nonisodelete} in which, after bringing the polynomial to the left side, the polynomial coefficient of $1_{*v}\tau_{i_1}\dots \tau_{i_{j-1}}(x_{i_j}-x_{i_{j}+1})^l$ is in $N$. Conditions \ref{it:goodpoly}, \ref{it:lcmerror}, and \ref{it:removethissoon} are immediate in the cases considered in this paper. The set $G_x$ is also not needed in this paper, but may be useful for types besides $A_n$. We will eventually take  $v=\vec{\beta}$, $G_x=\emptyset$, $G_\tau=\{\tau_{-1},\dots \tau_{-(n-1)}\}$, $R=1_{*v}(x_{-1}-x_{-2},\dots x_{-(n+1)}-x_{-n})\kf[x_1,\dots x_{-1}]$, and $L=1_{*v}\kf[x_1,\dots x_{-n}]$ to deduce the relevant result for the affine bimodule $M_\alpha$, so these motivating choices should be kept in mind during the proof. 
 
 Note that while $N$ may have multiple generating sets like $G_\tau$, not all of them will satisfy conditions \ref{it:lcmerror} through \ref{it:nonisodelete}. These 4 conditions essentially tell us how to take a given set of generators of this form and complete them to one that is compatible with the sorts of error terms that arise in this proof. A basic example of this procedure is in Lemma \ref{lem:newgens}.
Note also that this is generally not a splitting of right $H_\alpha$-modules.

\begin{proof}\renewcommand{\qedsymbol}{}
In this proof, we generally drop the idempotent $1_{*v}$ from elements of $1_{*v}H_\alpha$ for simplicity. We first show that $N_r$ is exactly $N$, which implies also that $N_r$ is a left $H_{\alpha-\beta}$-module. By conditions \ref{it:mostpoly}, \ref{it:rightside}, and \ref{it:goodpoly}, we see that $N_r \subset N$. For the reverse inclusion, first note that $G_x 1_{*v}H_\alpha\subset N_r$ by definition. For the various $\tau_\sigma\in G_\tau$, we study the image of the maps 
    \[\tau_\sigma\times:1_{*\sigma^{-1}(v)}H_\alpha\rightarrow 1_{*v}H_\alpha.\]
   
    We have the following decomposition of $\kf[x_1,\dots x_{-1}]$ modules;
    \[1_{*\sigma^{-1}(v)}H_\alpha = \bigoplus_{\omega\in S_{|\alpha|}}\kf[x_1,\dots x_{-1}]1_{*\sigma^{-1}(v)}\tau_\omega.\]
    By condition \ref{it:goodpoly}, $\tau_{\sigma}\times  (\kf[x_1,\dots x_{-1}]1_{*\sigma^{-1}(v)})\subset \kf[x_1,\dots x_{-1}]\tau_{\sigma} 1_{*\sigma^{-1}(v)}$. $N_r$ is closed under multiplication by any polynomial, so we need only show each $1_{*v}\tau_{\sigma} \tau_\omega\in N_r$. Fix $\omega \in S_{|\alpha|}$. We argue by induction on $l(\omega)+l(\sigma)$. We have two cases.

    \begin{enumerate}
        \item $l(\sigma\omega) < l(\sigma)+l(\omega)$. Write $\sigma=s_{i_1}s_{i_2}\dots s_{i_k}$ as its chosen reduced presentation. Then for some $j\leq k$, we have that $l(s_{i_{j+1}}\dots s_{i_k}\omega)=l(\omega)+k-j$, and $\omega':=s_{i_{j+1}}\dots s_{i_k}\omega$ has a reduced presentation prefixed by $s_{i_j}$. If $\Delta$ is the least common multiple of $s_{i_j}$ and $s_{i_{j+1}}\dots s_{i_k}$, then $\omega'$ therefore has a reduced presentation prefixed by $\Delta$ and $s_{i_{j+1}}\dots s_{i_k}$. Pick such a presentation, and fix the resulting presentation of $\Delta$. We consider the difference $\tau_{\sigma} \tau_\omega-\tau_{i_1}\dots \tau_{i_{j}}\tau'_\Delta \tau'_{\Delta^{-1}\omega'}$ where $\tau'_\Delta$ is associated to this new presentation, and $\tau'_{\Delta^{-1}\omega'}$ is the $1_{*v}H_\alpha$ element associated to some reduced presentation of $\Delta^{-1}\omega'$. By choice of the presentation of $\Delta$, this difference is $\tau_{\sigma}$ times a sum of error terms, and so this difference is in $N_r$ by induction.  By Lemma \ref{lem:abpres}, $\Delta$ has a presentation prefixed by $s_{i_j}\dots s_{i_k}$. If $\tau''_\Delta$ is associated to such a presentation, then by condition \ref{it:removethissoon} and induction, $\tau_{i_1}\dots \tau_{i_j}\tau'_\Delta\tau'_{\Delta^{-1} \omega'}-\tau_{i_1}\dots \tau_{i_j}\tau''_\Delta\tau'_{\Delta^{-1} \omega'}\in N_r$. We now compute $\tau_{i_1}\dots \tau_{i_j}\tau''_{\Delta}\tau'_{\Delta^{-1} \omega'}=\tau_{i_1}\dots \tau_{i_j}^2\tau'_{s_{i_j}\Delta}\tau'_{\Delta^{-1}\omega'}$. By condition \ref{it:goodpoly},  $\tau_{i_j}$  either swaps two braids of adjacent labels or of distinct non-adjacent labels. If this $\tau_{i_j}$ swaps two braids of adjacent labels, then $\tau_{i_1}\dots \tau_{i_j}^2\tau'_{s_{i_j}\Delta}\tau'_{\Delta^{-1}\omega'}=\pm \tau_{i_1}\dots \tau_{i_{j-1}}(x_{i_j}-x_{i_j+1})^l\tau'_{s_{i_j}\Delta}\tau'_{\Delta^{-1}\omega'}$. By condition \ref{it:nonisodelete} and induction, this is contained in $N_r$. If instead this $\tau_{i_j}$ swaps braids of non-adjacent distinct labels, then $\tau_{i_1}\dots \tau_{i_j}^2\tau_{s_{i_j}\Delta}'\tau'_{\Delta^{-1}\omega'}=\tau_{i_1}\dots \tau_{i_{j-1}}\tau_{s_{i_j}\Delta}'\tau'_{\Delta^{-1}\omega'}$. By condition \ref{it:isodelete} and induction, this too is contained in $N_r$. Altogether, this shows $\tau_{\sigma} \tau_\omega\in N_r$ as well.
        \item $l(\sigma\omega)=l(\sigma)+l(\omega)$. Then $\sigma\omega\in P$. If the chosen presentation for $\sigma\omega$ is prefixed by $\sigma$, then $\tau_{\sigma} \tau_\omega-\tau_{\sigma\omega}=\tau_{\sigma} \tau_\omega-\tau_{\sigma} \tau'_\omega=\tau_{\sigma}(\tau_\omega-\tau'_\omega)$, where $\tau'_\omega$ corresponds to some  presentation of $\omega$. By induction, this difference is in $N_r$. $\tau_{\sigma\omega}$ is in $N_r$ by definition, so this implies $\tau_{\sigma} \tau_\omega$ is as well.

        Now suppose instead there exists a distinct $\sigma'\in S$ for which the chosen presentation of $\sigma\omega$ is prefixed by $\sigma'$. Firstly, we take a reduced presentation for $\sigma\omega$ that is prefixed by both $\Delta:=lcm(\sigma,\sigma')$ and the chosen reduced presentation of $\sigma$. This yields also a particular reduced presentation of $\omega$. By a similar argument to above, $\tau_{\sigma} \tau_\omega-\tau_{\sigma} \tau'_\omega \in N_r$, where $\tau'_\omega$ is the $1_{*v}H_\alpha$ element associated to the new presentation for $\omega$. We may rewrite $\tau_{\sigma} \tau'_\omega = \tau'_\Delta \tau'_{\Delta^{-1}\sigma\omega}$, where $\tau'_\Delta$ is associated to a presentation for $\Delta$ that is prefixed by $\sigma$, and $\tau'_{\Delta^{-1}\sigma\omega}$ is the element associated to some reduced presentation of $\Delta^{-1}\sigma\omega$. By condition \ref{it:lcmerror} and induction, the difference $\tau'_\Delta \tau'_{\Delta^{-1}\sigma\omega}-\tau_\Delta \tau'_{\Delta^{-1}\sigma\omega}\in N_r$. Now, the chosen reduced presentations for $\Delta$  must begin in the chosen reduced presentation of $\sigma'$ by choice of reduced presentation. $\tau_\Delta \tau'_{\Delta^{-1}\sigma\omega}-\tau_{\sigma\omega}$ is $\tau_{\sigma'}$ times a sum of error terms. Induction gives that this difference is in $N_r$. Finally, $\tau_{\sigma\omega}$ is of course in $N_r$. Altogether, these facts imply $\tau_{\sigma} \tau_\omega\in N_r$.
    \end{enumerate}
\end{proof}

We require a few easy technical lemmas before we can prove that $N_l$ also has the structure of a left $H_{\alpha-\beta}$-module.

  \begin{lem}\label{prop1}
        Let $\sigma\in P^c$, and let its chosen reduced presentation be $s_{i_1}s_{i_2}\dots s_{i_k}$. Then if $i_1 <|\alpha|-n$, then $\sigma':=s_{i_2}\dots s_{i_k}\in P^c$ and has the given reduced presentation as its chosen one.
    \end{lem}
    \begin{proof}
         The elements of $S$ are words in the $s_{-i}$ for $i < n$ by condition \ref{it:rightside}.  Since $i_1 < |\alpha|-n$, we have that $s_{i_1}$ commutes with $\omega$ for each $\omega \in S$. So, $\sigma\in P^c$ iff $\sigma'$ is. The claim about the reduced presentation follows from the lexicographic choice of reduced presentations for $P^c$ elements.
    \end{proof}
    \begin{cor}\label{pcdecomp}
        For $\sigma\in P^c$, we may uniquely write $\sigma=\sigma'\sigma''$, where $\sigma'\in S_{|\alpha|-n}$, $\sigma''$ either is prefixed by $s_{j}$ for $j\geq |\alpha|-n$ or is trivial, and the chosen reduced presentation of $\sigma$ is the concatenation of those of $\sigma'$ and $\sigma''$.
    \end{cor}
     \begin{lem}\label{prop2}
        Let $\sigma\in P^c$, and let its chosen reduced presentation be $s_{i_1}s_{i_2}\dots s_{i_k}$. Let also $\sigma'\in S_{|\alpha|-n}$. If $i_1 \geq |\alpha|-n$, then $\sigma'\sigma\in P^c$ and its chosen reduced presentation is the concatenation of those of $\sigma'$ and $\sigma$.
    \end{lem}

    \begin{proof}
        Such a $\sigma'$  commutes with each of the $\omega$ with $\omega\in S$. So, a similar argument again gives that $\sigma'\sigma\in P^c$ iff $\sigma$ is. For the second claim, note that by Proposition \ref{prop:redinversion}, the concatenation of reduced presentations is not reduced iff there is a pair $1\leq a < b \leq |\alpha|$ with $(a,b)\in \text{In}_{\sigma'}$ and $(a,b)\notin \text{In}_{\sigma'\sigma}$, i.e. if $\sigma' \not\leq \sigma'\sigma$. By choice of reduced presentation, Proposition \ref{prop:inversion} gives that $\sigma^{-1}$ cannot invert any pair amongst the positions $1$ through $|\alpha|-n$. Since ${\sigma'}^{-1}$ can only invert such pairs, we have that $(a,b)\in \text{In}_{\sigma'}\implies (a,b)\in \text{In}_{\sigma'\sigma}$. The claim follows from this and from the choice of reduced presentation.   
    \end{proof}

    \begin{lem}\label{prop4}
         Let $\sigma\in P^c$, and factor it as $\sigma=\sigma'\sigma''$ according to Corollary \ref{pcdecomp}. Suppose that for some $a < b \leq |\alpha|-n$ we have that $(a,b)\in \text{In}_{\sigma}$. Then also $(a,b)\in \text{In}_{\sigma'}$.
    \end{lem}
    \begin{proof}
    Suppose towards contradiction that $\sigma'^{-1}(a) < \sigma'^{-1}(b)$. We must have $(\sigma'^{-1}(a), \sigma'^{-1}(b))\in \text{In}_{\sigma''}$. By definition of $\sigma'$, $\sigma'^{-1}(a)$ and $\sigma'^{-1}(b) \leq |\alpha|-n$. Therefore, Proposition \ref{prop:inversion} implies that $s_i \leq \sigma''$ for some $i < |\alpha|-n$, contradicting the definition of $\sigma''$ and the choice of reduced presentations. 

    \end{proof}

    \begin{proof}[Proof (of Theorem \ref{thm:newproj})]
    We now show that $N_l$ is a left $H_{\alpha-\beta}$-module.
    For $i < |\alpha|-n$, we study the image of the maps
    \[\tau_i\times :1_{*v}H_\alpha\rightarrow 1_{*v}H_\alpha.\]
    By condition \ref{it:mostpoly},  \[\tau_i(\bigoplus_{\sigma\in P^c}1_{*v}L\tau_\sigma)\subset \bigoplus_{\sigma\in P^c}1_{*v}L\tau_\sigma \oplus \bigoplus_{\sigma \in P^c}1_{*v}L\tau_i\tau_\sigma.\] So, it is sufficient to show that for $\sigma\in P^c$, we have that $\tau_i\tau_\sigma\in N_l$.  Note that $l(s_i\sigma)=l(\sigma)+1$ iff $\sigma$ has no reduced presentation prefixed by $s_i$. Following Corollary \ref{pcdecomp}, we may write $\sigma=\sigma'\sigma''$ where $\sigma'\in S_{|\alpha|-n}$, the first term of $\tau_{\sigma''}$ is $\tau_{j}$ for some $j\geq |\alpha|-n$, and $\tau_\sigma=\tau_{\sigma'}\tau_{\sigma''}$. We argue by induction on $l(\sigma)$. We have two cases.
    \begin{enumerate}
         
         \item 
        $l(s_i\sigma)=l(\sigma)+1$. By Corollary \ref{cor:error}, $\tau_i\tau_{\sigma'}-\tau_{s_i\sigma'} \in \bigoplus\limits_{\substack{\omega\in S_{|\alpha|-n}\\
        l(\omega)<l(s_i\sigma')}}\kf[x_1,\dots x_{-(n+1)}]\tau_{\omega}$.  Induction on $l(\sigma)$ implies $(\tau_i\tau_{\sigma'}-\tau_{s_i\sigma'})\tau_{\sigma''}\in N_l$. As stated above, $\tau_i\tau_{\sigma'}\tau_{\sigma''}=\tau_i\tau_\sigma$. Lemma \ref{prop2} implies $\tau_{s_i\sigma'}\tau_{\sigma''}=\tau_{s_i\sigma}$. We therefore have $\tau_i\tau_\sigma-\tau_{s_i\sigma}\in N_l$. Evidently, $\tau_{s_i\sigma}\in N_l$. We then have that $\tau_i\tau_\sigma=\tau_{s_i\sigma}+(\tau_i\tau_\sigma-\tau_{s_i\sigma})\in N_l$  as well.
      
    \item $l(s_i\sigma)=l(\sigma)-1$. Then by Lemma \ref{prop4}, since $\sigma$ has a reduced presentation prefixed by $s_i$, so does $\sigma'$. Applying the argument of case 1, we have that $\tau_{\sigma}-\tau_i\tau_{s_i\sigma}\in N_l$. An induction on length shows that $\tau_i(\tau_\sigma-\tau_i\tau_{s_i\sigma})=\tau_i\tau_\sigma-\tau_i^2\tau_{s_i\sigma}$ is contained in $N_l$. Note that, depending on the vector $w$, $1_{w}\tau_i^2$ may be $0$  or $\pm1_{w}(x_i-x_{i+1})^k$ for some $k\geq 0$. In any case, $\tau_i^2\tau_{s_i\sigma}$ is in $N_l$. This implies $\tau_i\tau_\sigma$ is as well.
    \end{enumerate}
    \end{proof}

   \begin{cor}\label{cor:affinebimodulebasis}
       The $(H_{\alpha-\beta},H_\alpha)$-bimodule $M_\alpha$ has a vector space basis given by the equivalence classes of the elements $\{1_{w\vec{\beta}}p(x_1,\dots x_{-n})\tau_\sigma\}$, where $w$ is an arbitrary element of $I^{\alpha-\beta}$, $p(x_1,\dots x_{-n})$ is an arbitrary monomial, and $\sigma$ is an arbitrary element of $S_{|\alpha|}$ not prefixed by any of $s_{-1}$ through $s_{-(n-1)}$.
   \end{cor}
   Note that taking $\alpha=\beta$ shows that $\Delta(\beta)\simeq \kf[x_1]$ as a vector space.
   \begin{proof}
        Lemma \ref{lem:newgens} gives us the relevant construction of $M_\alpha$. More precisely, the objects in Theorem \ref{thm:newproj} match up to those in the definition of $M_\alpha$ via $v=\vec{\beta}$, $N=1_{*\vec{\beta}}(\tau_{-1},\tau_{-2},\dots \tau_{-(n+1)})H_\alpha$, $G_x=\emptyset$, $G_\tau=\{\tau_{-1},\dots \tau_{-(n-1)}\}$, and $L=1_{*\vec{\beta}}\kf[x_1,\dots x_{-n}]$. Condition \ref{it:rightside} is immediate, and conditions \ref{it:goodpoly}, \ref{it:lcmerror}, and \ref{it:removethissoon} follow from noting that $\vec{\beta}$ has no repeated coordinates, and so the braid relations hold for the relevant $\tau_\sigma$. Condition \ref{it:isodelete} holds since none of these $\tau_{-i}$ are isomorphisms. Condition \ref{it:nonisodelete} holds since the $G_\tau$ elements are all length one. We compute that $1_{*\vec{\beta}}\tau_{-i}^2=1_{*\vec{\beta}}(x_{-(i+1)}-x_{-i})$ when $i< n$. It is easy to see from here that $R=1_{*\vec{\beta}}(x_{-1}-x_{-2},x_{-2}-x_{-3},\dots x_{-(n-1)}-x_{-n})\kf[x_{1},\dots x_{-1}]$, so condition \ref{it:mostpoly} holds. All conditions hold, so Theorem \ref{thm:newproj} applies. 
   \end{proof}

Note that Theorem \ref{thm:newproj} is strictly stronger than the earlier results of this subsection. Since $1_{*\vec{\beta}}H_\alpha$ is projective as a left $H_{\alpha-\beta}$-module, the theorem can be used to prove projectivity of the left module $M_\alpha$. One can also deduce Lemma \ref{lem:K0action}.

\subsection{Controlling error terms}\label{subsec:control}
We state and prove a few technical lemmas that will be used later. The natural transformations that we define in the next subsection all come from maps that make sense at the level of KLR algebras. At that level, the desired KLR relations only hold up to error terms involving shorter elements of the symmetric group. The lemmas here show that these error terms disappear upon passing to the quotient $M_\alpha$. This subsection will also make clear why there are only a few valid choices for our presentation of the lowest weight root vector.

For the rest of this subsection, we fix an ADE type quiver $\vec{Q}$ with vertex set $I$ and fix $\alpha\in \N[I]$.  We denote $k:=|\alpha|$ for simplicity of notation. As before, we drop $\vec{Q}$ from our algebras.

\begin{defn}
    Given any $\sigma\in S_k$, we say that a triple $1\leq a < b < c \leq k$ is a \emph{triple inversion} for $\sigma$ if $\sigma^{-1}(a) > \sigma^{-1}(b) > \sigma^{-1}(c)$.
\end{defn}
\begin{defn}
    Given any fixed $\sigma\in S_{k}$ and $v\in I^\alpha$, we define the \emph{obstruction set} of $(\sigma,v)$ to be the set $\mathcal{S}$ of triple inversions $(a,b,c)$ of $\sigma$ with $v_a=v_c$, and $v_b$ adjacent to $v_a$ in our Dynkin diagram. For $\tau_\sigma 1_{v}\in H_\alpha 1_{v}$ associated to any reduced presentation of $\sigma$, we may also refer to this set as the obstruction set of $\tau_\sigma 1_{v}$.
    \end{defn}
 These sorts of triples are exactly those for which $\tau_1\tau_2\tau_1 1_{v_av_bv_c}\neq \tau_2\tau_1\tau_21_{v_av_bv_c}$. We will only use this set $\mathcal{S}$ for bookkeeping purposes.

Theorem \ref{thm:bourb} immediately gives the following.
\begin{prop}
    If $\mathcal{S}$ is empty, then all associated  $H_\alpha 1_v$ elements for $\sigma$ are equal, i.e. the braid relations hold in $H_\alpha 1_v$ for $\sigma$.
\end{prop}
In the quivers of interest (ADE type), adjacent vertices are connected by a unique edge. So, for $(a,b,c)\in \mathcal{S}$, we have that the error $(\tau_1\tau_2\tau_1-\tau_2\tau_1\tau_2)1_{v_av_bv_c}=\pm 1_{v_av_bv_c}$. So, performing three-term swaps of this form within a longer element, we introduce ``error terms" that delete certain intermediate $
\tau_i$. The error terms that arise from these swaps are exactly the obstruction to the braid relations holding for $\sigma$ in the KLR algebra.

Many of the upcoming calculations for our natural transformations involve only certain types of symmetric group elements. Fix $v=v_1v_2v_3\dots v_k\in I^\alpha$ a sequence of vertices in our quiver. Fix some $1\leq k_1\leq k_2\leq k$, and consider the subwords $\vec{r}=v_1\dots v_{k_1}$, $\vec{m}=v_{k_1+1}\dots v_{k_2}$, $\vec{l}=v_{k_2+1}\dots v_{k}$. Then we write that $v$ is the concatenation $\vec{r}\vec{m}\vec{l}$. Consider the swapped concatenation $\vec{l}\vec{m}\vec{r}$. In the symmetric group $S_k$, there is a unique element $\sigma$ taking $\vec{r}\vec{m}\vec{l}$ to $\vec{l}\vec{m}\vec{r}$, where we view each position as distinct regardless of label. However, in the KLR algebra $H_{\alpha}$, the braid relations do not necessarily hold, and therefore there may be multiple ways of expressing this swap. Note also that in the KLR algebra, our notation is such that $\tau_\sigma=1_{\vec{l}\vec{m}\vec{r}}\tau_\sigma 1_{\vec{r}\vec{m}\vec{l}}$, for $\tau_\sigma$ any associated $H_\alpha$ element of $\sigma$.

The main sorts of computations we will have to consider are those in which $\vec{l}$ is a singleton. In this case, while the braid relations may not hold in the KLR algebra, we show that they will hold in our quotient module $M$ or in the appropriate tensor. All braid diagrams in this paper will be drawn from top to bottom.

\begin{lem}\label{lem:errorobst}
    Suppose that $\vec{l}$ is length 1. Suppose $\sigma$ has a reduced presentation of the form \[\sigma=\sigma' s_{i}s_{i\pm 1}s_{i}\sigma'',\] 
    where $\sigma'$ and $\sigma''$ are  symmetric group elements written in a reduced presentation. Then the error element $\sigma'\sigma''$ has no triple inversions.
\end{lem}
\begin{proof}
Note that any triple inversion of $\sigma$ consists of a single element from each of $\vec{l}$, $\vec{m}$, and $\vec{r}$. So, the reduced braid diagram for such an error element will be as follows.

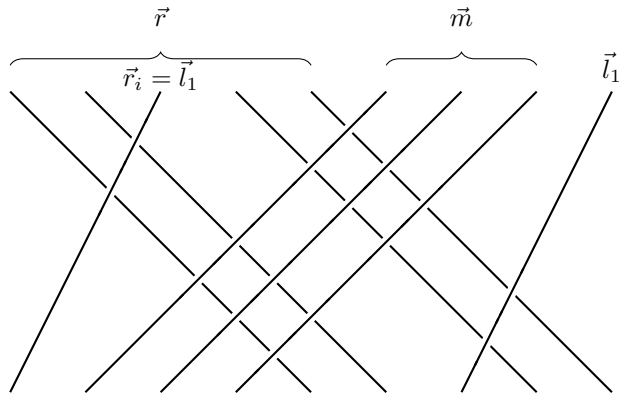
\begin{center}
\begin{tikzpicture}
\draw[decorate,decoration={brace,amplitude=5pt,raise=-1ex}](0,4.5) -- (4, 4.5);
\draw (2, 4.75) node[above]{$\vec{r}$};

\draw (2,3.9) node[above] {$\vec{r}_i=\vec{l}_1$};

\draw[decorate,decoration={brace,amplitude=5pt,raise=-1ex}](5,4.5) -- (7, 4.5);
\draw (6, 4.75) node[above]{$\vec{m}$};

\draw (8,4) node[above] {$\vec{l}_1$};
\begin{knot}[
clip width=5,
clip radius=8pt,
]
\strand [thick] (0,0)
to  (2,4);

\strand[thick] (1,0) to  (5,4);

\strand[thick] (2,0) to 
(6,4);

\strand[thick] (3,0) to (7,4);

\strand[thick] (4,0) to (0,4);

\strand[thick] (5,0) to (1,4);

\strand[thick] (6,0) to (8,4);

\strand[thick] (7,0) to (3,4);

\strand[thick] (8,0) to (4,4);

\end{knot}
\end{tikzpicture}
\end{center}
The claim is true by inspection.
\end{proof}

The following is the main result of this subsection.

\begin{lem}\label{rml}
    Suppose that $\vec{l}$ is length one and its sole vertex $v_k$ is adjacent to at most only the leftmost vertex $v_{k_1+1}$ of $\vec{m}$, where as above, $k_1$ is the rightmost position of $\vec{r}$ as it appears in $v$. Suppose also that $\mathcal{S}$ does not contain any triple $(k_1,k_i,k)$. Then, for any two associated KLR algebra elements $\tau_\sigma$ and $\tau'_\sigma\in H_\alpha 1_v$, we have that $\tau_\sigma 1_v=\tau'_\sigma 1_v$ up to addition of an element in the right submodule of $1_{\sigma(v)}H_\alpha$ generated by \[1_{\sigma(v)}\tau_{k-1}, 1_{\sigma(v)}\tau_{k-2}\tau_{k-1},\dots, 1_{\sigma(v)}\tau_{k-k_1+1}\dots \tau_{k-1}.\] 
\end{lem}

Note that even though $\vec{l}$ is length one, the size of $\mathcal{S}$ may be more than one depending on the contents of $\vec{m}$ and $\vec{r}$.

\begin{proof}

    By Theorem \ref{thm:bourb} we may assume that the two reduced presentations of $\sigma$ differ by a single braid move. So, we write that $\tau_\sigma 1_v = \tau_{i_1}\dots \tau_{i}\tau_{i+1}\tau_{i}\dots \tau_{i_j}1_v$ and $\tau'_\sigma 1_v = \tau_{i_1}\dots \tau_{i+1}\tau_{i}\tau_{i+1}\dots \tau_{i_j}1_v$. If there is no $\mathcal{S}$ member resulting from this braid move, then $\tau_\sigma 1_v=\tau'_\sigma 1_v$. Otherwise, the difference of these two elements is the error term $\epsilon:=\tau_{i_1}\dots \tau_{i_j}$. If this presentation for $\epsilon$ is not reduced, then by assumption on $\vec{m}$, and $\vec{l}$, it is only because it contains a subexpression of the form given in the diagram below. 
    
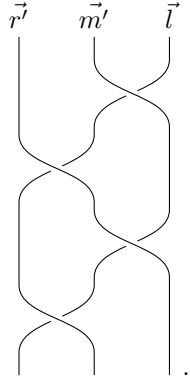
\begin{center}
    \begin{tikzpicture}
      
        \pic[
        ] at (6,0)
         {braid = {s_2 s_1 s_2 s_1}};
    
    \draw (6,0) node[above] {$\vec{r'}$};
    \draw (7,0) node[above] {$\vec{m'}$};
    \draw (8,0) node[above] {$\vec{l}$};

    \end{tikzpicture}.
\end{center}
In this diagram, $\vec{r'}$ is a subset of the $\vec{r}$ elements to the right of the $v_k$ copy involved in this swap, and $\vec{m'}$ is a subset of the $\vec{m}$ elements 
 to the right of $v_{k_1+1}$. By assumption on the adjacency of $\vec{m}$ and $\vec{l}$, the braid relations hold for this element, and so we may cancel the double swap of $\vec{l}$ with $\vec{m'}$ in $\epsilon$ to get either $1$ or $0$. We may therefore assume that the given presentation of $\epsilon$ is already reduced. So, $\epsilon$ is an associated $H_\alpha 1_v$ element for some $\sigma_\epsilon\in S_k$.  Next, Lemma \ref{lem:errorobst} shows that the obstruction set for $\sigma_\epsilon$ and this $v$ is empty, so it is enough to argue that $\sigma_\epsilon$ is prefixed by some $s_{k-1},\dots,  s_{k-k_1+1}\dots s_{k-2}s_{k-1}$. By assumption on $\mathcal{S}$, $\sigma_\epsilon$ does not fix $k$. The claim is now immediate from the diagram in Lemma \ref{lem:errorobst}.

\end{proof}

Note that this lemma also tells us how to identify which generator of this right ideal will prefix our error terms; if $\epsilon$ is such an error term resulting from the deletion of the triple $(a,b,k)$, then $\epsilon(k)=\sigma(a)$.

 We will conclude this subsection with some notation and a slight modification of the previous lemma. Recall that for any $p\in I$ and $n\in \N$, $H_{np}$ is the $n$-th nil affine Hecke algebra. Since each $\tau_i^2=0$ on this algebra, we cannot view the $\tau_i$ as generators of the group algebra for $S_n$. However, there is an embedding of algebras $\kf[q,q^{-1}][S_n]\hookrightarrow H_{np}$ given by $s_i\rightarrow (x_i-x_{i+1})\tau_i+1$. We henceforth denote by $s_i$ its image in $H_{np}$. Note that $\text{deg}(s_i)=0$. We collect a few simple results about these $s_i$ below. They are all deduced quickly from the KLR algebra relations.
\begin{prop}\label{sirels}
The following hold in $H_{np}$.
\begin{enumerate}
\item $s_ix_i=x_{i+1}s_i$.
\item $s_ix_{i+1}=x_is_i$.
\item $s_i=\tau_ix_{i+1}-x_{i+1}\tau_i=x_i\tau_i-\tau_ix_i$.
\item $s_is_{i+1}\tau_i=\tau_{i+1}s_is_{i+1}$.
\item $\tau_is_{i+1}s_i=s_{i+1}s_i\tau_{i+1}$.

\item 
$s_i\tau_{i+1}s_i=s_{i+1}\tau_is_{i+1}$

\end{enumerate}
\end{prop}
Braid relations involving one $s$ and two $\tau$ in $H_{np}$ may not hold, however, these will not be needed in the following subsections. If for some $v\in I^\alpha$ and $j< |\alpha|$ we have that $v_j=v_{j+1}$, then we also identify $s_j$ with its image under the obvious non-unital embedding $H_{2i}\hookrightarrow H_\alpha$.

The following is also easily verified.
\begin{prop}\label{sirels2}
Let $p$ and $q$ be adjacent vertices in $I$. The following hold in $H_{2\alpha_p+\alpha_q}$
\begin{enumerate}
    \item $\tau_1\tau_2s_11_{ppq}=s_2\tau_1\tau_2 1_{ppq}$.
    \item $s_1\tau_2\tau_1 1_{qpp}=\tau_2\tau_1s_2 1_{qpp}$.
    \item 
    $\tau_1s_2\tau_1 1_{pqp}=\tau_2 s_1\tau_2 1_{pqp}$.
\end{enumerate}
\end{prop}
\begin{rem}\label{s_iswap}
    There are certain cases in the next subsections in which some of the hypotheses of Lemma \ref{rml} will fail, e.g. $\vec{l}$ will be a singleton, but its unique element does appear within one of the forbidden swaps involving the rightmost element of $\vec{r}$. To remedy this, we will modify our assignment $S_{k}\times I^\alpha \rightarrow H_\alpha$. Before, we would pick a reduced presentation $\sigma=s_{i_1}\dots s_{i_l}$ for each member of $S_{k}$, and map $(\sigma,v)$ to $\tau_{i_1}\dots \tau_{i_l}1_v$. For certain $\sigma$ and $v$, we will instead replace some of the $\tau_{i_j}$ by $s_{i_j}$ via the algebra embedding above. Due to this, the relevant braid relations involving this $s_{i_j}$ hold exactly in the KLR algebra. So, one can effectively remove the problematic elements from $\mathcal{S}$ and run the  arguments of Lemma \ref{rml} again while also moving around the $s_i$ terms according to Propositions \ref{sirels} and \ref{sirels2}.
\end{rem}

\subsection{Defining the natural transformations}\label{dtnt}

We define the natural transformation data of our affine 2-representation. 

The natural transformations we use will each have a common form independent of $\alpha$, so we can define our maps and make arguments weight-wise. Since all of the functors for our 2-representation are given by tensoring with a bimodule, all natural transformations between composites of these functors arise from bimodule homomorphisms. The natural transformations $X_i$, $T_{ii}$, and $T_{ij}$ for $i,j\neq 0$ are exactly those in the right multiplication 2-representation of $U_q(\mathfrak{sl}_{n+1})$ on itself. These all correspond to right multiplication by the corresponding element in the KLR algebra.

We now describe the natural transformations that involve the functor $E_0=\bigoplus_\alpha M_\alpha \otimes_{\alpha} \text{\textendash}$. Define the natural transformation $X_0:E_0\rightarrow E_0$ as the map descended from left multiplication by $x_{-1}$ on $1_{*\vec{\beta}}H_\alpha$. This map is well-defined on $M_{\alpha}$ since left multiplication by $x_{-1}$ preserves the subbimodule $1_{*\vec{\beta}}(\tau_{-1},\dots \tau_{-(n-1)})H_\alpha$. Recall that the polynomials $1_{*\vec{\beta}}(x_{-j}-x_{-1})\in 1_{*\vec{\beta}}(\tau_{-1},\dots \tau_{-(n-1)})H_\alpha$ for each $j \leq n$. So, $X_0$ is equal to the natural transformation for left multiplication by any such $x_{-j}$. This natural transformation corresponds to the map $x\in \text{End}_{H_{\alpha}}(\Delta(\beta))_2$ defined in Section 3 of \cite{bkm}.

Natural transformations $E_0E_0\rightarrow E_0E_0$ correspond to bimodule endomorphisms of $M_{\alpha-\beta}\otimes_{\alpha-\beta} M_{\alpha}$. By the usual tensor product properties, this bimodule is equal to $H_{\alpha-2\beta}\otimes \Delta(\beta)\otimes \Delta(\beta)$, and it is a bimodule quotient of $1_{*\vec{\beta},\vec{\beta}}H_{\alpha}$. Now, let $\omega_{00}\in S_{|\alpha|}$ be the unique element swapping each position $-i$ with $-n-i$ for $i \leq n$. We define $T_{00}:E_0E_0\rightarrow E_0E_0$ as the map descended from left multiplication by $-\tau_{\omega_{00}}$. Note that $\omega_{00}$ has no triple inversions. Therefore, the braid relations hold for $1_{*\vec{\beta},\vec{\beta}}\tau_{\omega_{00}}$, and so our definition of $T_{00}$ does not depend on any specific reduced presentation of $\omega_{00}$. The work of \cite{bkm} shows that this natural transformation is well-defined. This natural transformation is similar to an R-matrix for KLR algebras in the sense of \cite{affinizationklr}.

\begin{prop}[\cite{bkm} Lemma 3.6]\label{lem:T00}
    $T_{00}$ is a well-defined natural transformation of degree -2.
\end{prop}

The presence of the minus sign is due to the fact that iterating $E_0$ fills in the indices in our idempotent $1_{v}$ from right to left as opposed to the left-to-right order in any other $E_i$. It is shown in Lemma 3.8 of \cite{bkm} that $\text{End}_{H_{m\beta}}(\Delta(\beta)^{\otimes m})\simeq H_{m\alpha_1}$. The ``y-axis reflection" involution of $H_{m\alpha_1}$ sending $x_i$ to $x_{-i}$ will send $\tau_i$ to $-\tau_{-i}$. The KLR relations will not hold between these natural transformations without the presence of the minus sign. We also appear to need the positive sign on $X_0$, as seen in part (3) of the proof of Theorem \ref{thm:thebigkahuna}.

We now define the $T_{i0}$ and $T_{0i}$ for $1\leq i \leq n$. These are also descended from maps of KLR algebras. The bimodule associated to $E_iE_0$ is $H_{\alpha-\beta+\alpha_i}1_{*i}\otimes_{\alpha-\beta} M_\alpha$, and the bimodule associated to $E_0E_i$ is $M_{\alpha+\alpha_i}1_{*i}$. 

To obtain our natural transformations, we will compose several maps of the following form.
\begin{prop}[\cite{kk} Corollary 3.4]\label{prop:kk}
    There is an isomorphism of graded $(H_{\alpha-\alpha_j+\alpha_i},H_\alpha)$-bimodules

    \[q^{-C_{ij}}H_{\alpha-\alpha_j+\alpha_i}1_{*i}\otimes_{\alpha-\alpha_j}1_{*j}H_\alpha \oplus \delta_{ij}H_{\alpha}[x_{|\alpha|+1}]\simeq 1_{*j}H_{\alpha+\alpha_i}1_{*i}\]
    given by
    \[(a\otimes_{\alpha-\alpha_j}b,\delta_{ij}c)\rightarrow a\tau_{-1}b + \delta_{ij}c.\]
    We have used the left embeddings of KLR algebras, i.e. $(1_{vi}\otimes_{\alpha-\alpha_j}1_{vj},1_{w})\rightarrow 1_{vij}\tau_{-1}1_{vji}+\delta_{ij}1_{wi}c$.
\end{prop}
Kang and Kashiwara observe in \cite{kk} that this yields natural isomorphisms of functors $q^{-C_{ij}}E_jE_i^* \oplus \delta_{ij} Id\otimes k[x_{-1}]\simeq E_i^*E_j$, which is a categorical version of the $q$-boson relations. Composing several of these morphisms,  we have the following decomposition of $(H_{\alpha-\beta+\alpha_i},H_{\alpha})$ bimodules.
\begin{cor}\label{cor:bigdecomp}
    There is an isomorphism of  $(H_{\alpha-\beta+\alpha_i},H_\alpha)$-bimodules
    \[  H_{\alpha-\beta+\alpha_i}1_{*i} \otimes_{\alpha-\beta}1_{*\vec{\beta}}H_{\alpha}\oplus 1_{*n,n-1,\dots i+1}H_{\alpha-\beta_{i-1}}[x_{|\alpha|-i+2}]\otimes_{\alpha-\beta_{i-1}}1_{*\vec{\beta}_{i-1}}H_{\alpha}\xrightarrow{\psi} 1_{*\vec{\beta}}H_{\alpha+\alpha_i}1_{*i} \]
    given by 
    \[\psi(a\otimes_{\alpha-\beta}b,c\otimes_{\alpha-\beta_{i-1}}d)=a\tau_{-n}\dots \tau_{-1}b+c\tau_{-(i-1)}\dots \tau_{-1}d,\]
    where the second sequence of $\tau_j$ is taken to be empty if $i=1$. Restricted to the first summand, the morphism $\psi$ is degree $-1$ if $i=1$ or $n$, and is degree $0$ otherwise.
\end{cor}

Let $N_{\alpha+\alpha_i}\subset 1_{*\vec{\beta}}H_{\alpha+\alpha_i}1_{*i}$ be the $(H_{\alpha-\beta+\alpha_i},H_\alpha)$-subbimodule $1_{*\vec{\beta}}(\tau_{-1},\tau_{-2},\dots \tau_{-(n-1)})H_{\alpha+\alpha_i}1_{*i}$, and let $N_{\alpha}\subset 1_{*\vec{\beta}}H_\alpha$ be the $(H_{\alpha-\beta},H_\alpha)$-subbimodule $1_{*\vec{\beta}}(\tau_{-1},\dots \tau_{-(n-1)})H_\alpha$.

For $i\neq 1$, denote by $\bar{\psi}$ the map $\psi$. If instead $i=1$, denote $\bar{\psi}$ as the map  \[  H_{\alpha-\beta+\alpha_i}1_{*1} \otimes_{\alpha-\beta}1_{*\vec{\beta}}H_{\alpha}\oplus 1_{*n,n-1,\dots 2}H_{\alpha}[x_{|\alpha|+1}]\xrightarrow{\bar{\psi}} 1_{*\vec{\beta}}H_{\alpha+\alpha_1}1_{*1} \]
    given by 
    \[\bar{\psi}(a\otimes_{\alpha-\beta}b,c)=a\tau_{-n}\dots s_{-1}b+c.\]
\begin{lem}\label{lem:Ti0welldef}
$\bar{\psi}(H_{\alpha-\beta+\alpha_i}1_{*i}\otimes_{\alpha-\beta} N_{\alpha})\subset N_{\alpha+\alpha_i}.$
\end{lem}
\begin{proof}
    By linearity and Lemma \ref{lem:newgens}, it is enough to consider the elements $1_{*i}\otimes_{\alpha-\beta} 1_{*\vec{\beta}}\tau_{-k}\dots \tau_{-1}$. We compute that \[\psi(1_{*i}\otimes_{\alpha-\beta} 1_{*\vec{\beta}}\tau_{-k}\dots \tau_{-1})= 1_{*i,\vec{\beta}}\tau_{-n}\dots \tau_{-1}\tau_{-(k+1)}\dots \tau_{-2}.\] The corresponding labelled braid diagram is as follows
    
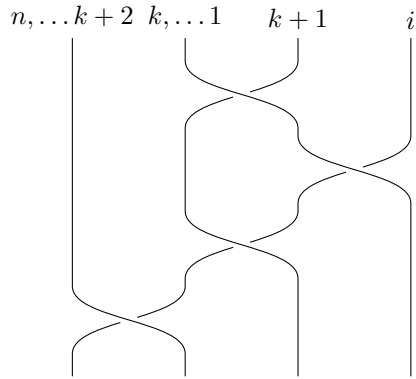
\begin{center}
    \begin{tikzpicture}
      
        \pic[braid/.cd,
        width=15mm,
        ] at (6,0)
         {braid = {s_2 s_3 s_2 s_1}};
    \draw (6,0) node[above] {$n,\dots k+2$};
    \draw (7.5,0) node[above] {$k,\dots 1$};
    \draw (9,0) node[above] {$k+1$};
    \draw (10.5,0) node[above] {$i$};
  
    \end{tikzpicture}
\end{center}
If $i\neq 1$, we can argue similarly to Lemma \ref{lem:T00}, and so Lemma \ref{rml} gives that this is contained in $N_{\alpha+\alpha_i}$. If $i=1$, then following Remark \ref{s_iswap}, we may take the obstruction set as empty due to the replacement of $\tau_{-1}$ by $s_{-1}$. Therefore, the braid relations hold for this element, and so it is also contained in $N_{\alpha+\alpha_i}$.
\end{proof}
As a result, $\bar{\psi}$ descends to a morphism $H_{\alpha-\beta+\alpha_i}1_{*i}\otimes_{\alpha-\beta} M_{\alpha}\rightarrow M_{\alpha+\alpha_i}1_{*i}$. We therefore define 
\[T_{i0}:E_iE_0\rightarrow E_0E_i, a\otimes_{\alpha-\beta}b\rightarrow a\tau_{-n}\dots \bar{\tau}_{-i}\dots \tau_{-1} b,\]
where $\bar{\tau}_{-n}=s_{-n}$, $\bar{\tau}_{-1}=s_{-1}$, and $\bar{\tau}_{-i}=\tau_{-i}$ otherwise. Note that, as elements of $M_{\alpha+\alpha_i}1_{*i}$, we have $s_{-n}\dots \tau_{-1}=((x_{-n-1}-x_{-n})\tau_{-n}+1)\dots \tau_{-1}=(x_{-n-1}-x_{-n})\tau_{-n}\dots \tau_{-1}+\tau_{-(n-1)}\dots \tau_{-1}=(x_{-n-1}-x_{-n})\tau_{-n}\dots \tau_{-1}$. This shows that $T_{n0}$ is in fact well-defined. Note also that $T_{i0}$ has the correct grading; it is zero if $i\neq 1,n$ and $1$ otherwise.

A more precise study of $\bar{\psi}$ tells us how to define $T_{0i}$.
\begin{lem}\label{lem:T0iwelldef}
If $i\neq 1$, then $\bar{\psi}^{-1}(N_{\alpha+\alpha_i})=H_{\alpha-\beta+\alpha_i}1_{*i}\otimes_{\alpha-\beta} N_{\alpha}\oplus 1_{*n,n-1,\dots i+1}H_{\alpha-\beta_{i-1}}[x_{|\alpha|-i+2}]\otimes_{\alpha-\beta_{i-1}}1_{*\vec{\beta}_{i-1}}H_{\alpha}$. If $i=1$, then this preimage of $\bar{\psi}$ restricted to $H_{\alpha-\beta+\alpha_i}1_{*1} \otimes_{\alpha-\beta}1_{*\vec{\beta}}H_{\alpha}$ is $H_{\alpha-\beta+\alpha_i}1_{*1}\otimes_{\alpha-\beta} N_\alpha$.
\end{lem}
\begin{proof}
    For $i\neq 1$, we compute $\bar{\psi}$ on elements of $1_{*n,n-1,\dots i+1}H_{\alpha-\beta_{i-1}}[x_{|\alpha|-i+2}]\otimes_{\alpha-\beta_{i-1}}1_{*\vec{\beta}_{i-1}}H_{\alpha}$. We note that this submodule is generated by elements of the form $p(x_{|\alpha|-i+2})\otimes_{\alpha-\beta_{i-1}} b$ for $p$ a polynomial. Then \[\bar{\psi}(p(x_{|\alpha|-i+2})\otimes_{\alpha-\beta_{i-1}} b)=1_{*\vec{\beta}}p(x_{-i})\tau_{-i+1}\dots \tau_{-2}\tau_{-1}b=1_{*\vec{\beta}}\tau_{-(i-1)}\dots \tau_{-2}\tau_{-1}p(x_{-1})b,\] which is evidently in $N_{n,\alpha+\alpha_i}$.

    For the reverse inclusion, the splitting from Theorem \ref{thm:newproj} gives that it is sufficient to show that nonzero elements in $H_{\alpha-\beta+\alpha_i}1_{*i}\otimes_{\alpha-\beta} \bigoplus_{\sigma\in P^c}1_{*\vec{\beta}}\kf[x_1,\dots x_{-n}]\tau_\sigma$ do not have image in $N_{\alpha+\alpha_i}$. Here, $P^c\subset S_{|\alpha|}$ is again the subset of symmetric group elements with no presentation prefixed by any of the $s_{|\alpha|-1}$ through $s_{|\alpha|-(n-1)}$. The result follows from Theorem \ref{thm:newproj} applied to the target $1_{*\vec{\beta}}H_{\alpha+\alpha_i} $ and the observation that for $\sigma\in P^c$, the $S_{|\alpha|+1}$ element $s_{-n}\dots s_{-1}\sigma$ does not have a reduced presentation prefixed by any of the $s_{-1}$ through $s_{-(n-1)}$.

    A similar argument works in the case $i=1$. We compute that $1_{*1,\vec{\beta}}\tau_{-n}\dots s_{-1}=1_{*1,\vec{\beta}}(x_{-(n+1)}-x_{-1})\tau_{-n}\dots \tau_{-1}+1_{*1,\vec{\beta}}\tau_{-n}\dots \tau_{-2}$. For any $\sigma\in S_{|\alpha|}$, we see that $l(s_{|\alpha|-n+1}\dots s_{|\alpha|}\sigma)=l(\sigma)+n$ and $l(s_{\alpha-n+1}\dots s_{|\alpha|-1}\sigma)\leq l(\sigma)+n-1$. So, in the PBW basis expansion, $\tau_{-n}\dots \tau_{-1}\tau_{\sigma}$ has a nonzero coefficient on $\tau_{s_{-n}\dots s_{-1}\sigma}$, and $\tau_{-n}\dots \tau_{-2}\tau_{\sigma}$ does not. So, it is enough to argue $1_{*1,\vec{\beta}}a(x_{-(n+1)}-x_{-1})\tau_{-n}\dots \tau_{-1}b\notin N_{\alpha+\alpha_1}$. We note that $(x_{-(n+1)}-x_{-1})\notin N_{\alpha+\alpha_1}$, so the argument is now the same as in the case $i\neq 1$.
\end{proof}

\begin{cor}\label{cor:tiso}
    For $i\neq 1$, the direct summand embedding $q^{-\delta_{in}}H_{\alpha-\beta+\alpha_i}1_{*i}\otimes_{\alpha-\beta}1_{*\vec{\beta}}H_\alpha \hookrightarrow 1_{*\vec{\beta}}H_{\alpha+\alpha_i}1_{*i}$ descends to an isomorphism $q^{-\delta_{in}}H_{\alpha-\beta+\alpha_i}1_{*i}\otimes_{\alpha-\beta}M_{\alpha}\xrightarrow{\sim} M_{\alpha+\alpha_i}1_{*i}$.
\end{cor}

For $i\neq 1$, we therefore define 
\[T_{0i}:E_0E_i\rightarrow E_iE_0\]
to be the inverse isomorphism. Note that this map also has the desired grading; $0$ if $i\neq n$ and $1$ otherwise.

We must handle the case of $i=1$ separately since $\psi$ is not as well-behaved in this case.

\begin{lem}\label{lem:t01proj}
    The degree 1 projection 
    \[ \phi:1_{*1,\vec{\beta}}H_{\alpha+\alpha_1}1_{*1}\twoheadrightarrow H_{\alpha-\beta+\alpha_1}1_{*1}\otimes_{\alpha-\beta}1_{*\vec{\beta}}H_\alpha\]
    descends to a surjective degree 1 morphism 
    \[ T_{01}:=M_{\alpha+\alpha_1}1_{*1}\twoheadrightarrow H_{\alpha-\beta+\alpha_1}1_{*1}\otimes_{\alpha-\beta} M_\alpha\]
\end{lem}

\begin{proof}
    We show that $\phi(N_{\alpha+\alpha_1})\subset H_{\alpha-\beta+\alpha_1}1_{*1}\otimes_{\alpha-\beta}N_\alpha$. Consider the maps
    \[\tau_{-k}\times:1_{*1,s_{-k}\vec{\beta}}H_{\alpha+\alpha_1}1_{*1}\rightarrow 1_{*1,\vec{\beta}}H_{\alpha+\alpha_1}1_{*1}\]
    for each $k < n$. We then must show that each $\phi \circ \tau_{-k}\times$ has image contained in $H_{\alpha-\beta+\alpha_1}1_{*1}\otimes_{\alpha-\beta} N_\alpha$. Similar to Corollary \ref{cor:bigdecomp}, the source of $\phi \circ \tau_{-k}\times$ decomposes as 
    \[ 1_{*1,s_{-k}\vec{\beta}}H_{\alpha+\alpha_1}1_{*1} \xrightarrow{\sim}^{\psi^{-1}} H_{\alpha-\beta+\alpha_1}1_{*1}\otimes_{\alpha-\beta}1_{*s_{-k}\vec{\beta}}H_\alpha \oplus 1_{*n,\dots k+2,k,k+1,\dots 2}H_{\alpha}[x_{|\alpha|+1}],\]
    where the idempotent in the right is $1_{*,n,\dots 2}$ if $k=1$. By the same corollary, the target decomposes as 
    \[ 1_{*1,\vec{\beta}}H_{\alpha+\alpha_1}1_{*1} \xrightarrow{\sim}^{\psi^{-1}} H_{\alpha-\beta+\alpha_1}1_{*1}\otimes_{\alpha-\beta}1_{*\vec{\beta}}H_\alpha \oplus 1_{*n,n-1,\dots 2}H_{\alpha}[x_{|\alpha|+1}].\]
    For the first summand in the domain, an element $a\otimes_{\alpha-\beta} b$ is mapped to 
    \begin{align*}
        a1_{*1}\otimes_{\alpha-\beta} 1_{*s_{-k}\vec{\beta}}b\xrightarrow{\psi} 1_{*1,s_{-k}\vec{\beta}}a\tau_{-n}\dots \tau_{-1}b & \xrightarrow{\tau_{-k}} 1_{*1,\vec{\beta}}\tau_{-k}a\tau_{-n}\dots \tau_{-1}b \\
        &= 1_{*1,\vec{\beta}}a\tau_{-n}\dots \tau_{-k}\tau_{-k-1}\tau_{-k}\dots \tau_{-1}b
    \end{align*}
    by the braid relations. If $k\neq 1$, then this equals $1_{*1,\vec{\beta}}a\tau_{-n}\dots \tau_{-1}\tau_{-(k+1)}b$, which projects to $a1_{*1}\otimes_{\alpha-\beta} 1_{*\vec{\beta}}\tau_{-k}b\in H_{\alpha-\beta+\alpha_1}1_{*1}\otimes_{\alpha-\beta} N_{\alpha}$. If instead $k=1$, then we have
    \begin{align*}
        1_{*1,\vec{\beta}}a\tau_{-n}\dots \tau_{-1}\tau_{-2}\tau_{-1}b &= 1_{*1,\vec{\beta}}a\tau_{-n}\dots \tau_{-1}\tau_{-2}b \pm 1_{*1,\vec{\beta}}a\tau_{-n}\dots \tau_{-3}b\\
        &\xrightarrow{\phi} a1_{*1}\otimes 1_{*\vec{\beta}}\tau_{-1}b+0\in H_{\alpha-\beta+\alpha_1}1_{*1}\otimes_{\alpha-\beta} N_{\alpha}.
    \end{align*}
    We now consider elements of the other summand of the domain, $a\in 1_{*n,\dots k+2,k,k+1,\dots 2}H_{\alpha}[x_{|\alpha|+1}]$. If $k\neq 1$, then this is mapped to 
    \[a \xrightarrow{\psi} 1_{*1,s_{-k}\vec{\beta}}a\xrightarrow{\tau_{-k}} 1_{*1,\vec{\beta}}\tau_{-k}a\xrightarrow{\phi} 0, \] and if $k=1$, then this is mapped to 
    \begin{align*}
    a \xrightarrow{\psi} 1_{1,s_{-1}\vec{\beta}}\tau_{-1}a &\xrightarrow{\tau_{-1}} 1_{*1,\vec{\beta}}\tau_{-1}\tau_{-1}a \\
        &=  1_{*1,\vec{\beta}} (x_{-2}-x_{-1})a\\
        &\xrightarrow{\phi} 0.
    \end{align*}
    So, $\phi$ maps $N_{\alpha+\alpha_1}$ into $H_{\alpha-\beta+\alpha_1}1_{*1}\otimes_{\alpha-\beta}N_{\alpha}$
\end{proof}

We can now formally state our main result. 

\begin{thm}\label{thm:thebigkahuna}
    The data of the endofunctor $E_0=\bigoplus_\alpha M_\alpha\otimes_\alpha \text{\textendash}$ and natural transformations from this subsection $T_{00}$, $T_{i0}$, $T_{0i}$, and $X_0$ for all $1\leq i \leq n$ extend the right multiplication 2-representation of $U_q^+(\mathfrak{sl}_{n+1})$ on $\He$ to a 2-representation of $U_q^+(\hat{\mathfrak{sl}}_{n+1})$.
\end{thm}

\subsection{Verifying KLR relations}\label{subsec:verify}
We conclude the construction of our 2-representation by checking that our natural transformations satisfy the KLR relations. At this point, all of the key insights have been made, and the remaining work is a series of straightforward verifications. The reader may then wish to skip this subsection on a first reading.

We first present a lemma that will allow us to check only a subset of the KLR relations.
\begin{lem}\label{klrreduce}
    For a quiver $\vec{Q}$ with vertex set $I$ along with an additive category $\mathcal{C}$, suppose that a 2-representation of $\U(\vec{Q})$ on $\mathcal{C}$ is given. Let $\vec{Q}'$ be any fixed quiver obtained from $\vec{Q}$ by adding a single new vertex $0$ along with any number of directed edges between $0$ and the vertices of $\vec{Q}$. Suppose that an additive endofunctor $E_0$ of $\mathcal{C}$ is given along with natural transformations $X_0\in End(E_0)$, $T_{00}\in End(E_0E_0), T_{0i}\in Nat(E_0E_i,E_iE_0),$ and $ T_{i0}\in Nat(E_iE_0,E_0E_i)$ for each $i\in I$. Suppose that the following KLR relations are satisfied by these new natural transformations for all $i,j\in I$. 
    \begin{align}
        T_{00}\circ X_0E_0-E_0X_0\circ T_{00}&=\text{id}_{00}, \label{smallklr1}\\
        T_{00}\circ E_0X_0-X_0E_0\circ T_{00} &= -\text{id}_{00},\label{smallklr2}\\
        T_{i0}\circ E_iX_0 &=X_0E_i\circ T_{i0},\label{smallklr3}\\
        T_{i0}\circ X_iE_0&=E_0X_i\circ T_{i0},\label{smallklr4}\\
        T_{00}^2&=0,\label{smallklr5}\\
        T_{i0}\circ T_{0i}&=Q_{i0}(E_0X_i,X_0E_i),\label{smallklr6}\\
        T_{j0}E_i\circ E_jT_{i0}\circ T_{ij}E_0&=E_0T_{ij}\circ T_{i0}E_j\circ E_iT_{j0},\label{smallklr7}\\
        T_{00}E_i\circ E_0T_{i0}\circ T_{i0}E_0&=E_0T_{i0}\circ T_{i0}E_0\circ E_iT_{00},\label{smallklr8}\\
        T_{00}E_0\circ E_0T_{00}\circ T_{00}E_0&=E_0T_{00}\circ T_{00}E_0\circ E_0T_{00}.\label{smallklr9}
    \end{align}
    
    If the $T_{i0}$, $T_{j0}E_i,E_jT_{i0},E_0T_{i0},T_{i0}E_0, E_0(X_iE_i-E_iX_i)$ and $(X_0E_0-E_0X_0)E_i$ are all monomorphisms of functors, then this data defines a 2-representation of $\U(\vec{Q}')$ on $\mathcal{C}$. In case $\mathcal{C}$ is graded and all functors and natural transformations are compatible with the gradings on $\U_q(\vec{Q}')$ and $\mathcal{C}$, then this defines a 2-representation of $\U_q(\vec{Q}')$ on $\mathcal{C}$. 
\end{lem}
In the cases we study, the $T_{i0}$ in fact come from injective bimodule homomorphisms.
\begin{proof}
    We verify that the remaining KLR relations hold. Note that the relations not involving the functor $E_0$ hold because the $E_i$, $X_i$, and $T_{ij}$ constitute a 2-representation of $\U(\vec{Q})$.
    \begin{enumerate}
        \item $T_{0i}\circ T_{i0}=Q_{0i}(E_iX_0,X_iE_0)$. Since $T_{i0}$ is monic, we may verify the equality after composing  $T_{i0}$ on the left. After doing so, the left-hand side becomes 
        \[T_{i0}\circ T_{0i}\circ T_{i0}=Q_{i0}(E_0X_i,X_0E_i)\circ T_{i0}\]
        by relation \eqref{smallklr6} above. The right-hand side is \[T_{i0}\circ Q_{0i}(E_iX_0,X_iE_0)=Q_{i0}(E_0X_i,X_0E_i)\circ T_{i0},\] as desired. 
        
        \item $T_{0i}\circ E_0X_i=X_iE_0\circ T_{0i}$. As above, we may verify equality after applying  $T_{i0}$ on the left. The left-hand side becomes 
        \[T_{i0}\circ T_{0i}\circ E_0X_i=Q_{i0}(E_0X_i,X_0E_i)\circ E_0X_i.\]
        By relations \eqref{smallklr4} and \eqref{smallklr6}, the right-hand side becomes 
        \[T_{i0}\circ X_iE_0\circ T_{0i}=E_0X_i\circ T_{i0}\circ T_{0i}=E_0X_i\circ Q_{i0}(E_0X_i,X_0E_i),\] as desired.

        \item $T_{0i}\circ X_0E_i=E_iX_0\circ T_{0i}$.
        This is essentially the same computation as above.
        \item $T_{0j}E_i\circ E_0T_{ij}\circ T_{i0}E_j-E_jT_{i0}\circ T_{ij}E_0\circ E_iT_{0j}=\delta_{i,j}\frac{Q_{i0}(X_iE_0,E_iX_0)E_i-E_iQ_{i0}(E_0X_i,X_0E_i)}{X_iE_0E_i-E_iE_0X_i}$. We may verify equality after composing $T_{j0}E_i$ on the left. The first term on the left-hand side becomes
        \[T_{j0}E_i\circ T_{0j}E_i\circ E_0T_{ij}\circ T_{i0}E_j=Q_{j0}(E_0X_j,X_0E_j)E_i\circ E_0T_{ij} \circ T_{i0}E_j\]
        by relation \eqref{smallklr6}. The second term on the left-hand side becomes
        \begin{equation*}
        \begin{split}
            &-T_{j0}E_i\circ E_jT_{i0}\circ T_{ij}E_0\circ E_iT_{0j}=\\
            &-E_0T_{ij}\circ T_{i0}E_j\circ E_iT_{j0}\circ E_iT_{0j}=\\
            &-E_0T_{ij}\circ T_{i0}E_j\circ E_iQ_{j0}(E_0X_j,X_0E_j)=\\
            &-E_0T_{ij}\circ Q_{j0}(E_0E_iX_j,X_0E_iE_j)\circ T_{i0}E_j
        \end{split}
        \end{equation*}
            by relations \eqref{smallklr7}, \eqref{smallklr6}, and \eqref{smallklr3}. If $i\neq j$, then we see that this difference is equal to zero. In this case, the right-hand side is $T_{j0}E_i\circ 0 \circ T_{0i}E_j=0$, so we have equality. Suppose instead $i=j$. The right-hand side now reads 
            \begin{equation*}
                \begin{split}
                    &T_{i0}E_i\circ \frac{Q_{i0}(X_iE_0,E_iX_0)E_i-E_iQ_{i0}(E_0X_i,X_0E_i)}{X_iE_0E_i-E_iE_0X_i}= \\
                    &\frac{Q_{i0}(E_0X_i,X_0E_i)E_i-Q_{i0}(E_0E_iX_i,X_0E_iE_i)}{E_0(X_iE_i-E_iX_i)}\circ T_{i0}E_i,
                \end{split}
            \end{equation*}
            where we have used relations \eqref{smallklr3} and \eqref{smallklr4}.
            It is therefore sufficient to verify that 
            \[Q_{i0}(E_0X_i,X_0E_i)E_i\circ E_0T_{ii}-E_0T_{ii}\circ Q_{i0}(E_0E_iX_i,X_0E_iE_i)=\frac{Q_{i0}(E_0X_i,X_0E_i)E_i-Q_{i0}(E_0E_iX_i,X_0E_iE_i)}{E_0(X_iE_i-E_iX_i)}.\]
           By assumption on monicity, it is enough to show equality after multiplying both sides on the left by $E_0(X_iE_i-E_iX_i)$. After rearranging this equation, we are reduced to showing that
           \[Q_{i0}(E_0X_iE_i,X_0E_iE_i)\circ E_0S_{ii}=E_0S_{ii} \circ Q_{i0}(E_0E_iX_i,X_0E_iE_i),\]
           where $S_{ii}=(X_iE_i-E_iX_i)\circ T_{ii}-E_iE_i$. This follows from the observation that, given the hypotheses of the lemma, $S_{ii}\circ E_iX_i=X_iE_i\circ S_{ii}$ and $S_{ii}\circ X_iE_i=E_iX_i\circ S_{ii}$. See also Proposition \ref{sirels}.
        \item $T_{ij}E_0\circ E_iT_{0j}\circ T_{0i}E_j=E_jT_{0i}\circ T_{0j}E_i\circ E_0T_{ij}$. Since $E_jT_{i0}$ is always monic, we may verify equality after composing  $E_jT_{i0}$ on the left. By the previous relation, the left-hand side is
        \begin{equation*}
            \begin{split}
                &E_jT_{i0}\circ T_{ij}E_0\circ E_iT_{0j}\circ T_{0i}E_j=\\
                &T_{0j}E_i\circ E_0T_{ij}\circ T_{i0}E_j\circ T_{0i}E_j+vT_{0i}E_j=\\
                &T_{0j}E_i\circ E_0T_{ij}\circ Q_{i0}(E_0X_i,X_0E_i)E_j+T_{0i}E_j\circ v,
            \end{split}
        \end{equation*}
        where \[v=\delta_{i,j}\frac{Q_{i0}(E_0E_iX_i,X_0E_iE_i)-Q_{i0}(E_0X_i,X_0E_i)E_i}{E_0(X_iE_i-E_iX_i)}.\] The right-hand side is now 
        \begin{equation*}
            \begin{split}
                &E_jT_{i0}\circ E_jT_{0i}\circ T_{0j}E_i\circ E_0T_{ij}=\\
                &E_jQ_{i0}(E_0X_i,X_0E_i)\circ T_{0j}E_i\circ E_0T_{ij}=\\
                &T_{0j}E_i\circ Q_{i0}(E_0E_jX_i,X_0E_jE_i)\circ E_0T_{ij}.
            \end{split}
        \end{equation*}
        A computation similar to the one in the prior part shows that these are equal.
        \item $T_{i0}E_0\circ E_iT_{00}\circ T_{0i}E_0-E_0T_{0i}\circ T_{00}E_i\circ E_0T_{i0}=\frac{Q_{0i}(X_0E_i,E_0X_i)E_0-E_0Q_{0i}(E_iX_0,X_iE_0)}{X_0E_iE_0-E_0E_iX_0}$. This can be shown similarly to parts (4) and (5).
        \item $T_{0i}E_0\circ E_0T_{0i}\circ T_{00}E_i=E_iT_{00}\circ T_{0i}E_0\circ E_0T_{0i}$. This can be shown similarly to parts (4) and (5).
    \end{enumerate}
    We have thus shown that this data yields a 2-representation of $\U(\vec{Q}')$. The claim about converting an appropriately graded 2-representation to one of $\U_q(\vec{Q}')$ is immediate from the definitions.
\end{proof}

\begin{rem}
    Note that this argument works for a more general choice of a $\kf[u,v]$-matrix $Q$ where we only require $Q_{ij}(u,v)=Q_{ji}(v,u)$. See \cite{qha} for a discussion of KLR-algebras in this generality.
\end{rem}
\begin{proof}[Proof (of Theorem \ref{thm:thebigkahuna})]
    We first verify that the $T_{i0}$ are  monomorphisms. The monicity for $T_{i0}$ when $i\neq n$ is established by Lemma \ref{lem:T0iwelldef}. For $i=n$, this lemma and the computations prior to this lemma show monicity. The remaining verifications of monic natural transformations follow from the definitions and from Theorem \ref{thm:newproj}.

   So, due to Lemma \ref{klrreduce}, we only need to check a subset of the KLR relations. Recall that the natural transformations $X_0$ and $T_{00}$ also appear in \cite{bkm} as endomorphisms of powers of $\Delta(\beta)$. Lemmas 3.7 and 3.9 of \cite{bkm} prove conditions 3.1, 3.2, 3.5, and 3.9 of Lemma \ref{klrreduce}. For the remaining relations, denote by $i$ and $j$ arbitrary elements of $I$. 
   \begin{enumerate}
       \item $T_{i0}\circ E_iX_0=X_0E_i\circ T_{i0}$. Both sides of this equality are $(H_{\alpha-\beta+\alpha_i},H_\alpha)$-linear maps defined on $H_{\alpha-\beta+\alpha_i}\otimes_{\alpha-\beta} M_\alpha$ with image in $M_{\alpha+\alpha_i}1_{*i}$. It is therefore sufficient to evaluate both sides on $1_{*i}\otimes_{\alpha-\beta} 1_{*\vec{\beta}}$. The left-hand side sends this to $ 1_{*i,\vec{\beta}}\tau_{-n}\dots \bar{\tau}_{-i}\dots \tau_{-1}x_{-2}$, and the right-hand side sends this to $ 1_{*i,\vec{\beta}}x_{-1}\tau_{-n}\dots \bar{\tau}_{-i}\dots \tau_{-1}$. Equality follows from basic computations.
       \item $T_{i0}\circ X_iE_0=E_0X_i\circ T_{i0}$. Both sides of this equality are $(H_{\alpha-\beta+\alpha_i},H_\alpha)$-linear maps defined on $H_{\alpha-\beta+\alpha_i}\otimes_{\alpha-\beta} M_\alpha$ with image in $M_{\alpha+\alpha_i}1_{*i}$. It is therefore sufficient to evaluate both sides on $1_{*i}\otimes_{\alpha-\beta} 1_{*\vec{\beta}}$. The left-hand side sends this to $1_{*1,\vec{\beta}}x_{-(n+1)}\tau_{-n}\dots \bar{\tau}_{-i}\dots \tau_{-1}$, and the right-hand side sends this to $ 1_{*1,\vec{\beta}}\tau_{-n}\dots \bar{\tau}_{-i}\dots \tau_{-1}x_{-1}$. If $i=1$ or $n$, then since $\bar{\tau}_{-i}=s_{-i}$, these are equal in  $1_{*\vec{\beta}}H_{\alpha+\alpha_i}1_{*i}$. Otherwise, their difference is the element $ 1_{*1,\vec{\beta}}\tau_{-n}\dots \tau_{-(i+1)}\tau_{-(i-1)}\dots \tau_{-1}$. This equals $1_{*1,\vec{\beta}}\tau_{-(i-1)}\dots \tau_{-1}\tau_{-n}\dots \tau_{-(i+1)}$. This difference term is zero in $M_{\alpha+\alpha_i}1_{*i}$.
       \item $T_{i0}\circ T_{0i}=Q_{i0}(E_0X_i,X_0E_i)$. This is clear if $i\neq 1,n$ since these $T_{i0}$ and $T_{0i}$ are inverse isomorphisms. Otherwise, note that both maps come from $(H_{\alpha-\beta+\alpha_i},H_{\alpha})$-linear endomorphisms of $M_{\alpha+\alpha_i}1_{*i}$. If $i=n$, then by Corollary \ref{cor:bigdecomp}, it is sufficient to evaluate on $1_{*n,\vec{\beta}}\tau_{-n}\dots \tau_{-1}$. The left-hand side sends this to $1_{*n,\vec{\beta}}s_{-n}\tau_{-(n-1)}\dots \tau_{-1}=1_{*n,\vec{\beta}}\tau_{-n}\dots \tau_{-1}x_{-1}-1_{*n,\vec{\beta}}x_{-n}\tau_{-n}\dots \tau_{-1}=(E_0X_n-X_0E_n)\circ(1_{*n,\vec{\beta}}\tau_{-n}\dots \tau_{-1})$, which is the desired equality. If $i=1$, a similar computation reveals that the image of $1_{*1,\vec{\beta}}\tau_{-n}\dots \tau_{-1}$ is $(E_0X_1-X_0E_1)\circ(1_{1,\vec{\beta}}\tau_{-n}\dots \tau_{-1})$. In this case, Corollary \ref{cor:bigdecomp} gives that we must also evaluate on elements of $M_{\alpha+\alpha_1}1_{*1}$ descended from $1_{*n,\dots 2}H_{\alpha}[x_{-1}]$. The left-hand side by definition sends such elements to zero. The right-hand side is the map $E_0X_1-X_0E_1$. Note that $E_0X_1$ is right multiplication by $x_{-1}$ and $X_0E_1$ is left multiplication by $x_{-1}$. These are the same on elements of $1_{*n,\dots 2}H_{\alpha}[x_{-1}]$, so the right-hand side also yields zero.
        \item $T_{j0}E_i\circ E_jT_{i0}\circ T_{ij}E_0=E_0T_{ij}\circ T_{i0}E_j\circ E_iT_{j0}$. Both sides of this equality are $(H_{\alpha-\beta+\alpha_i+\alpha_j},H_\alpha)$-linear maps defined on $H_{\alpha-\beta+\alpha_i+\alpha_j}1_{*j,i}\otimes_{\alpha-\beta}M_\alpha$ with image in $M_{\alpha+\alpha_i+\alpha_j}1_{*i,j}$. It is therefore sufficient to evaluate both sides on $1_{*j,i}\otimes_{\alpha-\beta}1_{*\vec{\beta}}$.  The first map sends $1_{*j,i}\otimes_{\alpha-\beta} 1_{*\vec{\beta}}$ to \[1_{*i,j,\vec{\beta}}\tau_{-(n+1)}\tau_{-n}\dots \bar{\tau}_{-i}\dots \tau_{-1}\tau_{-(n+1)}\dots \bar{\tau}_{-(j+1)}\dots \tau_{-2},\] and the second sends it to \[1_{*i,j,\vec{\beta}}\tau_{-n}\dots \bar{\tau}_{-j}\dots\tau_{-1}\tau_{-(n+1)}\dots \bar{\tau}_{-(i+1)}\dots \tau_{-2}\tau_{-1}.\]
        Using $\vec{l}=i$, $\vec{m}=j$, and $\vec{r}=\vec{\beta}$, we obtain the desired equality from Lemma \ref{rml} as long as $i,j\neq 1,n$. Otherwise, without loss of generality we have $i\in \{1,n\}$. If $i\neq j$, then $\bar{\tau}_{-i}=s_{-i}$, and the equality follows from Proposition \ref{sirels2} and Remark \ref{s_iswap}. If $i=j\in \{1,n\}$, then $\bar{\tau}_{-j}=\bar{\tau}_{-i}=s_{-i}$, and so we need only apply the $s_j$-braid relations described in Propositions \ref{sirels} and \ref{sirels2}, and so equality follows from Remark \ref{s_iswap}. 
        \item $T_{00}E_i\circ E_0T_{i0}\circ T_{i0}E_0=E_0T_{i0}\circ T_{i0}E_0\circ E_iT_{00}$. Both sides of this equality are $(H_{\alpha-2\beta+\alpha_i},H_\alpha)$-linear maps defined on $H_{\alpha-2\beta+\alpha_i}1_{*i}\otimes_{\alpha-2\beta} M_{\alpha-\beta}\otimes_{\alpha-\beta} M_\alpha$ with image in $M_{\alpha+\alpha_i-\beta}\otimes_{\alpha+\alpha_i-\beta} M_{\alpha+\alpha_i}1_{*i}$. It is therefore sufficient to evaluate both sides on $1_{*i}\otimes_{\alpha-2\beta} 1_{*\vec{\beta},\vec{\beta}}$. Denote by $\tau_{\omega_{00},-i}$ the element $\tau_{\omega_{00}}\in 1_{\vec{\beta},\vec{\beta}}H_{2\beta}$ acting on positions $-i$ through $-(i+2n+1)$. The first map sends $1_{*i}\otimes  1_{*\vec{\beta},\vec{\beta}}$ to \[-1_{*i,\vec{\beta},\vec{\beta}}\tau_{\omega_{00},-1}\tau_{-2n}\dots \bar{\tau}_{-(i+n)}\dots \tau_{-(n+1)}\tau_{-n}\dots \bar{\tau}_{-i}\dots \tau_{-1},\] and the second sends it to \[-1_{*i,\vec{\beta},\vec{\beta}}\tau_{-2n}\dots \bar{\tau}_{-(i+n)}\dots \tau_{-n-1}\tau_{-n}\dots \bar{\tau}_{-i}\dots \tau_{-1}\tau_{\omega_{00},-2}.\]
        Cancelling signs, we see both expressions correspond to the symmetric group element mapping $\vec{r}\vec{m}\vec{l}\rightarrow \vec{l}\vec{m}\vec{r}$ for $\vec{r}=\vec{m}=\vec{\beta}$ and $\vec{l}=i$, again viewing all braids as distinct. If $i=1$ or $n$, then since $\bar{\tau}_{-i}=s_{-i}$ and $\bar{\tau}_{-(i+n)}=s_{-(i+n)}$, Remark \ref{s_iswap} shows that there are no error terms, and hence, these two expressions are equal. Otherwise, the obstruction set contains only the two elements $(-(1+n+i),-(1+i)\pm 1, -1)$. Any error term resulting from the $\mathcal{S}$ element $(-(1+n+i),-(1+i)-1,-1)$ contains the following subexpression for some $k \leq i-1$.
    
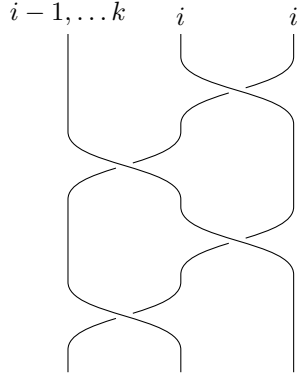
\begin{center}
    \begin{tikzpicture}
      
        \pic[braid/.cd,
        width=15mm,
        ] at (6,0)
         {braid = {s_2 s_1 s_2 s_1}};
    \draw (6,0) node[above] {$i-1,\dots k$};
    \draw (7.5,0) node[above] {$i$};
    \draw (9,0) node[above] {$i$};

    \end{tikzpicture}
\end{center}
        We observe that the braid relations hold for this element, and therefore it is zero due to the repeated $(i,i)$-swap. The argument for the other obstruction set element is dual. Note that this relation also holds in the KLR algebra and not just in the quotient.
   \end{enumerate}
\end{proof}

\subsection{Decategorification comments}\label{sec:extrarel}

 Decategorifying the 2-representation of Theorem \ref{thm:thebigkahuna}, we obtain a representation of $\C((q))\otimes U_q^+(\hat{\mathfrak{g}})$ on $\C((q))\otimes U_q^+(\mathfrak{g})$. Recall from Subsection \ref{targets} that our representation comes from one of $U_q^+(\hat{\mathfrak{g}})$ on $U_q^+(\mathfrak{g})$. In fact, since the action of $E_0$ decreases the $\N[I]$-degree by $\alpha_1+\alpha_2+\dots \alpha_n$, this action can be extended to one of $U_q(\hat{\mathfrak{b}})$. Here, $K_i(v)=K_ivK_i^{-1}$ for $i\in I$, and $K_0(v)=KvK^{-1}$ for $K=\prod_{i\in I} K_i^{-1}$.

Most of the quantum Serre relations can be deduced immediately from Corollary \ref{cor:tiso}. For $i\neq 1,n$, this corollary gives the relations $E_0E_i=E_iE_0$. It also shows $E_0E_n=q^{-1}E_nE_0$, which implies both that $E_0^2E_n-(q+q^{-1})E_0E_nE_0+E_nE_0^2=0$ and $E_n^2E_0-(q+q^{-1})E_nE_0E_n+E_0E_n^2=0$. The only quantum Serre relations which are difficult to establish are those involving both $E_1$ and $E_0$.

We might hope to find a non-quantum version of this action. We have $K_0(\He-\text{ungr})=0$, and we cannot directly specialize $(*,*)_L$ to $q=1$ due to the scalar $\frac{1}{1-q^2}$. We can again attempt to study a form on $U(\mathfrak{sl}_n)$ defined by $(1,1)=1, (E_i,E_j)=\delta_{ij}$ and extended to the rest of the algebra by the requiring it to be a bialgebra pairing. Unfortunately, this form is degenerate. For example, for $\mathfrak{sl}_3$, we have that $E_1E_2-E_2E_1$ is in the kernel of $(*,*)$. So, this form is not suitable for studying adjoints.
\section{The \texorpdfstring{$\mathfrak{sl}_2$}{sl2} case}\label{sec:sl2}
Table \ref{orientpic1} summarizes the vertex labels and orientations that we use in this section.
\begin{table}[ht]
\begin{center}
\begin{tikzcd}
    \mathfrak{sl}_2: & {} &  1 \\
   \mathfrak{sl}_3: & 0 \arrow[r] & 1 \\
   \hat{\mathfrak{sl}}_2:& 0  \arrow[r, bend left] \arrow[r, bend right]& 1 
    \end{tikzcd}
\end{center}
\caption{Dynkin diagram labels and orientations in this section}
\label{orientpic1}
\end{table}
\subsection{Preliminaries}
We briefly explain the necessary adjustments for type $\mathfrak{sl}_2$, as well as adjust notation. 

In type $\mathfrak{sl}_2$, it is easy to verify that there is a homomorphism $U_q^+(\hat{\mathfrak{sl}}_2)\rightarrow B(\mathfrak{sl}_2)$. We first compute $0=[e,\frac{1}{1-q^2}]=[e,fe-q^{-2}ef]=-q^{-2}e^2f+(1+q^{-2})efe-fe^2$. Rescaling, we have $e^2f-(1+q^{2})efe+q^{2}fe^2=0$. We note that this deforms the Serre relations for $\mathfrak{sl}_3$, and so gives some intuition for the 2-representation of Lemma \ref{sl3}. Now, if we denote $[x,y]_{q^{-2}}:=xy-q^{-2}yx$, then \[0=[e,e^2f-(1+q^{-2})efe+q^{-2}fe^2]_{q^{-2}}=e^3f-(q^2+1+q^{-2})e^2fe+(q^2+1+q^{-2})efe^2-fe^3,\]
which is precisely the quantum Serre relation for $\hat{\mathfrak{sl}}_2$ up to sign. 

Let $\vec{Q}$ be the Dynkin diagram for $\mathfrak{sl}_2$, i.e. a single vertex with no edges and thus a trivial orientation. As in the previous section, since we have fixed orientations in Table \ref{orientpic1}, we may omit $\vec{Q}$ in our algebras and categories. Note that $H_{n1}= {^0H}_n$ is the $n$-th nil affine Hecke algebra. There is thus only one idempotent generator in any $H_{\alpha}$, so we omit the $1_{v}$. Also, $\N[I]=\N$, so we replace $\alpha$ by $|\alpha|$ whenever it appears. 

Since the lowest root of $\mathfrak{sl}_2$ is just $F_1$, we will define our affine functor $E_0$ on $\He_n$ as 
\[E_{0,n}:=E_1^*=H_n\otimes_n:\He_n\rightarrow \He_{n-1}.\]
We define our natural transformations similarly to before. Firstly, note that $E_1E_1= H_{n+2}\otimes_n$, and $E_0E_0=H_{n+2}\otimes_{n+2}$.
\[X_1:E_1\rightarrow E_1, a\rightarrow ax_{n+1},\]
\[T_{11}: E_1E_1\rightarrow E_1E_1, a \rightarrow a\tau_{n+1},\]
\[X_0:E_0\rightarrow E_0, a\rightarrow x_{n+1}a,\]
\[T_{00}:E_0E_0\rightarrow E_0E_0, a\rightarrow -\tau_{n+1}a.\]

For the remaining natural transformations, we recall from Proposition \ref{prop:kk} that there is a degree $-2$ embedding of $(H_{n},H_n)$-bimodules $H_n\otimes_{n-1} H_n \hookrightarrow H_{n+1}$ via $a\otimes_{n-1} b\rightarrow a\tau_n b$. This embedding splits over $(H_n,H_n)$-bimodules
\begin{equation}\label{sl2split}
    H_{n+1}\simeq H_n\otimes_{n-1} H_n \oplus H_n[x_{n+1}].
\end{equation}
 Now, $E_1E_0=(H_n\otimes_{n-1}H_n)\otimes_n$, and $E_0E_1=H_{n+1}\otimes_n$, both viewed as endofunctors of $\mathcal{H}_n$.

We can therefore define the last two natural transformations
\[T_{01}:E_0E_1\rightarrow E_1E_0, H_{n+1}\simeq H_n\otimes_{n-1} H_n\oplus H_n[x_{n+1}]\twoheadrightarrow H_n\otimes_{n-1} H_n,\]
\[T_{10}:E_1E_0\rightarrow E_0E_1, a\otimes_{n-1} b\rightarrow a(x_{n}-x_{n+1})s_nb.\]
The factor of $(x_{n}-x_{n+1})$ in the definition of $T_{10}$ is the additional modification that is needed for $\mathfrak{sl}_2$ but not for other types. This is due to the fact that the Dynkin diagram for $\mathfrak{sl}_2$ is the unique simply laced finite-type Dynkin diagram whose affine counterpart has more than one edge between any pair of vertices.

The main result of this section is the following theorem. 
\begin{thm}\label{sl2rep}
    The functors $E_i$ and natural transformations $X_i$, $T_{ij}$ for $i,j\in \{0,1\}$ define a 2-representation of $U_q^+(\hat{\mathfrak{sl}}_2)$ on $\He$.
\end{thm}
We prove this theorem in two steps. Firstly, we will show that there is a simpler related 2-representation of $U^+(\mathfrak{sl}_3)$ on $\He-\text{ungr}$. We then show that we can modify the natural transformation $T_{10}$ in a precise way to turn this into the desired 2-representation. While the KLR algebra relations could be checked directly from the definitions above, the proof we use will be more general and will better elucidate why the KLR algebra relations are satisfied.
\subsection{2-representation of \texorpdfstring{$U^+(\mathfrak{sl}_3)$}{U+(sl 3)}}
In this subsection, we construct an intermediate 2-representation that is easier to verify. We will only modify a single natural transformation. Define \[T'_{10}:E_1E_0\rightarrow E_0E_1, a\otimes_{n-1} b \rightarrow as_nb.\]

\begin{lem}\label{sl3}
The functors $E_i$, along with natural transformations $X_0,X_1, T_{00},T_{11},T_{01},T'_{10}$ define a 2-representation of $U^+(\mathfrak{sl}_3)$ on $\mathcal{H}-\text{ungr}$.
\end{lem}
This cannot be quantized to a 2-representation of $U_q^+(\mathfrak{sl}_3)$ since both $T'_{10}$ and $T_{01}$ have the incorrect grading (0 and 2 instead of 1 and 1). 

\begin{proof}
    We show that $T'_{10}$ is monic. Following \cite{kk}, we note that $H_n=\bigoplus_{i=1}^n \tau_i\dots \tau_{n-1}\kf [x_n]\otimes_\kf H_{n-1}$, and therefore, \[H_n\otimes_{n-1}H_n \simeq \bigoplus_{i=1}^n \tau_i\dots \tau_{n-1} \kf [x_n]\otimes_{\kf} H_n.\]
    
    Therefore, the set $\{\tau_i\dots \tau_{n-1} x_n^j\otimes_{n-1} 1\}_{i\leq n, j\in \N}$ is a basis for $H_n\otimes_{n-1} H_n$ as a right $H_n$-module. Since $a\otimes_{n-1} b\rightarrow a\tau_n b$ is injective, the image of this set, the $\{\tau_i\dots \tau_{n-1}x_n^j \tau_n\}$ is linearly independent for the right $H_n$-action. We compute that \[T'_{10}(\tau_i\dots \tau_{n-1}x_n^j\otimes_{n-1} 1)=\tau_i\dots \tau_{n-1}x_n^j s_n=\]
    \[\tau_i\dots \tau_{n-1}x_n^{j+1}\tau_n-\tau_i\dots \tau_{n-1}x_n^j \tau_n x_n.\]
    
    It is easy to see that these elements are also linearly independent. So, $T'_{10}$ is monic. It follows from here that the remaining hypotheses of Lemma \ref{klrreduce} are satisfied. It is therefore sufficient to check the KLR relations enumerated in Lemma \ref{klrreduce}. Relations \eqref{smallklr1}, \eqref{smallklr2}, \eqref{smallklr5}, and \eqref{smallklr9} are immediate from the relations of the KLR algebra.
    \begin{enumerate}
        \item $T'_{10}\circ E_1X_0=X_0E_1\circ T'_{10}$. Both sides of this equality are $(H_n,H_n)$-bimodule maps defined on $H_n\otimes_{n-1}H_n$. So, it is sufficient to verify equality on $1\otimes_{n-1} 1$. We compute that $T'_{10}\circ E_1X_0 (1\otimes_{n-1}1)=T'_{10}\circ (1\otimes_{n-1} x_n)=s_nx_n$, and $X_0E_1\circ T'_{10}(1\otimes_{n-1} 1)=X_0E_1(s_n)=x_{n+1}s_n$. Due to Proposition \ref{sirels}, these are equal.
        \item $T'_{10}\circ X_iE_0=E_0X_i\circ T'_{10}$. Similar to above.
        \item $T'_{10}\circ T_{01}=Q_{10}(E_0X_1,X_0E_1)=E_0X_1-X_0E_1$. Both sides are $(H_n,H_n)$-bimodule maps defined on $H_{n+1}$. Due to the splitting in Equation \eqref{sl2split}, it is sufficient to verify equality after evaluating on $\tau_n$, which is the image under embedding of $1\otimes_{n-1} 1\in H_n\otimes_{n-1} H_n$, and on powers of $x_{n+1}$. We first compute that $T'_{10}\circ T_{01}(\tau_n)=T'_{10}(1\otimes_{n-1} 1)=s_n$. The right-hand side is $(E_0X_1-X_0E_1)(\tau_n)=\tau_nx_{n+1}-x_{n+1}\tau_n$. By Proposition \ref{sirels}, these are equal. We also compute that $T'_{10}\circ T_{01}(x_{n+1}^i)=T'_{10}(0)=0$, and $(E_0X_1-X_0E_1)(x_{n+1}^i)=x_{n+1}^{i+1}-x_{n+1}^{i+1}=0$.
        \item $T'_{10}E_1\circ E_1T'_{10}\circ T_{11}E_0=E_0T_{11}\circ T'_{10}E_1\circ E_1T'_{10}$. Both sides of this equality are $(H_{n+1},H_n)$-bimodule maps defined on $H_{n+1}\otimes_{n-1}H_n$. So, it is sufficient to verify equality on $1\otimes_{n-1} 1$. We compute that $T'_{10}E_1\circ E_1T'_{10}\circ T_{11}E_0 (1\otimes_{n-1} 1)=T'_{10}E_1\circ E_1T'_{10}(\tau_n\otimes_{n-1} 1)=T'_{10}E_1(\tau_n\otimes_n s_n)=\tau_ns_{n+1}s_n$. We also compute that $E_0T_{11}\circ T'_{10}E_1\circ E_1T'_{10}(1\otimes_{n-1}1)=E_0T_{11}\circ T'_{10}E_1(1\otimes_n s_n)=E_0T_{11}(s_{n+1}s_n)=s_{n+1}s_n\tau_{n+1}$. By Proposition \ref{sirels}, these are equal. 
        \item $T_{00}E_1\circ E_0T'_{10}\circ T'_{10}E_0=E_0T'_{10}\circ T'_{10}E_0\circ E_1T_{00}$. Similar to above.
    \end{enumerate}
\end{proof}
\begin{rem}\label{rem:sl2dequant}
    Note that $K_0(\He-\text{ungr})=0$. There is, however, a way of interpreting this action of $U(\mathfrak{sl}_3)$ on $U(\mathfrak{sl}_2)$. The form $(*,*)_L$ on $U_q(\mathfrak{sl}_2)$ cannot be specialized to $q=1$ directly due to the equation $(E_1,E_1)=\frac{1}{1-q^2}$. There is, however, a nondegenerate bilinear form on $U(\mathfrak{sl}_2)$ given by $(E_1^n,E_1^n)=n!$. The operation of right multiplication by $E_1$ has an adjoint $E_1^*$. We have that $E_1^*E_1-E_1E_1^*=1$, and therefore, $E_1$ and $E_1^*$ satisfy the Serre relations for $\mathfrak{sl}_3$. 
\end{rem}
\subsection{Virtually surjective functors}
We construct a functor between different categorified enveloping algebras that is surjective on Grothendieck groups. This functor will allow us to pullback the 2-representation obtained in the previous subsection to complete the proof of Theorem \ref{sl2rep}.

Let $C$ be any generalized Cartan matrix, and fix distinct $i,j \in I$. Let $C'$ be the generalized Cartan matrix obtained by subtracting 1 from both $C_{ij}$ and $C_{ji}$. There is a surjective map of algebras $U^+(C')\rightarrow U^+(C)$ sending each $E_i$ to $E_i$. This can be seen by noting that the classical Serre relations can be viewed as an iterated application of a linear operator;
\[\sum_{n=0}^{1-C_{ij}}\binom{1-C_{ij}}{n}(-1)^nE_i^nE_jE_i^{1-C_{ij}-n}=ad_{E_i}^n(E_j),\]
for $ad_{E_i}(E_j)=E_jE_i-E_iE_j$ the standard Lie bracket. Clearly, if $ad^n_{E_i}(E_j)=0$, then $ad^{n+1}_{E_i}(E_j)=0$ as well. We produce a categorical version of this surjection.

\begin{thm}\label{lacing}
    Let $C$ be any generalized Cartan matrix, and fix distinct $i,j \in I$. Let $C'$ be the generalized Cartan matrix obtained by subtracting 1 from both $C_{ij}$ and $C_{ji}$. Let $\vec{Q}$ be any orientation of the Dynkin diagram for $C$, and let $\vec{Q}'$ be the orientation of the Dynkin diagram of $C'$ obtained by adding to $\vec{Q}$ a unique directed edge from $i$ to $j$. Then there is an additive monoidal functor $P:\U(\vec{Q}')\rightarrow \U(\vec{Q})$ surjective on $ K_0$ given by the following formulae.
    \begin{align*}
         &P(E_i)=E_i,\\
        &P(X_i)=X_i,\\
        &P(T_{ab})=T_{ab} \text{ if }\{a,b\}\neq \{i,j\},\\
        &P(T_{ij})=T_{ij},\\
        &P(T_{ji})=T_{ji}\circ (X_jE_i-E_jX_i)=(E_iX_j-X_iE_j)\circ T_{ji}.
    \end{align*}
\end{thm}
\begin{proof}

     We check that the defining relations of $\U(\vec{Q}')$ are satisfied. Any KLR relation not involving $T_{ji}$ is satisfied, since these relations are identical in $\U(\vec{Q}')$ and $\U(\vec{Q})$. We denote by $Q'_{ij}(u,v)$ the polynomial matrix associated to $\vec{Q}'$ in the definition of $\U(\vec{Q})$. By our choice of orientation, $Q_{ij}(u,v)*(v-u)=Q'_{ij}(u,v)$, and $Q_{ji}(u,v)*(u-v)=Q'_{ji}(u,v)$.
    \begin{enumerate}
        \item $P(T_{ji})\circ X_jE_i=E_iX_j\circ P(T_{ji})$. We compute that $P(T_{ji})\circ X_jE_i=T_{ji}\circ (X_jE_i-E_jX_i)\circ X_jE_i=T_{ji}\circ X_jE_i\circ (X_jE_i-E_jX_i)=E_iX_j\circ T_{ji}\circ (X_jE_i-E_jX_i)=E_iX_j\circ P(T_{ji})$.
        \item $P(T_{ji})\circ E_jX_i=X_iE_j\circ P(T_{ji})$. Similar to above.
        \item $T_{ij}\circ P(T_{ji})=Q'_{ij}(E_jX_i,X_jE_i)$. We evaluate the left-hand side as 
        \begin{equation*}
            \begin{split}
                &T_{ij}\circ P(T_{ji})=\\
                &T_{ij}\circ T_{ji}\circ (X_jE_i-E_jX_i)=\\
                &Q_{ij}(E_jX_i,X_jE_i)\circ (X_jE_i-E_jX_i)=\\
                &Q'_{ij}(E_jX_i,X_jE_i),
            \end{split}
        \end{equation*}
        which is exactly the right-hand side.
        \item $P(T_{ji})\circ T_{ij}=Q'_{ji}(E_iX_j,X_iE_j)$. Similar to above.
        \item $T_{ik}E_j\circ E_iP(T_{jk})\circ P(T_{ji})E_k-E_kP(T_{ji})\circ P(T_{jk})E_i\circ E_jT_{ik}=\delta_{j,k}\frac{Q'_{ji}(X_jE_i,E_jX_i)E_j-E_jQ'_{ji}(E_iX_j,X_iE_j)}{X_jE_iE_j-E_jE_iX_j}$. We consider three cases. Firstly, if $k\neq i,j$, then the right-hand side is zero. The first term on the left-hand side becomes 
        \begin{equation*}
            \begin{split}
                &T_{ik}E_j\circ E_iT_{jk}\circ (E_iX_j-X_iE_j)E_k\circ T_{ji}E_k=\\
                &T_{ik}E_j\circ (E_iE_kX_j-X_iE_kE_j)\circ E_iT_{jk}\circ T_{ji}E_k= \\
                &E_k(E_iX_j-X_iE_j)\circ T_{ik}E_j\circ E_iT_{jk}\circ T_{ji}E_k.
            \end{split}
        \end{equation*}
        The second term on the right hand side is 
        \begin{equation*}
            \begin{split}
                &-E_k(E_iX_j-X_iE_j)\circ E_kT_{ji}\circ T_{jk}E_i\circ E_jT_{ik}.
            \end{split}
        \end{equation*}
        By the braid relations in $\U(\vec{Q})$, this difference is zero.

        Now suppose $k=i$. Then the first term on the left-hand side becomes
        \begin{equation*}
            \begin{split}
                &T_{ii}E_j\circ E_i(E_iX_j-X_iE_j)\circ E_iT_{ji}\circ (E_iX_j-X_iE_j)E_i\circ T_{ji}E_i=\\
                &T_{ii}E_j\circ E_i(E_iX_j-X_iE_j)\circ (E_iE_iX_j-X_iE_iE_j)\circ E_iT_{ji}\circ T_{ji}E_i=\\
                & T_{ii}E_j\circ (E_iE_iX_j^2-E_iX_iX_j-X_iE_iX_j+X_iX_iE_j)\circ E_iT_{ji}\circ T_{ji}E_i=\\
                &(E_iE_iX_j^2-E_iX_iX_j-X_iE_iX_j+X_iX_iE_j)\circ T_{ii}E_j\circ E_iT_{ji}\circ T_{ji}E_i,
            \end{split}
        \end{equation*}
        where in the last equality we have used Proposition \ref{symcommute}. The second term on the left-hand side becomes
        \begin{equation*}
        \begin{split}
            &-E_i(E_iX_j-X_iE_j)\circ E_iT_{ji}\circ (E_iX_j-X_iE_j)E_i\circ T_{ji}E_i\circ E_jT_{ii}=\\
            &-E_i(E_iX_j-X_iE_j)\circ (E_iE_iX_j-X_iE_iE_j)\circ E_iT_{ji}\circ T_{ji}E_i\circ E_jT_{ii}.
        \end{split}
        \end{equation*}
        By the braid relations in $\U(\vec{Q})$, this difference is zero. The right-hand side is also zero in this case.

        Finally, suppose $k=j$. Then the first term on the left-hand side becomes
        \begin{equation*}
            \begin{split}
                &T_{ij}E_j\circ E_iT_{jj}\circ (E_iX_j-X_iE_j)E_j\circ T_{ji}E_j=\\
                &T_{ij}E_j\circ (E_iE_jX_j-X_iE_jE_j)\circ E_iT_{jj}\circ T_{ji}E_j+T_{ij}E_j\circ T_{ji}E_j=\\
                &E_j(E_iX_j-X_iE_j)\circ T_{ij}E_j\circ E_iT_{jj}\circ T_{ji}E_j + Q_{ij}(E_jX_i,X_jE_i)E_j.
            \end{split}
        \end{equation*}
        The second term on the left-hand side is
        \[
                -E_j(E_iX_j-X_iE_j)\circ E_jT_{ji}\circ T_{jj}E_i\circ E_jT_{ij}.
        \]
        The left-hand side therefore evaluates to 
        \begin{equation*}
            \begin{split}
                &E_j(E_iX_j-X_iE_j)\circ \frac{Q_{ji}(X_jE_i,E_jX_i)E_j-E_jQ_{ji}(E_iX_j,X_iE_j)}{X_jE_iE_j-E_jE_iX_j}+Q_{ij}(E_jX_i,X_jE_i)E_j.
            \end{split}
        \end{equation*}
        Since we have that $(E_iX_j-X_iE_j)Q_{ji}(E_iX_j,X_iE_j)=Q'_{ji}(E_iX_j,X_iE_j)$, we need only verify the following equality
        \[E_j(E_iX_j-X_iE_j)\circ Q_{ji}(X_jE_i,E_jX_i)E_j+(X_jE_iE_j-E_jE_iX_j)\circ Q_{ij}(E_jX_i,X_jE_i)E_j=Q'_{ji}(X_jE_i,E_jX_i)E_j.\]
        We note now that $Q_{ij}(u,v)=Q_{ji}(v,u)$ by definition. So, we compute
        \begin{equation*}
            \begin{split}
                &E_j(E_iX_j-X_iE_j)\circ Q_{ji}(X_jE_i,E_jX_i)E_j+(X_jE_iE_j-E_jE_iX_j)\circ Q_{ij}(E_jX_i,X_jE_i)E_j=\\
                &E_j(E_iX_j-X_iE_j)\circ Q_{ji}(X_jE_i,E_jX_i)E_j+(X_jE_iE_j-E_jE_iX_j)\circ Q_{ji}(X_jE_i,E_jX_i)E_j=\\
                &(X_jE_i-E_jX_i)E_j\circ Q_{ji}(X_jE_i,E_jX_i)E_j=\\
                &Q'_{ji}(X_jE_i,E_jX_i)E_j,
            \end{split}
        \end{equation*}
        as desired. So, in all cases, we have equality.
        \item $P(T_{ji})E_k\circ E_jP(T_{ki})\circ T_{kj}E_i-E_iT_{kj}\circ P(T_{ki})E_j\circ E_kP(T_{ji})=\delta_{k,i}\frac{Q'_{ij}(X_iE_j,E_iX_j)E_i-E_iQ'_{ij}(E_jX_i,X_jE_i)}{X_iE_jE_i-E_iE_jX_i}$. Similar to above.
        \item $P(T_{ki})E_j\circ E_kP(T_{ji})\circ P(T_{jk})E_i-E_iP(T_{jk})\circ P(T_{ji})E_k\circ E_jP(T_{ki})=0$. If $k\neq i,j$, then this computation is similar to those above. If $k=i$, then this is covered by case 5, and if $k=j$, then this is covered by case 6.

    \end{enumerate}
    To see that the functor is surjective on $K_0$, we note that $K_0$ is spanned by the ``divided power" projectives \cite{khla}, and this functor evidently preserves the corresponding idempotents.
\end{proof}
It would be interesting to see whether the functor is itself surjective, although we do not need this.

We now have what we need to prove Theorem \ref{sl2rep}.
\begin{proof}[Proof (of  Theorem \ref{sl2rep})] The 2-representation of $\U(\mathfrak{sl}_3)$ in Lemma \ref{sl3} uses the orientation $0 \rightarrow 1$. Observe that $T'_{10}$ is related to $T_{10}$ by the formula $T_{10}=T'_{10}\circ (X_1E_0-E_1X_0)$. So, Theorem \ref{lacing} shows that the given data comes from a 2-representation of $U(\hat{\mathfrak{sl}}_2)$, where the orientation of the Dynkin diagram for $\hat{\mathfrak{sl}}_2$ has both edges pointing from $0$ to $1$. Finally, we observe that all of the natural transformations are compatible with the grading on $\mathcal{H}$ and have the correct grading for $\U(\hat{\mathfrak{sl}}_2)$; the $X_i$ have grading 2, the $T_{ii}$ have grading $-2$, and $T_{10}$ and $T_{01}$ both also have grading 2. Therefore, we can enrich to a 2-representation of $\U_q(\hat{\mathfrak{sl}}_2)$.
\end{proof}
\subsection{Surjective algebra morphisms}
We use the virtually surjective functors of the last subsection to produce new surjections of quantum enveloping algebras.

Let $C$ and $C'$ be the generalized Cartan matrices of Theorem \ref{lacing}. This theorem categorifies the natural surjection of algebras $U^+(\mathfrak{g}(C'))\twoheadrightarrow U^+(\mathfrak{g}(C))$. Moreover, in our specific application we used this to obtain a 2-representation not just of $U^+(\mathfrak{g}(C'))$ but of $U_q^+(\mathfrak{g}(C'))$. It is natural then to ask whether Theorem \ref{lacing} can be easily quantized. Unfortunately, the corresponding map of algebras does not exist in the quantum case. 

For $x,y$ elements of any associative $\C(q)$-algebra $A$, denote by $S_{q,n}(x,y)$ the $n$th \emph{quantum Serre operator} used in the definition of $U_q^+(\mathfrak{g})$:
\[S_{q,n}(x,y):=\sum_{i=0}^{n}\binom{n}{i}_q (-1)^{i}x^iyx^{n-i}.\]
Unlike in the classical case, the quantum Serre operators cannot be viewed as iterated application of some fixed linear operator, so the classical argument fails. To see that the surjection itself does not exist, consider the most basic example of $\mathfrak{g}(C)=\mathfrak{sl}_2\times \mathfrak{sl}_2$, and $\mathfrak{g}(C')=\mathfrak{sl}_3$. The generators $E_1$ and $E_2$ of $U_q^+(\mathfrak{g}(C))$ satisfy $S_{1,q}(E_1,E_2)=E_2E_1-E_1E_2=0$. The corresponding generators $E'_1$ and $E'_2$ of $U_q^+(\mathfrak{g}(C'))$ satisfy $S_{2,q}(E'_1,E'_2)={E'_1}^2E'_2-(q+q^{-1})E'_1E'_2E'_1+E'_2{E_1'}^2=0$. If this surjective map existed, then in $U_q(\mathfrak{g}(C))$, we must have that $0=E_1^2E_2-(q+q^{-1})E_1E_2E_1+E_2E_1^2=(2-q-q^{-1})E_1^2E_2$. This is false.

Fortunately, a closer look at Theorem \ref{lacing} shows us what the problem is. The grading of $P(T_{ji})$ is now 1 too high, and the grading of $P(T_{ij})$ is 1 too low. But, if we compose two of these functors, the gradings work out precisely. Specifically, we have the following results.
\begin{cor}
    Let $C$ be any generalized Cartan matrix, and fix distinct $i,j \in I$. Let $C''$ be the generalized Cartan matrix obtained by subtracting 2 from both $C_{ij}$ and $C_{ji}$. Let $\vec{Q}$ be any orientation of the Dynkin diagram for $C$, and let $\vec{Q}''$ be the orientation of the Dynkin diagram of $C$ obtained by adding a directed edge from $i$ to $j$ and another from $j$ to $i$ in $\vec{Q}$. Then there is an additive monoidal graded functor $P:\U_q(\vec{Q}'')\rightarrow \U_q(\vec{Q})$ surjective on $K_0$ given by the following formulae.
    \begin{enumerate}
        \item $P(E_i)=E_i$,
        \item $P(X_i)=X_i$,
        \item $P(T_{ab})=T_{ab}$ if $\{a,b\}\neq \{i,j\}$.
        \item $P(T_{ij})=T_{ij}\circ (X_iE_j-E_iX_j)=(E_jX_i-X_jE_i)\circ T_{ij}$,
        \item $P(T_{ji})=T_{ji}\circ (X_jE_i-E_jX_i)=(E_iX_j-X_iE_j)\circ T_{ji}$.
    \end{enumerate}

\end{cor}
Taking Grothendieck groups of both categories gives the following.
\begin{thm}\label{surject}
    For $C$ and $C''$ as in the prior corollary, there is a surjection of $\C(q)$-algebras $U_q^+(\mathfrak{g}(C''))\twoheadrightarrow U_q^+(\mathfrak{g}(C))$ given by $E_i\rightarrow E_i$.
\end{thm}
As simple as Theorem \ref{surject} may seem, to our knowledge, it is new. We can verify it directly in the simplest case where $\mathfrak{g}(C)=\mathfrak{sl}_2\times \mathfrak{sl}_2$ and $\mathfrak{g}(C'')=\hat{\mathfrak{sl}}_2$. We see that indeed $S_{q,3}(E_1,E_2)=0$ in $U_q^+(\mathfrak{sl}_2\times \mathfrak{sl}_2)$;
\begin{equation*}
    \begin{split}
        &-E_1^3E_2+(q^2+1+q^{-2})E_1^2E_2E_1-(q^2+1+q^{-2})E_1E_2E_1^2+E_2E_1^3=\\
        &-E_1^3E_2+(q^2+1+q^{-2})E_1^3E_1-(q^2+1+q^{-2})E_1^3E_2+E_1^3E_2=0.
    \end{split}
\end{equation*}

Now that we have seen the map on the level of algebras, a simpler, non-categorical proof can be provided.
\begin{proof}[Proof (of Theorem \ref{surject})]
We prove more generally that if $S_{q,n}(x,y)=0$, then $S_{q,n+2}(x,y)=0$ for $x,y$ in any fixed associative $\C(q)$ algebra $A$. Fix such $x$ and $y$. Let $\C(q)[L,R]$ be the polynomial ring over $\C(q)$ in two commuting variables, $L$ and $R$. It acts on $A$ via $L*v=xv$ and $R*v=vx$. So, 
\[(\sum_{i=0}^n \binom{n}{i}_q(-1)^nL^iR^{n-i})*y=S_{q,n}(x,y).\]
We factorize the polynomial appearing on the left-hand side of this equation. The factorization below is a version of the quantum binomial theorem.
\[\sum_{i=0}^n \binom{n}{i}_q(-1)^iL^iR^{n-i}=\prod_{i=0}^{n-1}(R-q^{2i-(n-1)}L).\]
If there is $y\in A$ for which $\prod_{i=0}^{n-1}(R-q^{2i-(n-1)}L)$ acts by zero on $y$, i.e. $S_{q,n}(x,y)=0$, then so does \[(R-q^{n+1}L)(\prod_{i=0}^{n-1}(R-q^{2i-(n-1)}L))(R-q^{-(n+1)}L)=\prod_{i=0}^{n+1}(R-q^{2i-(n+1)}L).\]
The right-hand acts by $S_{q,n+2}(x,-)$.
\end{proof}
It is also apparent from the proof above why the claim fails if we replace $S_{q,n+2}(x,y)$ by $S_{q,n+1}(x,y);$ the polynomial element of $\C(q)[L,R]$ for $S_{q,n}(x,-)$ does not divide that of $S_{q,n+1}(x,-)$.

\section{Other types}\label{sec:d4}
To define the image of $E_0$ in Jimbo's evaluation homomorphism $U_q(\hat{\mathfrak{sl}}_{n-1})\rightarrow U_q(\mathfrak{gl}_n)$, one first picks a presentation of a lowest weight root vector of $\mathfrak{sl}_n$ of the form $[F_{i_n},[F_{i_{n-1}},\dots [F_{i_2},F_{i_1}]]]$. Moreover, this homomorphism is only defined for the natural linear orderings on the simple roots. We will reformulate our $q$-boson representation to lift this restriction. This same reformulation will also allow us to work outside of type $A_n$.

The following proposition can easily be proven inductively.
\begin{prop}
    Let $E_i^*$ be the adjoint operator under the form $(*,*)_L$ on $U_q^+(\mathfrak{sl}_{n+1})$ for right multiplication by $E_i$. Then 
    \begin{equation*}
        [E_n^*,[E_{n-1}^*,\dots [E_2^*,E_1^*]_q]_q]_q = [[[E_n^*,E_{n-1}^*]_q\dots ,E_2^*]_q,E_1^*]_q.
    \end{equation*}
\end{prop}
So, one may instead think to write our lowest weight root vector in terms of left-nested brackets. The already-given actions of $E_0$ can be expressed in this way due to Dynkin diagram automorphisms of $\mathfrak{sl}_{n+1}$. However, only this left-nested approach works for other presentations of the lowest weight root vector.
\begin{prop}
  Let $E_i^*$ be the adjoint operator under $(*,*)_L$ on $U_q^+(\mathfrak{sl}_{4})$ for right multiplication by $E_i$. Then $E_0:=[[E_2^*,E_3^*]_q,E_1^*]_q$ extends the right multiplication representation to give an action of $U_q^+(\hat{\mathfrak{sl}}_4)$, but $E_0':=[E_1^*,[E_3^*,E_2^*]_q]_q$ does not.
\end{prop}
\begin{proof}
    We check directly the easiest of the relevant quantum Serre relations, which is that the operator must commute with right multiplication by $E_2$. We apply the $q$-boson relations to compute
    \begin{align*}
        [E_0,E_2] &= [E_2^*E_1^*E_3^*-qE_1^*E_2^*E_3^*-qE_3^*E_2^*E_1^*+q^2E_3^*E_1^*E_2^*,E_2]\\
        &=\frac{1}{1-q^2}(q^2E_1^*E_3^*-q^2E_1^*E_3^*-q^2E_3^*E_1^*+q^2E_3^*E_1^*) =0.
    \end{align*}
    We also see that 
    \begin{align*}
        [E_0',E_2] &= [E_3^*E_1^*E_2^*-qE_3^*E_2^*E_1^*-qE_1^*E_2^*E_3^*+q^2E_2^*E_1^*E_3^*,E_2] \\
        &=\frac{1}{1-q^2}(E_3^*E_1^*-q^2E_3^*E_1^*-q^2E_1^*E_3^*+q^4E_1^*E_3^*) \\
        &= \frac{1-2q^2+q^4}{1-q^2}E_3^*E_1^* \neq 0.
    \end{align*}
    The remaining relations for $E_0$ can be verified in SageMath via the QuaGroup package for GAP \cite{sagemath, QuaGroup1.8.4}. By nondegeneracy of our bilinear form, one needs only remove the $*$'s from the relevant equations and check the corresponding $U_q^+(\mathfrak{sl}_4)$ elements are zero.
\end{proof}
One can similarly check that left-nested expressions do \emph{not} give new evaluation homomorphisms of the expected form from $U_q(\hat{\mathfrak{sl}}_4)$ to $U_q(\mathfrak{gl}_{4})$. We have also verified that $[[[E_2^*,E_3^*]_q,E_1^*]_q,E_4^*]_q$ gives a valid action of $E_0$ on $U_q^+(\mathfrak{sl}_5)$.

The importance of left-nested brackets can be seen categorically. We observed in subsection \ref{subsec:control} that the rightmost vertex in the idempotent associated to the affine bimodule $M$ needed to be an ``extended" vertex, i.e. either $1$ or $n$. This is not possible for a general right-nested expression, but is automatic for a left-nested expression.

The same arguments of Section \ref{sec:general} give that in type $\mathfrak{sl}_4$, the $(H_{\alpha-\beta},H_\alpha)$-bimodule \[1_{*231}H_\alpha/1_{*231}(\tau_{-2},\tau_{-1}\tau_{-2})H_\alpha\] is left-projective and has action on the Grothendieck group equal to $[[E_2^*,E_3^*]_q,E_1^*]_q$. A full categorification in type $A_n$ for $E_0$'s defined via left-nested expressions is work in progress.

This approach also gives us a valid representation in type $\mathfrak{so}_8$, i.e. for the Dynkin diagram of type $D_4$.
\begin{prop}
    Let $E_i^*$ be the adjoint operator under the form $(*,*)_L$ on $U_q^+(\mathfrak{so}_{8})$ for right multiplication by $E_i$. Label the vertices of the Dynkin diagram $1$ through $4$ with $2$ labelling the central vertex. Then \[E_0:=[[[[E_1^*,E_2^*]_q,E_3^*]_q,E_4^*]_q,E_2^*]_q\]
    extends the right multiplication representation to give an action of $U_q^+(\hat{\mathfrak{so}}_8)$.
\end{prop}
\begin{proof}
    Commutation with right multiplication by $E_1$, $E_3$, or $E_4$ is again easy to verify by hand. We compute
    \begin{align*}
        \epsilon:=[E_0,E_2]_{q^{-1}} &=E_1^*E_2^*E_3^*E_4^*-qE_2^*E_1^*E_3^*E_4^*-qE_3^*E_1^*E_2^*E_4^*-qE_4^*E_1^*E_2^*E_3^*\\
        &+q^2E_1^*E_3^*E_4^*E_2^*+q^2E_3^*E_2^*E_1^*E_4^*+q^2E_4^*E_2^*E_1^*E_3^*-q^3E_4^*E_3^*E_2^*E_1^*.
        \end{align*}
Verifying the quantum Serre relation $E_2^*E_0-(q+q^{-1})E_2E_0E_2+E_0E_2^2$ amounts to verifying that $\epsilon E_2-qE_2\epsilon=0$. We check that 
\begin{align*}
    \epsilon E_2-qE_2 \epsilon &= q^2 E_1^*E_3^*E_4^*-q^4 E_1^*E_3^*E_4^*-q^2E_3^*E_1^*E_4^*-q^2E_4^*E_1^*E_3^* \\
    &+ q^2E_1^*E_3^*E_4^* +q^4 E_3^*E_1^*E_4^*+ q^4 E_4^*E_1^*E_3^*-q^4 E_4^*E_3^*E_1^*\\
    &=(2q^2-2q^2+2q^4-2q^4)E_1^*E_3^*E_4^* =0.
\end{align*}
The remaining quantum Serre relation is a verification that $[E_0,\epsilon]_q=0$. This is accomplished via SageMath.
\end{proof}
It is not difficult to check that the corresponding right-nested endomorphism does not yield an action of $E_0\in U_q^+(\hat{\mathfrak{so}}_8)$. Up to Dynkin diagram automorphism, there are only two left-nested expressions for the lowest weight root vector of $\mathfrak{so}_8$. One can verify that the remaining presentation, $[[[[F_2,F_1],F_3],F_4],F_2]$, does \emph{not} yield an action of $E_0$. This suggests either that there is a more precise dependence on the positive root poset for $\mathfrak{g}$ or that this is also not the ``correct" formulation of the action of $E_0$. One possible approach is to give the $E_0$ action in terms of Lusztig's braid group action as in \cite{unipotent} and \cite{unipotent2}. This approach would also make sense categorically due to the appearance of a standard root module in the definition of our $E_0$ functor. See \cite{unipotent2} and Remark \ref{rem:certainprefund} for more discussion on the limitations of these constructions.

One can similarly prove the following.
\begin{prop}
    Define the algebra $U_q^+(\mathfrak{sp}_4)$ and bilinear form $(*,*)_L$ as the type $C_2$ algebra \textbf{f} and bilinear form of \cite{lusbook}, with Lusztig's $v^{-1}$ equal to our $q$. Let $E_i^*$ be the adjoint operator under the form $(*,*)_L$ on $U_q^+(\mathfrak{sp}_{4})$ for right multiplication by $E_i$. Label the vertices of the Dynkin diagram $1$ and $2$ with $1$ corresponding to the short simple root. Then \[E_0:=[[E_2^*,E_1^*]_{q^2},E_1^*]=E_2^*E_1^*E_1^*-(1+q^2)E_1^*E_2^*E_1^*+q^2E_1^*E_1^*E_2^*\]
    extends the right multiplication representation to give an action of $U_q^+(\hat{\mathfrak{sp}}_4)$.
\end{prop}
 Categorification in all finite types is also work in progress.

\section{Prefundamental representations}\label{sec:prefund}
\subsection{Background}
In this section, fix $n\geq 2$. We will give a more careful study of the $U_q(\hat{\mathfrak{b}})$ representation on $U_q^+(\mathfrak{sl}_{n+1})$. 

We use several definitions from the literature in this section. Many of these are summarized for brevity, with the details given in the references. When we use a formula from the literature, we exchange $q$ with $q^{-1}$ and $K_i$ with $K_i^{-1}$ due to our convention that our $q$ is Lusztig's $v^{-1}$. In this section, we also specialize $q$ to be a fixed nonzero complex number that is not a root of unity to match the conventions of \cite{herjim}. So, our vector spaces are now over $\C$.

For any $2\leq k \leq n$ and $a\in \C^\times$, denote by 
\begin{equation}\label{eqn:e0ka}
    E_{0,k}\coloneqq a[E_k^*,[E_{k-1}^*,\dots [E_2^*,E_1^*]_q\dots ]_q]_q.\end{equation}

When $k=n$ and $a=1$, this is exactly the operator $E_0$ from Section $\ref{sec:general}$. We also denote by $E_{0,1}$ the operator $aE_1^*$. Since the quantum Serre relations are homogeneous in $E_0$, any choice of $a$ gives a well-defined action of $U_q(\hat{\mathfrak{b}})$ on $U_q^+(\mathfrak{sl}_{n+1})$.

Denote by $\Delta^+$ the set of positive roots of $\mathfrak{sl}_{n+1}$. For each $i\in I$, we identify the simple positive root $\alpha_i$ with the element $\alpha_i\in \N[I]$, and extend this identification to the rest of $\Delta^+$.

The representation theory of $U_q(\hat{\mathfrak{b}})$ is often studied via loop weights, or $\ell$-weights, which refine the ordinary notion of weights for a Kac-Moody Lie algebra. An $\ell$-weight is an $I$-tuple $(\psi_i(z))_{i\in I}$ of power series in a variable $z$ for which all $\psi_i(0)\neq 0$. The $z^k$-coefficient of $\psi_i(z)$ is the eigenvalue of an element $\psi_{i,k}^+\in U_q(\hat{\mathfrak{b}})$. These $\psi_{i,k}^+$ are certain generators in Drinfeld's loop presentation of $U_q(\hat{\mathfrak{sl}}_{n+1})$. There is a notion of lowest $\ell$-weights for certain $U_q(\hat{\mathfrak{b}})$ representations which is based on the usual notion of lowest weights for $\mathfrak{g}$ representations. 

\begin{defn}\label{defn:lowestlweight}
    A representation $V$ of $U_q(\hat{\mathfrak{b}})$ is of \emph{lowest $\ell$-weight} $(\psi_i(z))_{i\in I}$ if there is an element $v\in V$ for which
    \[U_q(\hat{\mathfrak{b}})(v)=V, \qquad U_q(\hat{\mathfrak{b}})^-(v)=\C v, \qquad \sum_{k=0}^\infty \psi_{i,k}^+(v)z^k=\psi_i(z)v ,\]
    where $U_q(\hat{\mathfrak{b}})^-$ is defined in terms of Drinfeld's loop generators, see \cite{herjim}.
\end{defn}

 Hernandez and Jimbo identified certain distinguished representations of $U_q(\hat{\mathfrak{b}})$, called prefundamental representations \cite{herjim}. These are not finite-dimensional. However, all simple representations in an appropriate version of category $\mathcal{O}$ for $U_q(\hat{\mathfrak{b}})$ arise as subquotients of tensors of the prefundamental representations and 1-dimensional representations. We will be mainly interested in the dual category, denoted by $\mathcal{O}^*$. The category $\mathcal{O}^*$ contains prefundamental representations denoted by $R_{i,b}^{\pm}$. For $i\in I$ and $b\in \C^\times$, the prefundamental representations $R_{i,b}^\pm$ are the unique simple representations in $\mathcal{O}^*$ with lowest $\ell$-weight $\psi_i(z)=(1-bz)^{\pm1}$ and $\psi_j(z)=1$ for $j\neq i$. The precise statement we will need is the following.
\begin{prop}[\cite{herjim} Propositions 3.4, 3.18]\label{prop:prefund_quot}
    Fix $i\in I$ and $b\in \C^*$. Let $V$ be a representation in $\mathcal{O}^*$ with lowest $\ell$-weight $(\psi_j(z))_{j\in I}$  such that $\psi_i(z)=(1-bz)^{\pm 1}$ and $\psi_j(z)=1$ otherwise. Then $R^{\pm}_{i,b}$ arises as a subquotient of $V$ containing the lowest $\ell$-weight vector.
\end{prop}
It is easy to check that our representation $U_q^+(\mathfrak{sl}_{n+1})$ is in the category $\mathcal{O}^*$ and that, taking $v=1$, the first two conditions of Definition \ref{defn:lowestlweight} are satisfied. This makes it possible to study the relation between $U_q^+(\mathfrak{sl}_{n+1})$ and the prefundamental modules $R_{i,b}^-$. One of the main results in this section is to compute the lowest $\ell$-weight of $U_q^+(\mathfrak{sl}_{n+1})$, which makes it possible to identify a quotient of $U_q^+(\mathfrak{sl}_{n+1})$ with some $R^-_{i,b}$.

The prefundamental representations were originally constructed as limits of Kirillov-Reshetikhin modules, and so the characters (as $U_q^+(\mathfrak{sl}_{n+1})$-modules) were given in terms of limits of characters. More concrete formulas for characters of the prefundamental representations (and generally, for minimal affinizations) were conjectured by Mukhin and Young and later proven by several authors, see \cite{charform} and references therein. 
\begin{defn}
    For $i\in I$, denote by 
    \[\chi_{MY,i}\coloneqq \frac{1}{\prod_{\alpha\in \Delta^+} (1-e^{\alpha})^{[\alpha]_i}}.\]
    Here, we denote by $[\alpha]_i$ the coefficient of $\alpha_i$ in $\alpha=\sum_{j\in I}[\alpha]_j\alpha_j$. We have used the multiplicative notation $e^{\alpha}$ to refer to the formal variable representing the weight space for the weight $\alpha$. 
\end{defn}
It is known that in type $A_n$, if $\chi(V)$ refers to the $U_q^+(\mathfrak{sl}_{n+1})$-character of a representation $V$, that $\chi(R_{i,b}^+)=\chi(R_{i,b}^-)=\chi_{MY,i}$. As a result of our construction, we will produce new proofs of these formulas for $i=n$.

Denote as before $\hat{I}$ the set of vertices of the Dynkin diagram for $\hat{\mathfrak{sl}}_{n+1}$, with $0$ denoting the extending affine vertex.
Lusztig introduced an action of the braid group $\hat{W}$ for $\hat{\mathfrak{sl}}_{n+1}$ on $U_q^+(\hat{\mathfrak{sl}}_{n+1})$ \cite{lusbook}.
\begin{defn}
    For $i\in \hat{I}$, denote by $T_i$ the $\C(q)$-algebra automorphism of $U_q(\hat{\mathfrak{sl}}_{n+1})$ given on generators as 
    \begin{align*}
        T_i(E_i)&= -F_iK_i^{-1}, && T_i(E_j) = \sum_{r+s=-C_{ij}}(-1)^rq^rE_i^{(s)}E_jE_i^{(r)} \text{ for $j\neq i$},\\
        T_i(F_i)&=-K_i^{}E_i, && T_i(F_j) = \sum_{r+s=-C_{ij}}(-1)^rq^{-r}F_i^{(r)}F_jF_i^{(s)} \text{ for $j\neq i$},\\
        T_i(K_j) &= K_jK_i^{-C_{ji}}.
    \end{align*}
\end{defn}
Here, $E_i^{(s)}$ refers to the divided power $E_i^s/[s]_q!$. Note again that we have inverted some exponents compared to \cite{unipotent} since our $q$ is Lusztig's $v^{-1}$.

Within the root system of $\hat{\mathfrak{sl}}_{n+1}$, denote by $\alpha_0$ the simple root associated to the algebra element $E_0$ and denote by $\theta$ the highest root of $\mathfrak{sl}_{n+1}$. Then denote by $\delta$ the imaginary root $\alpha_0+\theta$. Denote by $\tilde{W}\coloneqq \mathcal{T}\ltimes \hat{W}$ the extended affine Weyl group for $\hat{\mathfrak{sl}}_{n+1}$, where $\mathcal{T}$ is a certain subgroup of Dynkin diagram automorphisms for $\hat{\mathfrak{sl}}_{n+1}$, see \cite{beckforms}.

\begin{defn}
    Fix $i\in I$. Pick a reduced expression $\delta-\alpha_i=s_{i_1}s_{i_2}\dots s_{i_{k-1}}(\alpha_{i_k})$ with each $s_{i_j}\in \tilde{W}$  as in Lemma 1.1 of \cite{beckforms}. Then define the $U_q(\hat{\mathfrak{sl}}_{n+1})$ element \[E_{\delta-\alpha_i}\coloneqq T_{i_1}\dots T_{i_{k-1}}(E_{i_k}).\]
    This definition is independent of the choice of reduced presentation.
\end{defn}
For any $k>1$, one can similarly define $E_{k\delta-\alpha_i}$. For brevity, we shall instead use a formula obtained in \cite{unipotent} as a definition.
\begin{defn}{(\cite{unipotent} Lemma 4.3)}\label{lem:Ekp1deltaformula}
    For $k>0$ and $i\in I$, define
    \begin{align*}E_{(k+1)\delta-\alpha_i}\coloneqq \frac{-1}{q+q^{-1}}&(E_{\delta-\alpha_i}E_{i}E_{k\delta-\alpha_i}-q^{2}E_{i}E_{\delta-\alpha_i}E_{k\delta-\alpha_i}\\
    &-E_{k\delta-\alpha_i}E_{\delta-\alpha_i}E_{i}+q^2E_{k\delta-\alpha_i}E_{i}E_{\delta-\alpha_i}).
    \end{align*}
\end{defn}

 The following proposition follows quickly from the definition of our representation and from Lemma 1.5 of \cite{beckforms}.
\begin{prop}\label{prop:psiaction}
    Fix a function $o:I\rightarrow \{\pm 1\}$ such that $o(i)=-o(j)$ if vertices $i$ and $j$ are adjacent. For the action of $U_q(\hat{\mathfrak{b}})$ on $U_q^+(\mathfrak{sl}_{n+1})$ obtained using our choice of $E_0$ and for the element $1\in U_q^+(\mathfrak{sl}_{n+1})$, we have that $\psi_{i,0}^+(1)=K_i(1)=1$, and for $k>1$, 
    \[\psi_{i,k}^+(1)=o(i)^k(q^{-1}-q)E_{k\delta-\alpha_i}(E_{i}).\]
\end{prop}

Definition \ref{lem:Ekp1deltaformula} and Proposition \ref{prop:psiaction} reduce the computation of the lowest $\ell$-weight of $U_q^+(\mathfrak{sl}_{n+1})$ to a computation of a few values of the actions of the $E_{\delta-\alpha_i}$. We need one more formula to make this computation feasible.

The following is the appropriate dual to Lemma 4.7 of \cite{unipotent}, and it can be proven using the same techniques.
\begin{prop}\label{prop:Edeltaminusalphaiformula}
    For any $i\in I$, we have
    \[E_{\delta-\alpha_i}=E_0(E_1E_2\dots E_{i-1})(E_nE_{n-1}\dots E_{i+1})+\sum_{j_1,\dots j_n}C_{j_1,\dots j_n}E_{j_1}\dots E_{j_n},\]
    where the sum is over all sequences of $\hat{I}$-elements $(j_1,\dots,j_n)$ for which $\sum_{k=1}^n\alpha_{j_k}=\delta-\alpha_i$ and such that $j_1\neq 0$. Here, the elements $C_{j_1,\dots j_n}$ are some constants in $\C$.
\end{prop}
\subsection{Submodule of \texorpdfstring{$U_q^+(\mathfrak{sl}_{n+1})$}{Uq+(sl n+1)}}
In this subsection, we show that the operator $E_0$ on $U_q^+(\mathfrak{sl}_{n+1})$ satisfies properties similar to the twisted derivation properties of the $E_i^*$. We use this to produce a proper $U_q(\hat{\mathfrak{b}})$-submodule of $U_q^+(\mathfrak{sl}_{n+1})$ and to show that the corresponding quotient is simple. 

It is known that the operators $E_i^*$ are twisted derivations, meaning that for any $\N[I]$-homogeneous $A,B\in U_q^+(\mathfrak{sl}_{n+1})$ of weights $\alpha_A$ and $\alpha_B$, we have
\begin{equation}\label{eqn:deriv}
    E_i^*(AB)=q^{-(\alpha_i,\alpha_B)}E_i^*(A)B+AE_i^*(B).
\end{equation}
Here, the bilinear form on $\N[I]$ is defined on generators by $(\alpha_i,\alpha_j)=-C_{ij}$ and extended linearly to the rest of $\N[I]$. See, for example, chapter 1 of \cite{lusbook}.

\begin{lem}\label{lem:E0deriv}
    For any $\N[I]$-homogeneous $A,B\in U_q^+(\mathfrak{sl}_{n+1})$ of weights $\alpha_A$ and $\alpha_B$ and any $2\leq k \leq n$, we have
    \[E_{0,k}(AB)=E_k^*(E_{0,k-1}(AB))-q^{1-(\alpha_k,\alpha_B)}E_{0,k-1}(E_k^*(A)B)-qE_{0,k-1}(AE_k^*(B)).\]
\end{lem}
\begin{proof}
    By Equation \ref{eqn:e0ka}, 
    \begin{align*}
    E_{0,k}(AB)&=[E_k^*,E_{0,k-1}]_q(AB)\\
    &=E_k^*(E_{0,k-1}(AB))-qE_{0,k-1}(E_k^*(AB)).
    \end{align*}
    Equation \ref{eqn:deriv} gives the claim.
\end{proof}
\begin{lem}\label{lem:submodule}
    The operator $E_{0,n}$ commutes with left multiplication by $E_i$ for any $i<n$. In particular, if we denote by $M\subset U_q^+(\mathfrak{sl}_{n+1})$ the right ideal generated by $E_1,E_2,\dots E_{n-1}$, then $M$ is a $U_q(\hat{\mathfrak{b}})$-submodule of $U_q^+(\mathfrak{sl}_{n+1})$.
\end{lem}
\begin{proof}
     Fix $\N[I]$-homogeneous $B$  with $B$ of weight $\alpha_B$. If $n=1$, there is nothing to check. If $n=2$, we compute that 
    \[E_{0,2}(E_1B)=q^{-(\alpha_1+\alpha_2,\alpha_B)}E_{0,2}(E_1)B+q^{-(\alpha_2,\alpha_B)}(q^{-1}-q)E_2^*(E_1)E_1^*(B)+E_1E_{0,2}(B).\]
    The first two terms become zero, so we have the claim.
    
    Now, suppose $n\geq 3$. For $i<n$, denote by $M_i$ the right ideal generated by just $E_i$. Lemma \ref{lem:E0deriv} gives that
    \[E_{0,n}(E_iB)=E_n^*(E_{0,n-1}(E_iB))-q^{1-(\alpha_n,\alpha_B)}E_{0,n-1}(E_n^*(E_i)B)-qE_{0,n-1}(E_iE_n^*(B)).\]
    If $i\leq n-2$, then by Equation \ref{eqn:deriv}, this simplifies to 
    \[E_iE_n^*(E_{0,n-1}(B))-0-qE_iE_{0,n-1}(E_n^*(B))=E_iE_{0,n}^*(B)\in M_i.\] If instead $i=n-1$, then again the middle term of our expansion becomes zero. We must therefore show that $E_n^*(E_{0,n-1}(E_{n-1}B))-qE_{0,n-1}(E_{n-1}E_n^*(B))\in M_{n-1}$. Applying Lemma \ref{lem:E0deriv} again, we compute that 
    \begin{align*}
        &E_n^*(E_{0,n-1}(E_{n-1}B))-qE_{0,n-1}(E_{n-1}E_n^*(B))\\
        =&E_n^*(E_{n-1}^*(E_{0,n-2}(E_{n-1}B)))-q^{1-(\alpha_{n-1},\alpha_B)}E_n^*(E_{0,n-2}(E_{n-1}^*(E_{n-1})B))-qE_n^*(E_{0,n-2}(E_{n-1}E_{n-1}^*(B)))\\
        -&qE_{n-1}^*(E_{0,n-2}(E_{n-1}E_n^*(B)))+q^{2-(\alpha_{n-1},\alpha_B-\alpha_n)}E_{0,n-2}(E_{n-1}^*(E_{n-1})E_n^*(B))+q^2E_{0,n-2}(E_{n-1}E_{n-1}^*(E_n^*(B))).
    \end{align*}
    Equation \ref{eqn:deriv} gives that the two rightmost terms are in $M_{n-1}$. By the quantum Serre relations, $E_{0,n-2}$ commutes with $E_n^*$, so the middle two terms cancel. Equation \ref{eqn:deriv} allows us to rewrite the right two terms as 
    \begin{align*}
    &E_n^*(E_{n-1}^*(E_{0,n-2}(E_{n-1}B)))-qE_{n-1}^*(E_{0,n-2}(E_{n-1}E_n^*(B)))\\
    =&E_n^*(E_{n-1}^*(E_{n-1}E_{0,n-2}(B)))-qE_{n-1}^*(E_{n-1}E_{0,n-2}(E_n^*(B)))\\
    =&E_n^*(E_{n-1}E_{n-1}^*(E_{0,n-2}(B)))+q^{-(\alpha_{n-1},\alpha_B-\alpha_1-\dots\alpha_{n-2})}E_n^*(E_{n-1}^*(E_{n-1})E_{0,n-2}(B))\\
    -&qE_{n-1}E_{n-1}^*(E_{0,n-2}(E_n^*(B)))-q^{1-(\alpha_{n-1},\alpha_B-\alpha_1,...-\alpha_{n-2}-\alpha_n)}E_{k-1}^*(E_{k-1})E_{0,k-2}(E_k^*(B)).
    \end{align*}
    By the same logic as before, we see that the left two terms are in $M_{n-1}$ and the right two terms cancel. So indeed $E_{0,n}(E_{n-1}B)\in M_{n-1}$.
    
    So, for any $i\leq n-1$, we have $E_{0}(E_iB)\in M_i\subset M$. By definition, $M$ is generated by these $E_iB$, so $M$ is closed under the $E_0$ action. It is closed under the action of all other $E_i$ by construction, and so the second claim is true.
\end{proof}

We will soon show that $M$ is maximal.

\begin{lem}\label{lem:quotchar}
    $\chi(U_q^+(\mathfrak{sl}_{n+1})/M)=\chi_{MY,n}$.
\end{lem}
\begin{proof}
    Consider the ``down and to the right" convex ordering on the positive roots, i.e. the one given by $\alpha_1 < \alpha_2 + \alpha_1 < \alpha_2 < \alpha_1+\alpha_2+\alpha_3 <\dots <\alpha_n$. For any positive roots $\alpha$ and $\alpha'$, this order satisfies $\alpha<\alpha'$ if $[\alpha]_n < [\alpha']_n$. The PBW theorem for $U_q^+(\mathfrak{sl}_{n+1})$ gives us $\N[I]$-homogeneous elements $X_\alpha$ of weight $\alpha$ for each $\alpha\in \Delta^+$, as well as a basis of $U_q^+(\mathfrak{sl}_{n+1})$ of the form $\prod_{\alpha\in \Delta^+} X_\alpha^{k_\alpha}$, where each $k_\alpha\in \N$, and the product is taken with lower roots on the left. Due to our choice of total ordering, the PBW theorem also shows that a basis for $M$ is given by the set of the elements $\prod_{\alpha\in \Delta^+}X_\alpha^{k_\alpha}$ for which $k_\alpha>0$ for some $\alpha$ with $[\alpha]_n=0$. Therefore, a basis for $U_q^+(\mathfrak{sl}_{n+1})/M$ is given by the image of the elements $\prod_{\alpha\in \Delta^+}X_\alpha^{k_\alpha}$ for which $k_\alpha=0$ if $[\alpha]_n=0$. The claim then follows from the definition of $\chi_{MY,n}$.
    
\end{proof}

We can describe the $U_q(\hat{\mathfrak{b}})$-module structure on $U_q^+(\mathfrak{sl}_{n+1})/M$ very explicitly in terms of certain monomials in the $E_i$.

\begin{defn}
    A sequence of generators $(E_{i_1},E_{i_2},\dots E_{i_k})$ is called \emph{ordered} if each $i_{k}\geq i_{k+1}$. The monomial $E_{i_1}E_{i_2}\dots E_{i_k}$ associated to any ordered sequence can be written uniquely in the form $E_{n}^{k_n}E_{n-1}^{k_{n-1}}\dots E_1^{k_1}$ for some $k_i\geq 0$. We say that the corresponding monomial is \emph{descending} if each $k_i\geq k_{i+1}$.
\end{defn}

\begin{lem}\label{lem:prefundbasis}
    The images of the descending monomials in $U_q^+(\mathfrak{sl}_{n+1})/M$ form a basis.
\end{lem}
\begin{proof}
    We first show that they span the quotient. Denote by $B$ the set of all monomials in the $E_i$ and $B_{d}$ the set of ordered, descending monomials. By definition of $U_q^+(\mathfrak{sl}_{n+1})$, $B$ spans this algebra. We show that $B_{d}\cup (B\cap M)$ spans $U_q^+(\mathfrak{sl}_{n+1})$, which implies that the image of $B_{d}$ spans the quotient. It is enough to show that any monomial is in the span of $B_{d}\cup (B\cap M)$. For a sequence  $(E_{i_1},\dots E_{i_k})$ of generators, we define the \emph{disorder} of this sequence to be the number of pairs $1\leq a < b \leq k$ for which $i_a<i_b$. Each such sequence determines a monomial $E_{i_1}\dots E_{i_k}$ by concatenation. We argue by induction on the disorder of the given sequence for any monomial.
    
    For the base case, assume that the disorder is zero, i.e., that the given sequence for our monomial is ordered. We may therefore write the monomial as $E_{n}^{k_n}E_{n-1}^{k_{n-1}}\dots E_1^{k_1}$. If this monomial is not descending, then let $j$ be the greatest integer for which $k_{j}<k_{j-1}$. Iteratively applying the quantum Serre relations to $E_{j}^{k_{j}}E_{j-1}^{k_{j-1}}$ shows that this element is in $E_{j-1}U_q^+(\mathfrak{sl}_{n+1})$. Applying the quantum Serre relations again shows that $E_n^{k_n}\dots E_{j+1}^{k_{j+1}}E_{j-1}\in E_{j-1}U_q^+(\mathfrak{sl}_{n+1})\subset M$, and therefore the same is true of the entire monomial.

    For the inductive step, fix a sequence $(E_{i_1},\dots E_{i_k})$ of positive disorder and assume that the monomial associated to any sequence of less disorder is within the span of $B_d\cup (B\cap M)$. Let $j$ be the least integer for which $(E_{i_1},\dots E_{i_j})$ is ordered but $i_{j+1}>i_j$. If $i_{j+1}>i_j+1$, then we may swap these two entries of our sequence and obtain a sequence of less disorder, which is in the span of $B_d\cup (B\cap M)$ by induction. So, assume $i_{j+1}=i_j+1$. Then the monomial associated to this sequence may be written as $E_{n}^{k_n}E_{n-1}^{k_{n-1}}\dots E_{i_j+1}^{k_{i_j+1}}E_{i_j}^{k_{i_j}}E_{i_{j}+1}V$ for some monomial $V$. Moreover, by the base case, we may assume that each $k_i\geq k_{i-1}$. By assumption, $k_{i_j}$, and therefore also $k_{i_j+1}$ are positive. The quantum Serre relations show that $E_{i_j+1}^{k_{i_j+1}}E_{i_j}^{k_{i_j}}E_{i_{j+1}}\in \C(q)E_{i_{j+1}}^{k_{i_j+1}+1}E_{i_j}^{k_{i_j}}+E_{i_j}U_q^+(\mathfrak{sl}_{n+1})$. The first term yields a monomial of strictly less disorder, and the second term, after applying quantum Serre relations, yields a linear combination of monomials in $M$. This completes the induction. 

    We have shown that the descending monomials span $U_q^+(\mathfrak{sl}_{n+1})/M$. We now show linear independence. One can quickly verify from Lemma \ref{lem:quotchar} that if $\alpha$ is the weight of a descending monomial, that $\text{dim}((U_q^+(\mathfrak{sl}_{n+1})/M)_\alpha)=1$, and therefore each descending monomial has nonzero image in the quotient. There is at most one descending monomial of any particular $\N[I]$-degree. The images of the descending monomials are therefore linearly independent from each other due to the $\N[I]$-grading.
\end{proof}
\begin{lem}\label{lem:E0vals}
    Let $i_1,i_2,\dots i_n\in I$ be such that $\sum_{k=1}^n \alpha_{i_k}=\sum_{i=1}^n \alpha_{i}$. Fix nonnegative integers $l_1,l_2,\dots l_n$.
    Then we have
    \[E_0(E_{i_n}^{l_n}E_{i_{n-1}}^{l_{n-1}}\dots E_{i_1}^{l_{1}})=\begin{cases}
        \frac{aq^{1-l_n}(\prod_{i=1}^n [l_i]_q)}{1-q^2}E_{i_n}^{l_n-1}E_{i_{n-1}}^{l_{n-1}-1}\dots E_{i_1}^{l_{1}-1} & i_k=k \text{ for all $k$}\\
        0 & \text{otherwise}
    \end{cases}.\]
\end{lem}
\begin{proof}
    If $i_n\neq n$, then by Lemma \ref{lem:submodule}, $E_0$ commutes with multiplication on the left by $E_{i_n}^{l_n}$, so we compute that $E_0(E_{i_n}^{l_n}\dots E_{i_1}^{l_1})=E_{i_n}^{l_n}E_0(E_{i_{n-1}}^{l_{n-1}}\dots E_{i_1}^{l_1})$. Note that $E_0(E_{i_{n-1}}^{l_{n-1}}\dots E_{i_1}^{l_1})$ must be a $U_q^+(\mathfrak{sl}_{n+1})$ element whose $\N[I]$ weight $\alpha$ satisfies $[\alpha]_{i_n}=-1$, so this element therefore must be zero. Suppose instead $i_n=n$. If $n=1$, then we may use the $q$-boson relations to compute that \[aE_1^*(E_1^{l_1})=a(1+q^{-2}+\dots q^{2-2l_1})E_1^{l_1-1}/(1-q^2)=aq^{1-l_1}[l_1]_qE_1^{l_1-1}/(1-q^2),\] as desired. So, assume $n\geq 2$. Lemma \ref{lem:E0deriv} gives that \begin{align*}
        &E_{0,n}(E_n^{l_n}\times E_{i_{n-1}}^{l_{n-1}}\dots E_1^{l_1})\\
        =&E_n^*(E_{0,n-1}(E_{n}^{l_n}\dots E_{i_1}^{l_1}))-q^{1-(\alpha_n,\sum_j^{n-1} l_j\alpha_{i_j})}E_{0,n-1}(E_n^*(E_n^{l_n})E_{i_{n-1}}^{l_{n-1}}\dots E_{i_1}^{l_1})-qE_{0,n-1}(E_n^{l_{n}}E_n^*(E_{i_{n-1}}^{l_{n-1}}\dots E_{i_1}^{l_{1}})).
        \end{align*}
        Since no $i_j=n$ for $j<n$, the last term vanishes. The first term simplifies to
        \begin{align*}E_n^*(E_{0,n-1}(E_n^{l_n}\dots E_{i_1}^{l_1}))&=E_n^*(E_n^{l_n}E_{0,n-1}(E_{i_{n-1}}^{l_{n-1}}\dots E_{i_1}^{l_1}))\\
        &=q^{-(\alpha_n,\sum_j^{n-1} (l_j-1)\alpha_{i_j})}E_n^*(E_n^{l_n})E_{0,n-1}(E_{i_{n-1}}^{l_{n-1}}\dots E_{i_1}^{l_1})\\
        &=q^{l'_{n-1}-1}E_n^*(E_n^{l_n})E_{0,n-1}(E_{i_{n-1}}^{l_{n-1}}\dots E_{i_1}^{l_1}),\end{align*}
        where $l'_{n-1}$ is the exponent of $E_{n-1}$.
        The second term simplifies to 
        \[-q^{1-(\alpha_n,\sum_j^{n-1} l_j\alpha_{i_j})}E_{0,n-1}(E_n^*(E_n^{l_n})E_{i_{n-1}}^{l_{n-1}}\dots E_{i_1}^{l_1})=-q^{1+l'_{n-1}}E_n^*(E_n^{l_n})E_{0,n-1}(E_{i_{n-1}}^{l_{n-1}}\dots E_{i_1}^{l_1}).\]
        
        So, \begin{align*}
            E_{0,n}(E_n^l\dots E_{i_1})&=(q^{l'_{n-1}-1}-q^{1+l'_{n-1}})E_n^*(E_n^{l_n})E_{0,n-1}(E_{i_{n-1}}^{l_{n-1}}\dots E_{i_1}^{l_1})\\
            &=q^{l'_{n-1}-1}(1-q^2)E_n^*(E_n^{l_n})E_{0,n-1}(E_{i_{n-1}}^{l_{n-1}}\dots E_{i_1}^{l_1}).
        \end{align*} The claim is then true by induction.
\end{proof}

\begin{lem}\label{lem:quotsimple}
    $U_q^+(\mathfrak{sl}_{n+1})/M$ is a simple $U_q(\hat{\mathfrak{b}})$-module.
\end{lem}
\begin{proof}
    Fix a nonzero element $v$ of $U_q^+(\mathfrak{sl}_{n+1})/M$. We show that a nonzero scalar is within the submodule generated by $v$. Due to the action of the $K_i$, we may assume $v$ is $\N[I]$-homogeneous, and so by Lemma \ref{lem:prefundbasis}, we may assume $v$ is the descending monomial $E_n^{k_n}\dots E_1^{k_1}$ with each $k_i\geq k_{i-1}$. If each $k_i=0$ then we are done. Otherwise, assume some $k_i\neq 0$ but $k_{i-1}=\dots=k_1=0$. Then \[E_1^{k_i}E_2^{k_i}\dots E_{i-1}^{k_i}(v)=E_n^{k_n}\dots E_i^{k_i}E_{i-1}^{k_i}\dots E_1^{k_i}.\] By Lemma \ref{lem:E0vals}, $E_0^{k_i}(E_n^{k_n}\dots E_i^{k_i}E_{i-1}^{k_i}\dots E_1^{k_i})$ is a nonzero multiple of $E_n^{k_n-k_i}\dots E_{i+1}^{k_{i+1}-k_i}$. This monomial has strictly fewer nonzero exponents than did $v$, so an induction on the number of nonzero exponents gives the claim.
\end{proof}
\subsection{Lowest \texorpdfstring{$\ell$}{l}-weight}
We show that the lowest $\ell$-weight of $U_q^+(\mathfrak{sl}_{n+1})$ agrees with that of $R^-_{n,o(n)aq^{-1}}$, and therefore deduce that $U_q^+(\mathfrak{sl}_{n+1})/M\simeq R_{n,o(n)aq^{-1}}^-$.

\begin{lem}\label{lem:lowestloopweight}
    The $\ell$-weight of $1\in U_q^+(\mathfrak{sl}_{n+1})$ is $\psi_i(z)(1)=1$ if $i\neq n$ and $\psi_n(z)(1)=(1-o(n)aq^{-1}z)^{-1}$, i.e.
    \[\psi_{i,k}^+(1)=\begin{cases}
        (o(n)aq^{-1})^k & i=n\\
        \delta_{k,0} & i\neq n.
    \end{cases}\]
\end{lem}
\begin{proof}
    We induct on $k$. For $k=0$, $\psi_{i,0}^+=k_i$, so the claim is clear. For $k=1$, Proposition \ref{prop:psiaction} gives that $\psi_{i,1}^+(1)=o(i)(q^{-1}-q)E_{\delta-\alpha_i}(E_{i})$. Due to the $\N[I]$-grading, the terms inside the summation in Proposition \ref{prop:Edeltaminusalphaiformula} act by zero on $E_{i}$ since they involve an action of $E_0$ on a $U_q^+(\mathfrak{sl}_{n+1})$ element of weight less than $\theta$. So, by the same proposition, $E_{\delta-\alpha_i}(E_{i})=E_0(E_{i}E_{i+1}\dots E_nE_{i-1}E_{i-2}\dots E_1)$. Lemma \ref{lem:E0vals} gives that this is $\delta_{i,n}a/(1-q^2)$. Therefore, $\psi_{i,1}^+(1)=\delta_{i,n}o(i)aq^{-1}$, as desired.

    For the inductive step, assume that $\psi_{i,k}^+(1)$ is given by the claimed formula for some fixed $k\geq 1$. Proposition \ref{prop:psiaction} again gives that $\psi_{i,k+1}^+(1)=o(i)^{k+1}(q^{-1}-q)E_{(k+1)\delta-\alpha_i}(E_i)$. Unwinding the definition, we have
    \begin{align*}
        \psi_{i,k+1}^+(1)=-\frac{o(i)^{k+1}(q^{-1}-q)}{q+q^{-1}}&(E_{\delta-\alpha_i}E_iE_{k\delta-\alpha_i}-q^2E_iE_{\delta-\alpha_i}E_{k\delta-\alpha_i}\\
        &-E_{k\delta-\alpha_i}E_{\delta-\alpha_i}E_i+q^2E_{k\delta-\alpha_i}E_iE_{\delta-\alpha_i})(E_i).
    \end{align*}
    By the inductive hypothesis and Lemma \ref{lem:E0vals}, we may compute each term separately. For simplicity, denote $C\coloneqq o(n)aq^{-1}$.
    \begin{align*}
        E_{\delta-\alpha_i}E_iE_{k\delta-\alpha_i}(E_i)&=E_{\delta-\alpha_i}(\delta_{i,n}E_iC^k/(o(i)^k(q^{-1}-q)))\\
        &=\delta_{i,n}aC^k/(o(n)^k(q^{-1}-q)(1-q^2)),
    \end{align*}
    \begin{align*}
        -q^2E_iE_{\delta-\alpha_i}E_{k\delta-\alpha_i}(E_i)=0
    \end{align*}
    since $E_{\delta-\alpha_i}E_{k\delta-\alpha_i}(E_i)$ must be homogeneous of degree $-\alpha_i$,
    \begin{align*}
    -E_{k\delta-\alpha_i}E_{\delta-\alpha_i}E_i(E_i)&=-\delta_{i,n}E_{k\delta-\alpha_i}(a(1+q^{-2})E_n/(1-q^2))\\
    &=-\delta_{i,n}a(1+q^{-2})C^k/(o(n)^k(q^{-1}-q)(1-q^2)),
    \end{align*}
    and
    \begin{align*}
        q^2E_{k\delta-\alpha_i}E_iE_{\delta-\alpha_i}(E_i) &=\delta_{i,n}q^2E_{k\delta-\alpha_i}(aE_i/(1-q^2))\\
        &=aq^2C^k/(o(n)^k(q^{-1}-q)(1-q^2)).
    \end{align*}
    Therefore,
    \begin{align*}
        \psi_{i,k+1}^+(1)&=\delta_{i,n}\frac{-ao(n)C^k}{(1-q^2)(q+q^{-1})}(1-1-q^{-2}+q^2)\\
        &=\delta_{i,n}o(n)aq^{-1}C^k=\delta_{i,n}C^{k+1},
    \end{align*}
    as desired. So, the claim holds for all $k\geq 0$.
\end{proof}
\begin{rem}
    Note the lack of a factor of $q-q^{-1}$ in the lowest $\ell$-weight compared to the highest $\ell$-weights obtained in \cite{unipotent}. This is due to the different choice of normalization of the bilinear form on $U_q^+(\mathfrak{sl}_{n+1})$.
\end{rem}
\begin{thm}\label{thm:quotisprefund}
As $U_q(\hat{\mathfrak{b}})$-modules, $U_q^+(\mathfrak{sl}_{n+1})/M\simeq R^-_{n,o(n)aq^{-1}}$
\end{thm}
\begin{proof}
    By Lemma \ref{lem:lowestloopweight}, $U_q^+(\mathfrak{sl}_{n+1})/M$ and $R^-_{n,o(n)aq^{-1}}$ have the same lowest $\ell$-weight. Due to Proposition \ref{prop:prefund_quot}, we have that $R^-_{n,o(n)aq^{-1}}$ arises as a subquotient of $U_q^+(\mathfrak{sl}_{n+1})/M$. But by Lemma \ref{lem:quotsimple}, $U_q^+(\mathfrak{sl}_{n+1})/M$ is simple. So, we must have the isomorphism $U_q^+(\mathfrak{sl}_{n+1})/M\simeq R^-_{n,o(n)aq^{-1}}$.
\end{proof}
\begin{cor}
    $\chi(R_{n,o(n)aq^{-1}}^-)=\chi_{MY,n}$.
\end{cor}
\begin{proof}
    Combine Theorem \ref{thm:quotisprefund} with Lemma \ref{lem:quotchar}.
\end{proof}
\begin{rem}
    Due to the relatively explicit nature of this quotient in terms of the Kac-Moody generators, it is natural to hope for a categorification of the prefundamental representation $R^-_{n,o(n)aq^{-1}}$. This is especially interesting since these representations are generally defined in terms of the actions of Drinfeld's loop generators, but no such action of loop generators is known categorically. Such a categorification is work in progress.
\end{rem}
\begin{ex}
    We combine the results in this section to describe precisely the $U_q(\hat{\mathfrak{b}})$-module structure on the prefundamental $R^-_{n,o(n)aq^{-1}}\simeq U_q^+(\mathfrak{sl}_{n+1})/M$. A basis for this space is given by length-$n$ sequences of nonincreasing nonnegative integers $(k_n,k_{n-1},\dots k_1)$. In the formulas below, a sequence that is not nonincreasing or not nonnegative is taken to be the zero vector.
    \begin{align*}
        K_i(k_n,\dots k_1)&=q^{\sum_{j\in I} k_j\hat{C}_{ij}}(k_n,\dots k_1),\\
        E_1(k_n,\dots k_1)&=(k_n,\dots ,k_1+1),\\
        E_i(k_n,\dots k_{i},k_{i-1},\dots k_1)&=\frac{[k_i+1-k_{i-1}]_q}{[k_i+1]_q}(k_n,\dots k_i+1,k_{i-1},\dots k_1) \text{ for $i>1$},\\
        E_0(k_n,\dots k_1)&= \frac{aq^{1-k_n}(\prod_{i=1}^n [k_i]_q)}{1-q^2}(k_n-1,\dots k_1-1).
    \end{align*}
    
    These formulas should be compared to the examples in Section 5 of \cite{herjim} and Section 5 of \cite{unipotent}.
\end{ex}
\begin{rem}
    We have described the action with respect to certain ordered monomials of $U_q^+(\mathfrak{sl}_{n+1})$. In light of how PBW bases are used in \cite{unipotent}, \cite{bkm}, and our proof of Lemma \ref{lem:quotchar}, it may be interesting to also express this action in terms of a well-chosen PBW-type basis of $U_q^+(\mathfrak{sl}_{n+1})$.
\end{rem}
\begin{rem}\label{rem:certainprefund}
    It is natural to look for related representations for other vertices in type $A_n$ or in types besides $A_n$ similar to the constructions in \cite{unipotent} and \cite{unipotent2}. We have verified that in type $A_3$, taking $E_0=[[E_2^*,E_3^*]_q,E_1^*]_q$ as in Section \ref{sec:d4} gives a $U_q(\hat{\mathfrak{b}})$ action on $U_q^+(\mathfrak{sl}_{4})$ with submodule $(E_1,E_3)U_q^+(\mathfrak{sl}_4)$. Similarly, we have verified that in type $C_2$, taking $E_0=[[E_2^*,E_1^*]_{q^2},E_1^*]$ gives a $U_q(\hat{\mathfrak{b}})$ action on $U_q^+(\mathfrak{sp}_4)$ with submodule $E_1U_q^+(\mathfrak{sp}_4)$. It is straightforward to verify that the corresponding quotients have the expected characters $\chi_{MY,i}$.

    There is likely a limit to which prefundamental representations can be constructed in this way. For certain vertices $i$ outside of type $A_n$, the prefundamental characters $\chi_{MY,i}$ may have factors $1/(1-e^{\alpha})$ of multiplicity more than one, and therefore our proof of the analog of Lemma \ref{lem:quotchar} would fail. This suggests that the prefundamental representation $R_{i,b}^-$ is too complex to be constructed as a quotient of $U_q^+(\mathfrak{g})$ using the techniques of this paper. A related difficulty for these vertices is encountered in \cite{unipotent2}, where the unipotent coordinate ring representation is no longer simple. This suggests that, given $E_0=[E_{i_1}^*,E_{i_2}^*]_{q^{l_1}},E_{i_3}^*]_{q^{l_2}},\dots E_{i_k}^*]_{q^{l_{k-1}}}$, we do not obtain an action of $U_q(\hat{\mathfrak{b}})$ on $U_q^+(\mathfrak{g})$ unless $\alpha_{i_1}$ has multiplicity $1$ in the highest root $\theta$ of $\mathfrak{g}$. This aligns with some difficulties in types $D_4$ and $C_2$ mentioned in Section \ref{sec:d4}. To obtain the prefundamental representation $R_{i,b}^-$, one may want to start instead with a representation of size at least $U_q^+(\mathfrak{g})^{\otimes [\theta]_{i}}$.
\end{rem}

\bibliographystyle{amsalpha}
    \bibliography{ref.bib}
\end{document}